\documentclass[12pt]{article}

\usepackage[T1]{fontenc}
\usepackage[utf8]{inputenc}

\usepackage{amsmath,amssymb,amsthm,mathtools}
\usepackage[mathscr]{euscript}
\usepackage{stmaryrd}

\usepackage[margin=1.2in]{geometry}
\usepackage[protrusion=true,expansion=false]{microtype}

\usepackage{tikz}
\usepackage{tikz-cd}
\usetikzlibrary{
  arrows.meta,
  positioning,
  calc,
  shapes.geometric,
  fit,
  backgrounds,
  decorations.pathreplacing
}

\usepackage{xcolor}
\usepackage[
  colorlinks=true,
  linkcolor=blue,
  citecolor=blue,
  urlcolor=blue,
  pdftitle={Categoricity in Arithmetic Degrees},
  pdfauthor={Jun Le Goh and Chieu-Minh Tran}
]{hyperref}

\usepackage{tocloft}
\usepackage{mdframed}
\usepackage{enumitem}

\setlist{itemsep=2pt, topsep=2pt}
\setlist[enumerate,1]{label=(\arabic*), ref=\arabic*}
\setlist[enumerate,2]{label=(\roman*), ref=\theenumi.(\roman*)}
\setlist[enumerate,3]{label=(\alph*), ref=\theenumii.(\alph*)}
\setlist[itemize]{label=\textbullet}

\newtheorem{theorem}{Theorem}[section]
\newtheorem{lemma}[theorem]{Lemma}
\newtheorem{corollary}[theorem]{Corollary}
\newtheorem{proposition}[theorem]{Proposition}
\newtheorem{fact}[theorem]{Fact}

\theoremstyle{definition}
\newtheorem{definition}[theorem]{Definition}

\newtheorem{remark}[theorem]{Remark}

\newcommand{\ACF}{\mathrm{ACF}}

\newcommand{\DLO}{\mathrm{DLO}}

\newcommand{\N}{\mathbb{N}}
\newcommand{\Q}{\mathbb{Q}}

\newcommand{\degAr}{\deg_{\mathrm{ar}}}

\newcommand{\leAr}{\le_{\mathrm{ar}}}

\newcommand{\acl}{\operatorname{acl}}
\newcommand{\dcl}{\operatorname{dcl}}

\newcommand{\Th}{\operatorname{Th}}
\newcommand{\tp}{\operatorname{tp}}

\newcommand{\Code}{\operatorname{Code}}
\newcommand{\cl}{\operatorname{cl}}
\newcommand{\CB}{\operatorname{CB}}

\newcommand{\Diag}{\mathrm{Diag}}

\newcommand{\CTerm}{\mathrm{CTerm}}
\newcommand{\eq}{\mathrm{eq}}

\newcommand{\LL}{\mathcal{L}}
\newcommand{\KK}{\mathscr{K}}
\newcommand{\PP}{\mathscr{P}}
\newcommand{\BB}{\mathscr{B}}
\newcommand{\CC}{\mathscr{C}}

\newcommand{\MM}{\mathscr{M}}
\newcommand{\NN}{\mathscr{N}}
\newcommand{\AAA}{\mathscr{A}}

\newcommand{\Hull}{\operatorname{Hull}}

\newcommand{\restr}{\upharpoonright}

\newcommand{\pair}{\mathrm{pair}}

\title{\textbf{Categoricity without Power}}
\author{Jun Le Goh \and Chieu-Minh Tran}
\date{National University of Singapore}

\begin{document}
\maketitle

\begin{abstract}
We prove an analogue of Morley's categoricity theorem where cardinality is replaced by the recursion-theoretic notion of arithmetic degree. We say that a complete arithmetically definable
theory \(T\) is \emph{\(D\)-categorical} if any two arithmetically
extendible models of \(T\) of arithmetic degree \(D\), considered over a
common elementary submodel with arithmetical elementary diagram, are isomorphic over
that submodel by an isomorphism which preserves the complexity of sets of degree $D$. Here an
\emph{arithmetically extendible} model means an elementary substructure
of a model whose
elementary diagram is arithmetical. Our main result is:
\begin{quote}
\emph{If \(T\) is \(D_1\)-categorical for some nonzero arithmetic
degree \(D_1\), then \(T\) is \(D_2\)-categorical for every nonzero
arithmetic degree \(D_2\).}
\end{quote}
\noindent
We also show that, assuming \textsf{ZFC}, \(D\)-categoricity for some
nonzero arithmetic degree is equivalent to uncountable categoricity.
\end{abstract}

\section{Introduction}
\label{sec:intro}

\noindent
An algebraically closed field is classified, up to isomorphism, by two
invariants: its characteristic and its transcendence degree over the
prime field. Thus, for each fixed characteristic
\(p\in\{0\}\cup\{\text{primes}\}\), the theory \(\ACF_p\) has exactly
\(\aleph_0\) many countable models up to isomorphism, corresponding to
finite or countably infinite transcendence degree, while for every
uncountable cardinal \(\kappa>\aleph_0\) it has exactly one model of
cardinality \(\kappa\) up to isomorphism. Against this background, \L{}o\'s
formulated his celebrated categoricity conjecture:

\begin{quote}
If a countable first-order theory \(T\) is
\(\kappa\)-categorical for some uncountable cardinal \(\kappa\), then
\(T\) is \(\lambda\)-categorical for every uncountable cardinal
\(\lambda\).
\end{quote}

Morley proved this conjecture in his landmark paper
\emph{Categoricity in Power} \cite{morley1965}, thereby founding modern stability theory.
More important than the transfer statement itself, Morley's result
revealed a deep structural principle: uniqueness in one uncountable
cardinal forces a rigid internal organization of models, much as
algebraically closed fields are governed by characteristic and
transcendence degree. The transfer of categoricity then appears as a
consequence of this structure. This insight was later developed by
Shelah into classification theory.

However, this ``uniqueness forces structure'' phenomenon is not
confined to model theory. A classical example from elsewhere in
mathematics is the isoperimetric inequality:
\begin{quote}
Among all sets \(E\subseteq\mathbb{R}^n\) of fixed perimeter, the
Euclidean ball uniquely maximizes volume.
\end{quote}
Here too, uniqueness is not merely an extremal curiosity. It arises
from, and points back to, a rigid underlying geometry.

The present paper is written with the optimistic viewpoint that these are
not merely analogous results, but manifestations of a common principle.
There is now significant evidence for this. In Euclidean space, the
isoperimetric inequality is closely tied to, and in suitable formulations
essentially equivalent to, the Brunn--Minkowski inequality
\[
\mu(A+B)^{1/n}\geq \mu(A)^{1/n}+\mu(B)^{1/n},
\]
with equality in the special case of \(A=B\) characterizing convexity.
There is by now a substantial literature drawing parallels between
convexity and finite VC-dimension~\cite{Matousek2004BoundedVCFractionalHelly,Kaplan2024DefinablePQTheoremNIP}, the latter being a weakening of
stability-theoretic tameness. More importantly, this is a useful
viewpoint. In the nonabelian setting, Hrushovski~\cite{hrushovski2012}, and later Breuillard,
Green, and Tao~\cite{breuillardgreentao2012} in their work on approximate groups, used ideas closely
related to stability in studying phenomena connected to
Brunn--Minkowski-type behavior. See also the resolution of the
Breuillard--Green measure-doubling conjecture in
\(\mathrm{SO}(3,\mathbb R)\) by Jing, the second author, and Zhang
\cite{jingtranzhang2023}, and the subsequent compact-Lie-group extension
by Machado~\cite{machado2024}.

If this viewpoint is correct, then some version of the same story should
also appear in a computability-theoretic setting, intermediate between
the finite and the transfinite. The first motivation of this paper is to
verify that this is indeed the case. Replacing cardinality by a suitable
arithmetic complexity invariant, the Morley phenomenon still has a genuine
counterpart. On the other hand, this is not automatic: the usual replacement
leads toward computable structure theory, which has a rather different
flavor from Morley's categoricity theorem.

We now set up the framework in which cardinality is replaced by a
complexity notion from recursion theory. Recall that a set
\(X\subseteq\N\) is \emph{arithmetical} if it is definable in the
first-order structure \((\N,+,\times,0,1,<)\),
equivalently, if $X$ is computable relative to some finite Turing jump $0^{(n)}$. More generally, we say $X \subseteq \N$ is \emph{arithmetically reducible} to $Y \subseteq \N$ if $X \leq_T Y^{(n)}$ for some $n \in \N$. Two sets have the same
\emph{arithmetic degree} if they are mutually arithmetically reducible,
and we write \(\degAr(X)\) for the arithmetic degree of \(X\). Arithmetic degree measures complexity up to mutual arithmetical
definability, and will serve as the analogue of cardinality in the
present setting.

We also allow the ambient language and theory themselves to be given
arithmetically. Under a fixed standard G\"odel coding, a countable
language \(\LL\) is \emph{arithmetically definable} if the set of codes
of its symbols is arithmetical, and a complete theory \(T\) in such a
language is \emph{arithmetically definable} if the set of codes of its
axioms is arithmetical. If $M \subseteq \N$, we code $n$-ary relations on $M$ as subsets of $\N$ using fixed tupling functions. We say that such a relation is arithmetical if its code is an arithmetical subset of $\N$. For structures with domain contained in
\(\N\), the relation \(\MM\models T\) is understood in the usual coded
sense recalled in Appendix~\ref{sec:appendix-foundational}.

A first guess would be simply to replace cardinality by arithmetic
degree while keeping ordinary isomorphism as the relevant notion of
sameness. But ordinary isomorphism is too weak for this purpose: it
preserves the first-order structure while saying nothing about
arithmetic complexity. Thus, if arithmetic degree is to play the role
that cardinality plays in the classical setting, one must strengthen
the notion of isomorphism itself.

Accordingly, if \(f:\MM_1\cong\MM_2\) is an isomorphism between structures with domains
\(M_1,M_2\subseteq\N\), and \(D>\mathbf 0\) is an arithmetic degree, the
following is the strengthened notion of isomorphism we use.

\begin{definition}[\(D\)-preserving isomorphism]
\label{def:d-preserving}
An isomorphism \(f:\MM_1\cong\MM_2\) is \emph{\(D\)-preserving} if both
the graph of \(f\) and the graph of \(f^{-1}\) are \(D\)-arithmetical,
and if \(f\) preserves degree \(D\) on parameter-definable sets:
whenever
\(X=\theta(M_1^n,\bar a)\), for a formula \(\theta(\bar x,\bar y)\) and
a finite tuple \(\bar a\in M_1^{<\omega}\), has arithmetic degree \(D\),
the image \(f[X]=\theta(M_2^n,f(\bar a))\) also has arithmetic degree
\(D\), and conversely, whenever a parameter-definable
\(Y\subseteq M_2^n\) has arithmetic degree \(D\), its preimage
\(f^{-1}[Y]\subseteq M_1^n\) also has arithmetic degree \(D\).
\end{definition}

We emphasize that in the above definition, we require that $f[X]$ and $f^{-1}[Y]$ have arithmetic degree exactly \(D\), not merely \(\leq D\).

This stronger notion of isomorphism immediately creates a new problem: if one allows
completely arbitrary presentations, then \(D\)-preservation becomes too
rigid. The issue is not that the first-order isomorphism type changes:
one may present the same algebraically closed field on domains and with
operations whose arithmetic behavior is completely unrelated. Ordinary
isomorphism forgets this presentation complexity. Since algebraically
closed fields should be among the basic examples of the theory,
arbitrary copies are not the right objects to compare. The remedy is to
restrict attention to a better behaved class of presentations.

An \(\LL\)-structure \(\MM\) with domain \(M\subseteq\N\) is called
\emph{arithmetically decidable} if its elementary diagram is an
arithmetical subset of \(\N\); equivalently, the satisfaction relation
\[
\MM\models \varphi(\bar a)
\]
is arithmetical uniformly in the code of the formula \(\varphi\) and
the tuple \(\bar a\in M^{<\omega}\). If \(T\) is a complete theory, we
say that an \(\LL\)-structure \(\MM\models T\) is
\emph{arithmetically extendible} if there exists an arithmetically
decidable model \(\widehat{\MM}\models T\) such that
\[
\MM\preccurlyeq \widehat{\MM}.
\]
Thus we restrict attention to presentations that elementarily embed into
arithmetically decidable models. This rules out the pathological copies
just mentioned (often studied in computable structure theory) while retaining the presentations relevant to the
categoricity phenomenon.

We now state the central definition.

\begin{definition}[$D$-categorical]
\label{def:cat}
Let \(T\) be a complete arithmetically definable theory in an
arithmetically definable language with no finite models, and let
\(D>\mathbf 0\) be an
arithmetic degree. We say that \(T\) is \emph{\(D\)-categorical} if for
every arithmetically decidable model \(\MM_0\models T\) and all
arithmetically extendible models \(\MM_1,\MM_2\models T\) such that
\[
\MM_0\preccurlyeq \MM_1,\MM_2
\qquad\text{and}\qquad
\degAr(M_1)=\degAr(M_2)=D,
\]
there exists a \(D\)-preserving isomorphism \(f:\MM_1\cong\MM_2 \) such that
\[
f\restr_{M_0}=\mathrm{id}_{M_0}.
\]
\end{definition}

A reader expecting a closer analogy with \(\kappa\)-categoricity may
wonder why the notion is formulated over a common arithmetically
decidable elementary submodel. Part of the reason is classical:
uncountable categoricity can also be read over a countable elementary
base; see Remark~\ref{rem:base-version-uncountable-categoricity}. In the
arithmetic setting there is an additional issue, namely that arithmetic
complexity does not behave with the same uniform rigidity as cardinality.
Working over a common arithmetically decidable elementary submodel is the
cleanest way we know to recover the needed rigidity.

We are now ready to state the main result.

\begin{theorem}[Arithmetic Morley theorem; Theorem~\ref{thm:morley}]
\label{thm:morley-intro}
Let \(T\) be a complete arithmetically definable theory in a countable
language with no finite models. If \(T\) is \(D_1\)-categorical for some
nonzero arithmetic degree \(D_1\), then \(T\) is \(D_2\)-categorical
for every nonzero arithmetic degree \(D_2\).
\end{theorem}

This theorem is not merely a transfer statement. It is the arithmetic
avatar of the general principle that uniqueness forces structure:
categoricity in one nonzero arithmetic degree already reflects enough
internal rigidity to determine the theory in every other nonzero
arithmetic degree.

A notable feature of the proof is its finite-complexity character.
Although we use Cantor--Bendixson language for spaces of arithmetic
types, only fixed finite derivative stages and arithmetically presented
constructions occur. We believe that the non-\textsf{ZFC} part of our argument can be
formalized over \(\mathrm{ACA}_0\). The present version does not include that
verification; we plan to address it in a later version of this paper, or
in a sequel.

At this point, the reader may reasonably wonder whether
\(D\)-categoricity is really the correct analogue of Morley
categoricity. The answer is given by the following comparison theorem.

\begin{theorem}[Equivalence with uncountable categoricity]
\label{thm:morley-zfc-intro}
Assume \textsf{ZFC}. Let \(T\) be a complete arithmetically definable
theory in a countable language with no finite models. Then \(T\) is
\(D\)-categorical for some, equivalently for every, nonzero arithmetic
degree \(D\) if and only if \(T\) is \(\kappa\)-categorical for some,
equivalently for every, uncountable cardinal \(\kappa\).
\end{theorem}

The proofs of Theorems~\ref{thm:morley-intro} and
\ref{thm:morley-zfc-intro} follow the same broad skeleton as the
classical proof of Morley's theorem via the Baldwin--Lachlan strategy.
One first develops a suitable stability notion, then rules out
Vaughtian pairs, and finally shows that the resulting structure is
controlled by a strongly minimal set whose induced geometry governs all
relevant models.

At the technical level, however, several new ingredients are required.
We list three of them:

\begin{enumerate}[label=(\arabic*), leftmargin=2em]
\item \textbf{Skolem-hull constructions in place of unions of chains.}
In classical model theory, one often builds models by taking unions of
elementary chains. In the present setting this breaks down, since the
union of a chain of arithmetically extendible models need not itself be
arithmetically extendible. We therefore replace outward constructions by
arguments using Skolem functions. Conceptually, this reverses the usual
direction: instead of building models by successive extensions, we build
them inward by taking definable closures of suitable generating sets.

\item \textbf{The introduction of arithmetical stability.}
In the classical setting, stability is formulated in terms of the size
of full type spaces. In our framework, we instead work with the family of
arithmetical types and show that their complexity is uniformly bounded.
This provides the replacement for the role played by \(\omega\)-stability
in the classical Baldwin--Lachlan analysis.

\item \textbf{A degree-recovery step ensuring \(D\)-preservation.}
Once strongly minimal control is obtained, the classical argument
proceeds by matching bases in the induced pregeometry. In the present
setting, that is not enough: one must also show that the resulting
isomorphism is \(D\)-preserving. For this, it is necessary to prove that
every infinite definable set not already contained in the algebraic
closure of its parameters recovers the full ambient arithmetic degree.
This degree-recovery step is what ensures that the final back-and-forth
isomorphism respects arithmetic complexity, and not merely elementary
structure.
\end{enumerate}

We now briefly describe the organization of the paper.

\begin{itemize}[leftmargin=2em]
\item \textbf{Basic facts and examples.}
Section~\ref{sec:basic-facts-examples-dcat} establishes basic properties
of \(D\)-categoricity and presents the first examples and
counterexamples, including strongly minimal theories and \(\DLO\).

\item \textbf{Independent Skolem sequences.}
Section~\ref{sec:independent-skolem-sequences} develops the machinery of
independent Skolem sequences and the corresponding Skolem-hull
constructions, which replace the usual union-of-chains arguments in our
setting.

\item \textbf{Arithmetical stability.}
Section~\ref{sec:arith-stability} introduces arithmetical stability and
shows that categoricity in one nonzero arithmetic degree implies this
stability notion.

\item \textbf{Vaughtian pairs.}
Section~\ref{sec:vaughtian-pairs} proves that the same categoricity
hypothesis rules out Vaughtian pairs.

\item \textbf{Arithmetic type analysis.}
Section~\ref{sec:morley-analysis-arith-types} develops the finite
Cantor--Bendixson rank analysis for spaces of arithmetic types.

\item \textbf{Prime models and strong minimality.}
Section~\ref{sec:prime-model-strong-minimality} constructs
arithmetically decidable prime models and extracts a strongly minimal
formula controlling all models over a prime base.

\item \textbf{Degree recovery.}
Section~\ref{sec:sm-control-arith-degree} supplies the degree-theoretic
input needed for the final categoricity argument, showing that
sufficiently large infinite definable sets recover the full ambient
arithmetic degree.

\item \textbf{Main theorems.}
Section~\ref{sec:proof-main-theorems} assembles these ingredients into
the arithmetic Baldwin--Lachlan theorem, the arithmetic Morley transfer
theorem, and the comparison with classical uncountable categoricity.

\item \textbf{Concluding remarks and open problems.}
Section~\ref{sec:concluding-remarks} discusses some natural directions
for further work, including refinements of the complexity invariant and
the prospects for reformulating the present theory at the level of
Turing degrees.

\item \textbf{Appendices.}
The appendix collects the coding of structures, compactness, omitting
types, Skolemization, and the restricted imaginary-element machinery used
in the main text.
\end{itemize}

\subsection*{Conventions}
\label{sec:conventions}

We work throughout with a fixed standard G\"odel coding of countable
languages, formulas, proofs, finite tuples, and structures with domain
contained in \(\N\). As in
Definition~\ref{def:coded-structure}, domains are always assumed
\emph{not} to be cofinite in \(\N\); this is a coding convention used
to leave room for the uniform coding of tuples, formulas, and other
auxiliary objects. Thus, when we speak of an arithmetically definable
language or theory, or of a coded structure \(\MM\) with domain
\(M\subseteq\N\), this is always understood relative to that background
coding. The underlying coding of structures and the associated
satisfaction relation are recalled in
Appendix~\ref{sec:appendix-foundational}, especially
Definition~\ref{def:coded-structure},
Definition~\ref{def:arith-decidable-structure-appendix}, and
Definition~\ref{def:sat-coded-structure}; see also
Remark~\ref{rem:tarski-truth}.

If \(X,Y\subseteq\N\), we write \(X\leAr Y\) to mean that \(X\) is
arithmetical in \(Y\), and \(\degAr(X)\) for the corresponding
arithmetic degree. We write \(\mathbf 0\) for the least arithmetic
degree, namely the degree of the arithmetical sets. If \(D\) is an arithmetic degree, we say that a set is \(D\)-arithmetical if it is arithmetical in a representative
of \(D\).

All structures considered in the main body of the paper have domain
contained in \(\N\). When \(\MM\) is such a structure, we write \(M\)
for its underlying domain and identify \(\MM\) with its code whenever
this causes no confusion. If \(\MM\preccurlyeq\widehat{\MM}\) and
\(\widehat{\MM}\) is arithmetically decidable, then truth in \(\MM\) is
computed through the ambient coded structure \(\widehat{\MM}\), with all
quantifiers restricted to the domain predicate for \(M\). Since the
extension is elementary, this agrees with the usual satisfaction
relation of \(\MM\). We will use this convention without further comment
in arguments concerning definability and arithmetic complexity.

We use the standard model-theoretic notation
\(\acl\), \(\dcl\), \(\tp\), and \(\Th\) for algebraic closure,
definable closure, complete type, and complete theory. When the ambient
structure matters, we write \(\acl_{\MM}(A)\) and
\(\tp^{\MM}(\bar a/B)\). Likewise, if \(\MM_0\preccurlyeq\MM\models T\)
and \(\varphi(x,\bar a)\) is strongly minimal over \(M_0\), we use the
induced pregeometry on the corresponding strongly minimal set without
reintroducing the closure operator each time.

Finally, finite tuples and finite parameter sets are treated as
arithmetical objects by default. In particular, adjoining finitely many
parameters does not change arithmetic degree. This convention will be
used repeatedly in degree calculations, for example when passing from a
definable set \(X\) to \(X\oplus \bar a\) for a finite tuple \(\bar a\).

\section{Basic facts and examples}
\label{sec:basic-facts-examples-dcat}

In this section we record the basic properties of arithmetically
extendible models that will be used throughout the paper, and we give
the first examples and counterexamples of \(D\)-categoricity.

The main positive example is provided by strongly minimal theories.
Besides serving as the basic motivating example, the proof of
Proposition~\ref{prop:strongly-minimal-D-cat} already exhibits the
pattern that will later reappear in the arithmetic
Baldwin--Lachlan argument: one matches suitable bases over a fixed
arithmetically decidable base model, extends the resulting partial
correspondences over finite algebraic pieces, and then carries out a
back-and-forth construction while tracking arithmetic complexity. The
main negative example is \(\DLO\).

We begin with a general lemma showing that an arithmetically extendible
model inherits strong arithmetical control from an ambient
arithmetically decidable elementary extension.

\begin{lemma}[Inherited arithmetical bounds]
\label{lem:ambient-arithmetical-inequalities}
Let \(T\) be a complete arithmetically definable theory, and let
\(\MM\models T\) be arithmetically extendible. Then:
\begin{enumerate}[label=(\roman*)]
\item if \(X\subseteq M^n\) is definable in \(\MM\) with parameters from
      an arithmetical subset \(A\subseteq M\), then \(X\leAr M\);
\item if \(A\subseteq M\) is arithmetical, then the substructure of
      \(\MM\) generated by \(A\) has arithmetical domain;
\item if \(M_0\subseteq M\) is arithmetical and \(\bar c\in M^{<\omega}\),
      then \(\acl_{\MM}(M_0\cup \bar c)\) is arithmetical.       Furthermore, after fixing an arithmetically decidable
      elementary extension \(\MM\preccurlyeq\widehat{\MM}\), this
      arithmetical definition is uniform in the code of \(\bar c\).
\end{enumerate}
\end{lemma}

\begin{proof}
Choose an arithmetically decidable model
\(\widehat{\MM}\models T\) such that
\(\MM\preccurlyeq \widehat{\MM}\).

For (i), let \(X=\varphi(\MM^n,\bar a)\) with
\(\bar a\in A^{<\omega}\). Truth in
\(\MM\) is computed by evaluating formulas in \(\widehat{\MM}\) while
restricting all quantifiers to the domain predicate for \(M\). Hence membership in
\(X\) is arithmetical in \(M\), so \(X\leAr M\).

For (ii), an element belongs to the substructure generated by \(A\) iff
it is the value of some \(\LL\)-term applied to a tuple from \(A\).
Since term evaluation in \(\widehat{\MM}\) is arithmetical and \(A\) is arithmetical, the
domain of the generated substructure is arithmetical.

For (iii), first note that the standard
elementarity argument gives
\[
\acl_{\MM}(M_0\cup \bar c)=\acl_{\widehat{\MM}}(M_0\cup \bar c).
\]
One can now arithmetically compute \(\acl_{\widehat{\MM}}(M_0\cup \bar c)\): to decide whether \(a\in \widehat{M}\) lies in it, ask if there is some \(\bar y \in M_0^{<\omega}\) and some formula \(\varphi(x,\bar y,\bar z)\) such that
\(\widehat{\MM}\models \varphi(a,\bar y,\bar c)\) and the solution set
\(\{u\in \widehat{M}:\widehat{\MM}\models \varphi(u,\bar y,\bar c)\}\) is finite. Since \(\widehat{\MM}\) is arithmetically
decidable and \(M_0\) is arithmetical, the above shows that membership in
\(\acl_{\widehat{\MM}}(M_0\cup \bar c)\) is arithmetical, uniformly in
the code of \(\bar c\).
\end{proof}

We next recall the standard geometric facts concerning strongly minimal
theories. We shall use them in the form recorded below; see, for
example, \cite[Chapter~6]{marker2002} or \cite{baldwinlachlan1971}.

\begin{fact}[Standard geometry of strongly minimal theories]
\label{fact:sm-standard-facts}
Let \(T\) be a strongly minimal theory, let
\(\MM_0\preccurlyeq\MM\models T\), and define
\(\cl(A):=\acl(M_0\cup A)\) for \(A\subseteq M\). Then:
\begin{enumerate}[label=(\roman*)]
\item \(\cl\) is a pregeometry on \(M\);
\item every independent subset of \(M\) is contained in a basis;
\item if \(B\subseteq M\) is a basis over \(M_0\), then
      \(M=\acl(M_0\cup B)\);
\item if \(\MM_1,\MM_2\succeq\MM_0\) and \(B_i\subseteq M_i\) are bases
      over \(M_0\) of the same cardinality, then every bijection
      \(B_1\to B_2\) extends to an isomorphism
      \(\MM_1\cong_{M_0}\MM_2\);
\item if \(\bar b\in M^{<\omega}\), then
      \(\acl(M_0\cup\bar b)\preccurlyeq \MM\).
\end{enumerate}
\end{fact}

Part~(v) follows from stationarity of the generic type over
algebraically closed bases. In particular, the finite-dimensional
algebraic pieces used in the strongly minimal proof below are elementary
submodels; this is a specific strongly minimal fact, not a general
consequence of arithmetical extendibility.

The next lemma allows us to carry out compactness arguments in the arithmetic setting. Stronger versions are known \cite{jockuschlewisremmel1991} but we do not need a precise bound.

\begin{lemma}[Effective K\H{o}nig's lemma]
\label{lem:finite-jump-path-finitely-branching-tree}
Let \(\mathcal T\subseteq\N^{<\omega}\) be an infinite finitely
branching tree. Then \(\mathcal T\) has an infinite path which is computable in the double jump \(\mathcal{T}''\) of \(\mathcal{T}\).
\end{lemma}

The first application of
Lemma~\ref{lem:ambient-arithmetical-inequalities} is that, in a strongly
minimal arithmetically extendible model of nonzero arithmetic degree,
every infinite definable set already has the full ambient degree.

\begin{lemma}[Infinite definable sets have full degree]
\label{lem:sm-definable-full-degree}
Let \(T\) be a strongly minimal arithmetically definable theory, let
\(\MM\models T\) be arithmetically extendible, and assume
\(\degAr(M)=D>\mathbf 0\). If \(X\subseteq M^m\) is definable in
\(\MM\) and infinite, then \(\degAr(X)=D\).
\end{lemma}

\begin{proof}
By Lemma~\ref{lem:ambient-arithmetical-inequalities}(i), \(X\leAr M\).
Since \(X\) is infinite, some coordinate projection
\(\pi_i(X)\subseteq M\) is infinite, hence cofinite by strong
minimality. Thus \(M\leAr \pi_i(X)\leAr X\leAr M\).
\end{proof}

We may now verify that strongly minimal theories satisfy
\(D\)-categoricity.

\begin{proposition}[Strongly minimal theories are \(D\)-categorical]
\label{prop:strongly-minimal-D-cat}
Let \(T\) be a strongly minimal arithmetically
definable theory in a countable language. Then \(T\) is
\(D\)-categorical for every arithmetic degree \(D>\mathbf 0\).
\end{proposition}

\begin{proof}
Fix \(D>\mathbf 0\), an arithmetically decidable model
\(\MM_0\models T\), and arithmetically extendible models
\(\MM_1,\MM_2\models T\) such that
\(\MM_0\preccurlyeq \MM_1,\MM_2\) and
\(\degAr(M_1)=\degAr(M_2)=D\).

Each \(\MM_i\) has infinite dimension over \(M_0\):
otherwise \(M_i=\acl(M_0\cup B)\) for a finite \(B\), making \(M_i\)
arithmetical by Lemma~\ref{lem:ambient-arithmetical-inequalities}(iii)
and contradicting \(\degAr(M_i)=D>\mathbf 0\).

Choose a basis \(B_i\subseteq M_i\) over \(M_0\) by the least-choice
recursion
\[
b^i_n=\min\Bigl(M_i\setminus
\acl_{\MM_i}(M_0\cup\{b^i_0,\dots,b^i_{n-1}\})\Bigr).
\]
By Lemma~\ref{lem:ambient-arithmetical-inequalities}(iii), each
finite-dimensional closure is arithmetical uniformly in the previous
tuple, so this recursion has graph arithmetical in \(M_i\); hence
\(B_i\leAr M_i\). That \(B_i\) is indeed a basis follows from finite
character: any \(c\in M_i\setminus\acl(M_0\cup B_i)\) would be outside
\(\acl(M_0\cup\{b^i_0,\dots,b^i_{n-1}\})\) for every \(n\), and once
all basis elements below \(c\) have been selected, the recursion would
choose an element of code at most \(c\), a contradiction. Conversely,
\(M_i=\acl(M_0\cup B_i)\) by
Fact~\ref{fact:sm-standard-facts}(iii), and membership in this closure
is arithmetical in \(B_i\), so \(M_i\leAr B_i\). Therefore
\(\degAr(B_i)=D\).

Since \(D>\mathbf 0\), each \(B_i\) is infinite. Choose increasing
\(D\)-arithmetical enumerations
\(B_1=\{b^1_0,b^1_1,\dots\}\) and \(B_2=\{b^2_0,b^2_1,\dots\}\), and
let \(g:B_1\to B_2\) be the induced \(D\)-arithmetical bijection.

We next extend the basis bijection over finite algebraic pieces.

\medskip\noindent
\textbf{Claim 1.}
Let \(\bar b\) be a finite tuple from \(B_1\), and let
\(\bar b':=g(\bar b)\). Suppose
\[
\tp^{\MM_1}(\bar u,\bar b/M_0)=\tp^{\MM_2}(\bar u',\bar b'/M_0),
\]
where \(\bar u\in \acl(M_0\cup \bar b)\) and
\(\bar u'\in \acl(M_0\cup \bar b')\). If
\(\bar a\in \acl(M_0\cup \bar b)\), then there is
\(\bar a'\in M_2^{|\bar a|}\) such that
\[
\tp^{\MM_1}(\bar u,\bar a,\bar b/M_0)
=
\tp^{\MM_2}(\bar u',\bar a',\bar b'/M_0).
\]

\medskip\noindent
\emph{Proof of Claim~1.}
Choose a formula \(\theta(\bar x,\bar z,\bar y,\bar m)\), with
\(\bar m\in M_0^{<\omega}\), witnessing that
\(\bar a\in\acl(M_0\cup\bar b)\), say with exactly \(N\)
realizations in \(\MM_1\). The type hypothesis transfers
\(\exists^{=N}\bar x\,\theta(\bar x,\bar u,\bar b,\bar m)\) to
\(\MM_2\). Let \(F\) be the set of realizations of
\(\theta(\bar x,\bar u',\bar b',\bar m)\) in \(\MM_2\); it has size
\(N\). Some \(\bar a'\in F\) must realize the same type over
\((\bar u',\bar b',M_0)\) that \(\bar a\) realizes over
\((\bar u,\bar b,M_0)\): otherwise, for each \(\bar c\in F\) choose a
formula separating the types, and the conjunction of \(\theta\) with all
these separating formulas is realized over \((\bar u,\bar b)\) in
\(\MM_1\), hence over \((\bar u',\bar b')\) in \(\MM_2\), but every
realization lies in \(F\) and fails a conjunct. Contradiction.
\hfill\(\dashv\)

\medskip

By symmetry, the analogous statement holds with the roles of
\(\MM_1\) and \(\MM_2\) reversed.

We now define the finite partial maps used in the back-and-forth. A
finite partial map \(h:A\to M_2\) is called \emph{good} if:
\begin{enumerate}[label=(\roman*)]
\item \(A\subseteq M_1\) is finite;
\item \(h\) fixes \(M_0\cap A\) pointwise;
\item there is a finite tuple \(\bar b\) from \(B_1\) such that
      \(A\subseteq \acl(M_0\cup \bar b)\),
      \(h[A]\subseteq \acl(M_0\cup g(\bar b))\), and
      \[
      \tp^{\MM_1}(A,\bar b/M_0)
      =
      \tp^{\MM_2}(h[A],g(\bar b)/M_0).
      \]
\end{enumerate}

Claim~1 shows that every finite tuple from \(M_1\) is contained in the
domain of some good map (apply Claim~1 with \(\bar u=\emptyset\)), and
by symmetry every finite tuple from \(M_2\) is contained in the range of
some good map.

\medskip\noindent
\textbf{Claim 2.}
Every good map extends to a good map whose domain contains any
prescribed element of \(M_1\), and also to a good map whose range
contains any prescribed element of \(M_2\).

\medskip\noindent
\emph{Proof of Claim~2.}
Let \(h:A\to M_2\) be good, witnessed by \(\bar b\subseteq B_1\), and
let \(a\in M_1\). Choose a finite tuple \(\bar d\) from \(B_1\)
containing \(\bar b\) such that
\(A\cup\{a\}\subseteq \acl(M_0\cup \bar d)\). We claim that \(h\) is
also good as witnessed by \(\bar d\). The joint type
\(\tp^{\MM_1}(A,\bar d/M_0)\) is determined by two pieces of data: the
type \(\tp^{\MM_1}(A,\bar b/M_0)\), and the type of the new basis
elements \(\bar d\setminus \bar b\) over
\(\acl(M_0\cup \bar b)\). The closures \(\acl(M_0\cup\bar b)\) and
\(\acl(M_0\cup g(\bar b))\) are algebraically closed elementary bases by
Fact~\ref{fact:sm-standard-facts}(v). Since the new elements are
independent over these bases, the second piece is the stationary generic
type, which is the same on both sides. The first piece agrees by the
goodness of \(h\), so \(h\) is good as witnessed by \(\bar d\).
Applying Claim~1 then gives \(a'\in M_2\) extending \(h\) to a good map
containing \(a\) in the domain. The range extension is symmetric.
\hfill\(\dashv\)

\medskip

To extract an isomorphism from the back-and-forth system, we normalize
the choices into a finitely branching tree and apply K\H{o}nig's lemma.
Take increasing \(D\)-arithmetical enumerations
\(M_1\setminus M_0=\{a_0,a_1,\dots\}\) and
\(M_2\setminus M_0=\{c_0,c_1,\dots\}\). A node of the tree
\(\mathcal T\) at level \(s\) is a triple \((h_s,A_s,\bar b_s)\)
encoding a witnessed good map. The root is the empty map with empty
witness. A child at level \(s+1\) is built by first choosing the finite
tuple \(\bar b^*\supseteq\bar b\) from \(B_1\) with least tuple-code
such that \(h\) is good as witnessed by \(\bar b^*\) and
\[
a_s\in\acl(M_0\cup\bar b^*),
\qquad
c_s\in\acl(M_0\cup g(\bar b^*)),
\]
and then extending \(h\) minimally: add an admissible image of \(a_s\),
and if \(c_s\) is not yet in the range, add one admissible preimage of
\(c_s\). By Claim~2, \(\mathcal T\) is infinite.

\medskip\noindent
\textbf{Claim 3.} \(\mathcal T\) is finitely branching.

\medskip\noindent
\emph{Proof of Claim~3.}
The witness \(\bar b^*\) is uniquely determined as the least tuple-code
satisfying the displayed conditions; its goodness follows from the
stationarity argument in Claim~2. Once \(\bar b^*\) is fixed, the
possible images of \(a_s\) lie among finitely many algebraic conjugates
over \(M_0\cup g(\bar b^*)\), and any needed preimage of \(c_s\) lies
among finitely many conjugates over \(M_0\cup\bar b^*\).
\hfill\(\dashv\)

\medskip\noindent
\textbf{Claim 4.} \(\mathcal T\) is arithmetical in \(D\), and hence
has a \(D\)-arithmetical infinite path.

\medskip\noindent
\emph{Proof of Claim~4.}
Membership in \(\mathcal T\) at level \(s\) involves checking that
\(A_s\subseteq M_1\) contains \(\{a_0,\dots,a_{s-1}\}\), that
\(\bar b_s\) is a tuple from \(B_1\), that \(A_s\) and \(h_s[A_s]\) lie
in the appropriate algebraic closures, that the types match, that
\(\{c_0,\dots,c_{s-1}\}\subseteq h_s[A_s]\), and that the transition is
normalized. Each condition involves \(D\)-arithmetical data: the bases
\(B_1,B_2\) and the map \(g\) are \(D\)-arithmetical; finite algebraic
closures are arithmetical by
Lemma~\ref{lem:ambient-arithmetical-inequalities}(iii); and type
comparisons reduce to formula checks in arithmetically decidable ambient
extensions. The least-witness condition is a bounded search. Thus
membership is computable in \(D^{(r)}\) for some finite \(r\). By
Lemma~\ref{lem:finite-jump-path-finitely-branching-tree}, the tree has
a path computable in \(D^{(r+2)}\), hence \(D\)-arithmetical.
\hfill\(\dashv\)

\medskip

Let \(f:=\operatorname{id}_{M_0}\cup\bigcup_s h_s\), where the union is
taken along such an infinite path. Then
\(f:M_1\to M_2\) is a total \(D\)-arithmetical bijection fixing \(M_0\)
pointwise, and every finite restriction is elementary over \(M_0\), so
\(f:\MM_1\cong_{M_0}\MM_2\). Its inverse is \(D\)-arithmetical as well.

It remains to verify \(D\)-preservation. Let \(X\subseteq M_1^m\) be
parameter-definable, say \(X=\psi(M_1^m,\bar a)\), with
\(\degAr(X)=D\). Then \(X\) is infinite, since \(D>\mathbf 0\). By
elementarity, \(f[X]=\psi(M_2^m,f(\bar a))\) is also infinite and
parameter-definable in \(\MM_2\), so
\(\degAr(f[X])=D\) by Lemma~\ref{lem:sm-definable-full-degree}.
The reverse direction is symmetric, so \(f\) is \(D\)-preserving.
\end{proof}

The proof just given is worth keeping in mind: later, the general
categoricity argument will follow the same pattern, with the classical
strongly minimal basis geometry replaced by its arithmetic
Baldwin--Lachlan analogue.

\begin{corollary}
\label{cor:acf}
For each \(p\in\{0\}\cup\{\text{primes}\}\) and every arithmetic degree
\(D>\mathbf 0\), the theory \(\ACF_p\), in the language of rings, is
\(D\)-categorical.
\end{corollary}

\begin{proof}
\(\ACF_p\) is strongly minimal by quantifier elimination. Apply
Proposition~\ref{prop:strongly-minimal-D-cat}.
\end{proof}

We now turn to the basic counterexample.

\begin{proposition}[\(\DLO\) is not \(D\)-categorical]
\label{prop:dlo}
Let \(T=\DLO\), viewed in the language \(\{<\}\). Then \(T\) is not
\(D\)-categorical for any arithmetic degree \(D>\mathbf 0\).
\end{proposition}

\begin{proof}
Fix \(D>\mathbf 0\), choose \(A\subseteq\N\) with \(\degAr(A)=D\), and
let \(\rho:\N\to\Q\) be a computable bijection. Set
\(\MM:=(\N,<_\rho)\models T\), where \(x<_\rho y\) iff
\(\rho(x)<\rho(y)\).

Let
\[
M_0:=\{x\in\N:\rho(x)<0 \text{ or } 1<\rho(x)<2 \text{ or } \rho(x)>3\},
\]
and let \(\MM_0:=\MM\!\restr_{M_0}\). Since \(\DLO\) has quantifier
elimination and the induced order on \(M_0\) is computable, \(\MM_0\) is
an arithmetically decidable model of \(T\). The induced order is dense
without endpoints: the three displayed intervals are individually dense,
and points in distinct intervals are separated by points in the leftmost
of the two.

We build two arithmetically extendible models of degree \(D\) over
\(\MM_0\) that realize different types. Choose a computable dense set
\(P^+\subseteq\{x:0<\rho(x)<1\}\) and a computable injection
\(\alpha:\N\to\N\) with image in
\(\{x:0<\rho(x)<1\}\setminus P^+\). Let
\(C_A:=\alpha[A]\) and \(M_1:=M_0\cup P^+\cup C_A\). Since
\(P^+\cup C_A\) is dense in \((0,1)\), the model
\(\MM_1:=\MM\!\restr_{M_1}\) satisfies \(T\). Quantifier elimination
for \(\DLO\) makes every order-embedding elementary, so
\(\MM_0\preccurlyeq \MM_1\preccurlyeq \MM\). Since \(\MM\) is
computable, it is arithmetically decidable, and hence \(\MM_1\) is
arithmetically extendible. Moreover \(\degAr(M_1)=D\): \(M_1\) is
arithmetical in \(A\) because \(M_0\) and \(P^+\) are computable, and
conversely \(n\in A\) iff
\(\alpha(n)\in M_1\setminus(M_0\cup P^+)\).

Construct \(\MM_2\) symmetrically, placing the coded copy of \(A\) in
the interval \((2,3)\) instead of \((0,1)\): choose a computable dense
\(P^\ast\subseteq\{x:2<\rho(x)<3\}\), a computable injection \(\beta\)
into the complement, set \(M_2:=M_0\cup P^\ast\cup\beta[A]\), and verify
\(\MM_0\preccurlyeq\MM_2\preccurlyeq\MM\) and \(\degAr(M_2)=D\) by the
same argument.

Now let \(p(x)\) be the complete \(1\)-type over \(M_0\) determined by
the cut at the gap \((0,1)\). The model \(\MM_1\) realizes \(p\), since
it has elements in \((0,1)\), while \(\MM_2\) omits \(p\), since it has
none. Any isomorphism \(\MM_1\cong\MM_2\) fixing \(M_0\) pointwise
would preserve realization of \(p\), so no such isomorphism exists.
\end{proof}

\section{Independent Skolem sequences}
\label{sec:independent-skolem-sequences}

In this section we introduce the Skolem-hull constructions that replace
the usual union-of-chains arguments in our setting. We first build
Skolem-independent sequences and use them to realize arbitrary nonzero
arithmetic degrees, both absolutely and over a fixed arithmetically
decidable base. This isolates the basic coding mechanism in its simplest
form: one chooses a Skolem-independent sequence and recovers a set
\(S\subseteq \N\) from the Skolem hull generated by the corresponding
subsequence. We then refine the construction to indiscernible
Skolem-independent sequences over a fixed arithmetically decidable base.
This stronger refinement will be used later in the proof of arithmetical
stability.

Throughout this section, \(T\) is a complete arithmetically definable
theory in a countable arithmetically definable language and has no
finite models. The existence of Skolem expansions and the compactness
theorem used below are established in the appendix; see
Theorem~\ref{thm:sk} and Theorem~\ref{thm:compact}. We will use the following effectivization of Ramsey's theorem. Stronger versions are known (e.g.\ \cite{jockusch1972}) but we do not need a precise bound.

\begin{fact}[Effective Ramsey theorem]
\label{fact:effective-ramsey}
Fix \(k,\ell<\omega\). If \(r<\omega\) and
\(d:[\N]^k\to\ell\) is an $\ell$-coloring of $k$-tuples of $\N$, then \(d\) has an
infinite homogeneous set computable in $d^{(3)}$.
\end{fact}

\begin{lemma}[Finite Skolem avoidance]
\label{lem:finite-skolem-avoidance}
Let \(\MM\) be an infinite \(\LL\)-structure, let
\(X\subseteq M\) be infinite, and fix \(N<\omega\). Let
\(\mathcal T\) be a finite set of pairs \((i,t)\), where
\(i<N\) and \(t(x_0,\dots,\widehat{x_i},\dots,x_{N-1})\) is an
\(\LL\)-term not involving the variable \(x_i\). Then there are
\(a_0,\dots,a_{N-1}\in X\) such that, for every
\((i,t)\in\mathcal T\),
\(a_i\neq t^\MM(a_0,\dots,\widehat{a_i},\dots,a_{N-1})\).
\end{lemma}

\begin{proof}
Choose a finite set \(F\subseteq X\) with \(|F|>|\mathcal T|\).
For each pair \((i,t)\in\mathcal T\), the set of tuples
\(\bar a\in F^N\) with
\(a_i=t^\MM(a_0,\dots,\widehat{a_i},\dots,a_{N-1})\)
has size at most \(|F|^{N-1}\), because once all coordinates except
the \(i\)-th are fixed, there is at most one possible value of
\(a_i\). Hence the union of all bad sets has size at most
\(|\mathcal T|\cdot |F|^{N-1}<|F|^N\).
Some tuple in \(F^N\) avoids every bad set, as required.
\end{proof}

The lemma does not by itself require the coordinates \(a_i\) to be
distinct. In the applications where distinctness is needed, the finite
set \(\mathcal T\) includes the projection terms \(x_j\) for \(j\neq i\),
so the avoided inequalities include \(a_i\neq a_j\).

\begin{proposition}[Skolem-independent sequences]
\label{prop:skolem-independent-sequence}
Let \(\NN\models T^{\mathrm{Sk}}\) be an arithmetically decidable model.
Then there exist an arithmetically decidable elementary extension
\(\NN\preccurlyeq\NN^\ast\) and a sequence \((a_n)_{n\in\N}\) in
\(N^\ast\) such that:
\begin{enumerate}[label=(\roman*)]
\item for every \(n\in\N\),
      \(a_n\notin \Hull^{\NN^\ast}(\{a_i:i\neq n\})\);
\item the map \(n\mapsto a_n\) is arithmetical.
\end{enumerate}

Moreover, if \(A\subseteq N\) is arithmetical and
\(X=\psi(N,\bar a)\) is infinite for some \(\bar a\in A^{<\omega}\),
then, working in the language \(\LL^{\mathrm{Sk}}(A)\) naming every
element of \(A\), the extension and sequence may be chosen so that
\((a_i)_{i\in\N}\subseteq \psi(N^\ast_A,\bar a)\) and, for every \(i\),
\[
a_i\notin \Hull^{\NN_A^\ast}(A\cup\{a_j:j\neq i\}).
\]
\end{proposition}

\begin{proof}
Add new constants \(c_n\) for \(n\in\N\) to the language
\(\LL^{\mathrm{Sk}}\), and let \(\Sigma\) consist of the elementary
diagram of \(\NN\) together with the sentences
\(c_n\neq t(c_{i_1},\dots,c_{i_k})\)
for every \(\LL^{\mathrm{Sk}}\)-term
\(t(v_{i_1},\dots,v_{i_k})\) and every
\(n\notin\{i_1,\dots,i_k\}\), and also the inequalities
\(c_i\neq \bar n\) for all \(i,n\in\N\), where \(\bar n\) denotes the
constant naming \(n\) in the elementary diagram of \(\NN\). Since
\(\NN\) is arithmetically decidable, its elementary diagram is
arithmetical. The rest of the scheme is arithmetical because the Skolem
language, term codes, indices \(i_1,\ldots,i_k,n\), and the side condition
\(n\notin\{i_1,\ldots,i_k\}\) are all primitively coded. Hence
\(\Sigma\) is arithmetically definable.

For finite satisfiability, a finite fragment \(\Sigma_0\) mentions
only finitely many constants \(c_0,\dots,c_{N-1}\), finitely many terms,
and finitely many old named elements to avoid. Since \(\NN\) is
infinite, removing these finitely many forbidden elements leaves an
infinite set to which Lemma~\ref{lem:finite-skolem-avoidance} applies.

By Theorem~\ref{thm:compact}, \(\Sigma\) has an arithmetically
decidable model \(\widetilde{\NN}\).

The added inequalities \(c_i\neq \bar n\) place each
\(c_i^{\widetilde{\NN}}\) outside the distinguished copy of \(N\), and
the projection-term inequalities make them pairwise distinct. Thus
\(\widetilde N\) has infinitely many elements outside the distinguished
copy of \(N\). We now recode the domain. Choose an arithmetical set
\(U\subseteq\N\) with \(N\subseteq U\), \(U\setminus N\) infinite, and
\(\N\setminus U\) infinite. Fix an arithmetical bijection
\(f:\widetilde N\to U\) sending each distinguished old element
\(\bar n^{\widetilde{\NN}}\) to \(n\), obtained by matching increasing
enumerations of the non-old part of \(\widetilde N\) and the reserve set
\(U\setminus N\) while fixing the old named copy. This is the standing
infinite-complement recoding convention from
Appendix~\ref{subsec:coded-structures-truth}. Transporting
\(\widetilde{\NN}\) along \(f\) gives an arithmetically decidable model
\(\NN^\ast\) with domain \(U\subseteq\N\) and
\(\NN\preccurlyeq \NN^\ast\). Setting
\(a_n:=f(c_n^{\widetilde{\NN}})\), the defining scheme gives~(i).

For~(ii), the graph of \(n\mapsto a_n\) is arithmetical because
\(a_n=m\) iff \(f^{-1}(m)\) is the interpretation of \(c_n\) in
\(\widetilde{\NN}\), and both the elementary diagram of
\(\widetilde{\NN}\) and the bijection \(f\) are arithmetical.

For the moreover clause, repeat the argument in the language
\(\LL^{\mathrm{Sk}}(A)\). The named expansion \(\NN_A\) is
arithmetically decidable, so the same scheme \(\Sigma\)---now with
\(\LL^{\mathrm{Sk}}(A)\)-terms and with the additional requirements
\(\psi(c_i,\bar a)\)---is arithmetically definable. Finite
satisfiability uses Lemma~\ref{lem:finite-skolem-avoidance} inside the
infinite set \(X=\psi(N,\bar a)\). By
Lemma~\ref{lem:base-preserving-coded-compactness}, the resulting model
preserves the actual old domain \(N\), and the interpretations of the
\(c_i\) give the required sequence.
\end{proof}

\begin{lemma}[Degree of a Skolem hull over a named base]
\label{lem:skolem-hull-degree-over-base}
Let \(\NN^{\mathrm{Sk}}\) be arithmetically decidable in a Skolem
language naming an arithmetical base \(A\). Suppose
\((a_n)_{n\in\N}\) is an arithmetical sequence and
\[
a_n\notin\Hull(A\cup\{a_i:i\neq n\})
\]
for every \(n\). For \(S\subseteq\N\), put
\[
H_S:=\Hull(A\cup\{a_n:n\in S\}).
\]
Then \(H_S\leAr S\), and
\[
a_n\in H_S\quad\Longleftrightarrow\quad n\in S.
\]
Consequently, if \(S\) has nonzero arithmetic degree \(D\), then
\(\degAr(H_S)=D\).
\end{lemma}

\begin{proof}
Membership in \(H_S\) is arithmetical in \(S\): an element belongs to
the hull exactly when it is the value, in the fixed Skolem expansion, of
some term applied to finitely many parameters from \(A\) and finitely
many \(a_i\)'s with \(i\in S\). The named base \(A\) is arithmetical and
therefore contributes no degree. Conversely, if \(a_n\in H_S\) and
\(n\notin S\), then \(a_n\) is in the hull of
\(A\cup\{a_i:i\neq n\}\), contradicting the assumed independence. Thus
\(a_n\in H_S\) iff \(n\in S\). Since \(n\mapsto a_n\) is arithmetical,
this gives \(S\leAr H_S\).
\end{proof}

\begin{corollary}[Skolem-independence of subsequences]
\label{cor:sequence-independent}
In the setting of Proposition~\ref{prop:skolem-independent-sequence},
for every \(S\subseteq\N\),
\[
\Hull^{\NN^\ast}(\{a_n:n\in S\})\cap\{a_n:n\in\N\}=\{a_n:n\in S\}.
\]
\end{corollary}

\begin{proof}
The inclusion \(\supseteq\) is immediate. For the reverse inclusion,
suppose \(a_n\in \Hull^{\NN^\ast}(\{a_i:i\in S\})\). Then
\(a_n=t^{\NN^\ast}(a_{i_1},\dots,a_{i_k})\) for some term
\(t(v_{i_1},\dots,v_{i_k})\) with \(i_1,\dots,i_k\in S\). If
\(n\notin\{i_1,\dots,i_k\}\), this contradicts
Proposition~\ref{prop:skolem-independent-sequence}. Hence
\(n\in\{i_1,\dots,i_k\}\subseteq S\).
\end{proof}

\begin{corollary}[Realizing arbitrary nonzero arithmetic degrees]
\label{cor:realizing-arbitrary-degree}
\leavevmode
\begin{enumerate}[label=(\roman*)]
\item For every nonzero arithmetic degree \(D\), there exists an
      arithmetically extendible model \(\MM\models T\) whose domain has
      arithmetic degree \(D\).

\item Let \(\MM_0\models T\) be arithmetically decidable, and work in
      the language \(\LL(M_0)\) obtained by naming every element of
      \(M_0\). For every nonzero arithmetic degree \(D\), there exists an
      arithmetically extendible \(\LL(M_0)\)-structure
      \(\MM\models \Th_{\LL(M_0)}(\MM_0)\) such that
      \(\MM_0\preccurlyeq \MM\) and \(\degAr(M)=D\).
\end{enumerate}
\end{corollary}

\begin{proof}
For~(i), start with any arithmetically decidable model
\(\NN\models T^{\mathrm{Sk}}\), and apply
Proposition~\ref{prop:skolem-independent-sequence} to obtain an
arithmetically decidable elementary extension
\(\NN\preccurlyeq\NN^\ast\) together with a Skolem-independent sequence
\((a_n)_{n\in\N}\) in \(N^\ast\).

Fix \(S\subseteq\N\) with \(\degAr(S)=D\), and let
\[
\MM_S:=\Hull^{\NN^\ast}(\{a_n:n\in S\}).
\]
By Theorem~\ref{thm:sk}(3),
\(\MM_S\!\restr_\LL\preccurlyeq\NN^\ast\!\restr_\LL\),
so \(\MM_S\!\restr_\LL\) is arithmetically extendible.
Lemma~\ref{lem:skolem-hull-degree-over-base} (with empty base) gives
\(\degAr(M_S)=D\).

For~(ii), pass to an arithmetically decidable Skolem expansion of
\(\MM_0\) in the language \(\LL(M_0)\). Apply
Proposition~\ref{prop:skolem-independent-sequence}, moreover clause, with
\(A=M_0\) and \(\psi(x)\equiv(x=x)\) to obtain an arithmetically
decidable elementary extension \(\NN^\ast\) and a Skolem-independent
sequence \((a_n)\) over \(M_0\). Fix \(S\subseteq\N\) with
\(\degAr(S)=D\), and form
\(\MM:=\Hull^{\NN^\ast}(M_0\cup\{a_n:n\in S\})\!\restr_{\LL(M_0)}\).
By Theorem~\ref{thm:sk}(3), \(\MM_0\preccurlyeq \MM\), and
Lemma~\ref{lem:skolem-hull-degree-over-base} with base \(A=M_0\) gives
\(\degAr(M)=D\).
\end{proof}

The preceding corollary isolates the basic degree-coding mechanism,
first in absolute form and then relative to a fixed arithmetically
decidable base. For the later analysis of arithmetical stability,
however, this is not yet enough. There we will need a coding sequence
which is not only Skolem-independent over the base, but also
indiscernible over it. We now refine the construction accordingly.

\begin{lemma}[Finite-fragment arithmetical indiscernibles]
\label{lem:finite-fragment-ind}
Let \(\MM\) be an arithmetically decidable \(\LL\)-structure, let
\(B\subseteq M\) be arithmetical, and let \((a_n)_{n\in\N}\) be a
sequence in \(M\) such that \(n\mapsto a_n\) is arithmetical. If
\(\Delta(\bar x)\) is a finite set of \(\LL(B)\)-formulas in
\(\bar x=(x_0,\dots,x_{k-1})\) (formulas with fewer displayed variables
are regarded as formulas in this same \(k\)-tuple), then there is an
infinite arithmetical
set \(H=\{n_0<n_1<\cdots\}\subseteq\N\) such that
\((a_{n_i})_{i\in\N}\) is \(\Delta\)-indiscernible over \(B\), and the
map \(i\mapsto a_{n_i}\) is arithmetical.
\end{lemma}

\begin{proof}
Colour each increasing \(k\)-tuple \(i_0<\cdots<i_{k-1}\) by the
\(\Delta\)-truth pattern of \((a_{i_0},\dots,a_{i_{k-1}})\) over \(B\).
Since \(\MM\) is arithmetically decidable, \(B\) is arithmetical, and
\(\Delta\) is finite, this gives an arithmetical finite colouring of
\([\N]^k\). By Fact~\ref{fact:effective-ramsey}, there is an infinite
arithmetical homogeneous set \(H\subseteq\N\). Listing \(H\) increasingly as
\(\{n_0<n_1<\cdots\}\), homogeneity says exactly that
\((a_{n_i})_{i\in\N}\) is \(\Delta\)-indiscernible over \(B\). Since
the increasing enumeration of an arithmetical infinite set is again
arithmetical, the map \(i\mapsto a_{n_i}\) is arithmetical.
\end{proof}

\begin{proposition}[Indiscernible Skolem-independent sequence]
\label{prop:ind-skolem-sequence}
Let \(\MM_0\models T\) be arithmetically decidable. Work in the language
\(\LL^{\mathrm{Sk}}(M_0)\) obtained by naming every element of
\(M_0\). Then there exist an arithmetically decidable elementary
extension
\[
\MM_0^{\mathrm{Sk}}\preccurlyeq \NN^{\mathrm{Sk}}
\]
and a sequence \((a_i)_{i\in\N}\) in \(N\) such that the map
\(i\mapsto a_i\) is arithmetical, the sequence is indiscernible over
\(M_0\), and for every \(n\) one has
\[
a_n\notin \Hull^{\NN^{\mathrm{Sk}}}(M_0\cup\{a_i:i\neq n\}).
\]
\end{proposition}

\begin{proof}
Pass to an arithmetically decidable Skolem expansion of \(\MM_0\) in the
language \(\LL(M_0)\). Apply
Proposition~\ref{prop:skolem-independent-sequence}, moreover clause, with \(A=M_0\) and
\(\psi(x)\equiv(x=x)\) to obtain an arithmetically decidable elementary
extension
\[
\MM_0^{\mathrm{Sk}}\preccurlyeq \BB^{\mathrm{Sk}}
\]
and a Skolem-independent sequence \((e_n)_{n\in\N}\) in \(B\) over
\(M_0\), with \(n\mapsto e_n\) arithmetical.

Expand \(\LL^{\mathrm{Sk}}(M_0)\) further by constants
\(c_0,c_1,\dots\), and let \(\Sigma\) consist of:
\begin{enumerate}[label=(\roman*)]
\item the elementary diagram of \(\MM_0^{\mathrm{Sk}}\);
\item the indiscernibility scheme for \((c_i)_{i\in\N}\) over \(M_0\);
\item for every \(\LL^{\mathrm{Sk}}\)-term
      \(t(v_{i_1},\dots,v_{i_k},\bar m)\), every
      \(n\notin\{i_1,\dots,i_k\}\), and every tuple \(\bar m\) from
      \(M_0\), the inequality
      \[
      c_n\neq t(c_{i_1},\dots,c_{i_k},\bar m).
      \]
\end{enumerate}

We verify finite satisfiability. Fix a finite fragment
\(\Sigma_0\subseteq\Sigma\), involving constants \(c_0,\dots,c_N\). Let
\(\Delta(\bar x)\) be the finite set of formulas appearing in the
indiscernibility instances of \(\Sigma_0\). By
Lemma~\ref{lem:finite-fragment-ind}, applied in
\(\BB^{\mathrm{Sk}}\) to the sequence \((e_n)\), there is an infinite
arithmetical set \(H=\{m_0<m_1<\cdots\}\) such that
\((e_{m_i})_{i\in\N}\) is \(\Delta\)-indiscernible over \(M_0\).
Interpret \(c_j:=e_{m_j}\) for \(j\le N\). This satisfies the
indiscernibility requirements in \(\Sigma_0\). For the
Skolem-independence inequalities, note that since
\(j\mapsto m_j\) is strictly increasing, distinct indices among the
\(c_j\)'s correspond to distinct indices among the \(e_{m_j}\)'s, and
the required inequalities follow from the Skolem independence of the
original sequence \((e_n)\) over \(M_0\).

By Theorem~\ref{thm:compact}, \(\Sigma\) has an arithmetically decidable
model \(\widetilde{\NN}^{\mathrm{Sk}}\). The Skolem-independence
inequalities include constant terms from \(M_0\) and projection terms,
so the elements \(c_i^{\widetilde{\NN}^{\mathrm{Sk}}}\) are pairwise
distinct and all lie outside the distinguished copy of \(M_0\).
Recoding the domain as in the proof of
Proposition~\ref{prop:skolem-independent-sequence}---via an arithmetical
bijection \(f\) from \(\widetilde N\) onto a suitable subset of
\(\N\) extending \(M_0\)---we obtain an arithmetically decidable model
\(\NN^{\mathrm{Sk}}\) with
\(\MM_0^{\mathrm{Sk}}\preccurlyeq \NN^{\mathrm{Sk}}\). Setting
\(a_i:=f(c_i^{\widetilde{\NN}^{\mathrm{Sk}}})\) gives the required
sequence.
\end{proof}

\section{Arithmetic stability}
\label{sec:arith-stability}

We now introduce the arithmetic analogue of \(\omega\)-stability and
show that \(D\)-categoricity forces it.  The formulation used in this
section is slightly more semantic than the usual cardinality-based one:
arithmetical stability is equivalent to the assertion that arithmetic
types over arithmetical parameter sets can be realized all at once in an
arithmetically extendible elementary extension.

This characterization is the main point of the section.  The forward
direction is a compactness argument using the bounded complexity of the
coded family of arithmetic types.  The reverse direction uses the fact
that an arithmetically extendible model sits inside an arithmetically
decidable one, whose satisfaction relation gives a uniform finite-jump
bound for all types it realizes.

The categoricity argument then becomes simple.  If \(T\) is not
arithmetically stable, there is an arithmetically decidable base
\(\MM_0\) such that no arithmetically extendible extension realizes all
arithmetic types over \(M_0\).  We build one model of the prescribed
nonzero arithmetic degree \(D\) over \(\MM_0\), choose an arithmetic
type it omits, and then build another model of the same degree over
\(\MM_0\) realizing that type.  This contradicts \(D\)-categoricity.

Throughout this section, \(T\) is a complete arithmetically definable
theory in a countable arithmetically definable language, and \(T\) has
no finite models.

\begin{definition}[\(n\)-types and arithmetic \(n\)-types]
\label{def:arith-type}
Let \(\MM\) be an arithmetically extendible structure, and let
\(A\subseteq M\). An \emph{\(n\)-type over \(A\) inside \(\MM\)} is a
maximal set \(p(\bar x)\) of \(\LL(A)\)-formulas in the variables
\(\bar x=(x_0,\dots,x_{n-1})\) which is consistent with
\(\Diag_{\mathrm{el}}(\MM)\). We write \(S_n(A,\MM)\) for the space of
all such types.

We say that \(p(\bar x)\in S_n(A,\MM)\) is an
\emph{arithmetic \(n\)-type over \(A\) inside \(\MM\)} if \(A\) is
arithmetical and its code
\[
\Code_A^{\MM}(p):=
\{\ulcorner \varphi(\bar x,\bar a)\urcorner:
\bar a\in A^{<\omega},\ \varphi(\bar x,\bar a)\in p\}
\]
is an arithmetical subset of \(\N\). We write
\(S_n^{\mathrm{ar}}(A,\MM)\) for the countable family of arithmetic \(n\)-types
over \(A\) inside \(\MM\). A \emph{code} for \(S_n^{\mathrm{ar}}(A,\MM)\) is an effective join of codes for its types, such that each type is coded at least once.
\end{definition}

When \(A=M\), we write \(S_n^{\mathrm{ar}}(M,\MM)\), or simply
\(S_n^{\mathrm{ar}}(M)\) when the ambient model is clear.
Throughout this section, formulas over \(A\) are coded as formulas in
the language \(\LL(A)\), with constants only for elements of \(A\).
Equivalently, a code records an ordinary formula code together with a
finite tuple of parameters from \(A\). This is the canonical parameter
coding used for all type codes below.

\begin{lemma}[Invariance under elementary extension]
\label{lem:arith-type-space-invariance}
Let \(\MM\preccurlyeq\NN\) be structures, and let \(A\subseteq M\).
Then, for every \(n\ge 1\), \(S_n(A,\MM)=S_n(A,\NN)\).
If moreover \(\MM\) and \(\NN\) are arithmetically extendible and
\(A\) is arithmetical, then
\(S_n^{\mathrm{ar}}(A,\MM)=S_n^{\mathrm{ar}}(A,\NN)\).
\end{lemma}

\begin{proof}
Elementarity gives the same maximal consistent sets of
\(\LL(A)\)-formulas over both diagrams, and the code
\(\Code_A(p)\) is unchanged since the same literal set \(A\) and
parameter-coding convention are used in both structures.
\end{proof}

\begin{lemma}
\label{lem:arith-type-realized}
Let \(\MM\) be an arithmetically extendible structure, let
\(A\subseteq M\) be arithmetical, and let
\(p(\bar x)\in S_n(A,\MM)\). Then the following are equivalent:
\begin{enumerate}[label=(\arabic*)]
\item \(p(\bar x)\) is an arithmetic \(n\)-type over \(A\) inside
      \(\MM\);
\item some arithmetically decidable elementary extension of \(\MM\)
      realizes \(p\);
\item some arithmetically extendible elementary extension of \(\MM\)
      realizes \(p\).
\end{enumerate}
\end{lemma}

\begin{proof}
The implication \((2)\Rightarrow(3)\) is immediate.

For \((3)\Rightarrow(1)\), let
\(\MM\preccurlyeq\KK\preccurlyeq\widehat{\KK}\), where
\(\widehat{\KK}\) is arithmetically decidable, and suppose
\(\bar b\in K^n\) realizes \(p\). For every formula
\(\varphi(\bar x,\bar a)\) with parameters from \(A\), membership
\(\varphi(\bar x,\bar a)\in p\) is equivalent to
\(\widehat{\KK}\models \varphi(\bar b,\bar a)\).
Since \(\widehat{\KK}\) is arithmetically decidable, \(A\) is
arithmetical, and \(\bar b\) is a fixed finite tuple that may be used as
a parameter in the arithmetical satisfaction relation, the right-hand
side is an arithmetical condition on the code of
\(\varphi(\bar x,\bar a)\). Hence \(p\) is arithmetic.

For \((1)\Rightarrow(2)\), choose an arithmetically decidable elementary
extension \(\MM\preccurlyeq \widehat{\MM}\).
By Lemma~\ref{lem:arith-type-space-invariance}, we may regard \(p\) as
a type over \(A\) inside \(\widehat{\MM}\). Work in the language
\(\LL(\widehat M)\) obtained by adding constants for all elements of
\(\widehat M\), and add new constants
\(\bar c=(c_0,\dots,c_{n-1})\). The theory
\[
\Diag_{\mathrm{el}}(\widehat{\MM})\cup p(\bar c)
\]
is arithmetically definable: the elementary diagram of
\(\widehat{\MM}\) is arithmetical, and the set of formulas in
\(p(\bar c)\) is arithmetical by assumption. Every finite subset is
satisfiable because \(p\in S_n(A,\widehat{\MM})\), i.e. every finite
fragment of \(p\) is consistent with
\(\Diag_{\mathrm{el}}(\widehat{\MM})\). By
Lemma~\ref{lem:base-preserving-coded-compactness}, this theory has an
arithmetically decidable model whose named copy of \(\widehat{\MM}\) is
the actual old structure. Its \(\LL\)-reduct is an arithmetically
decidable elementary extension of \(\MM\) realizing \(p\).
\end{proof}

\begin{definition}[Arithmetically stable]
\label{def:arith-stable}
Let \(T\) be a complete arithmetically definable theory in a countable
language. We say that \(T\) is \emph{arithmetically stable} if for every
arithmetically decidable model \(\MM\models T\) and every arithmetical
set \(A\subseteq M\), 
the family
\[
S_{<\omega}^{\mathrm{ar}}(A,\MM)
:=
\left(S_n^{\mathrm{ar}}(A,\MM)\right)_{n \geq 1}
\]
has an arithmetical code, i.e., there is an effective join \(X = \bigoplus_{n \geq 1} X_n\) such that for each \(n \geq 1\), $X_n$ is a code for \(S_n^{\mathrm{ar}}(A,\MM)\).
\end{definition}

\begin{remark}[Convention for type spaces]
\label{rem:indexed-type-space-convention}
If \(T\) is arithmetically stable and we fix an arithmetical code for \(X = \bigoplus_{n \geq 1} X_n\) for \(S_{<\omega}^{\mathrm{ar}}(A,\MM)\), we may assume that each \(X_n\) contains no repetition, i.e., each arithmetic \(n\)-type in \(S_n^{\mathrm{ar}}(A,\MM)\) is coded exactly once in \(X_n\). This is because we can remove duplicates in an arithmetical way.
\end{remark}

\begin{definition}[Simultaneous realization]
\label{def:realizes-all-arith-types}
Let \(\MM\models T\) be arithmetically extendible and let
\(A\subseteq M\) be arithmetical. An elementary extension
\(\MM\preccurlyeq\NN\models T\) \emph{realizes all arithmetic finite
types over \(A\)} if, for every \(n\geq 1\), every type
\(p(\bar x)\in S_n^{\mathrm{ar}}(A,\MM)\) is realized by some tuple from
\(N^n\).
\end{definition}

\begin{lemma}[Finite simultaneous realization]
\label{lem:finite-simultaneous-realization}
Let \(\MM\models T\) be a model, let \(A\subseteq M\), and let
\(p_0(\bar x_0),\dots,p_{r-1}(\bar x_{r-1})\) be types over \(A\) inside
\(\MM\). For each \(i<r\), let
\(\Delta_i\subseteq p_i\) be finite. Then
\[
\Diag_{\mathrm{el}}(\MM)\cup\bigcup_{i<r}\Delta_i(\bar c_i)
\]
is satisfiable, where the tuples \(\bar c_i\) are new constants of the
corresponding arities.
\end{lemma}

\begin{proof}
Realize the finite fragments one at a time in successive elementary
extensions. The induction invariant is that the base parameter set
\(A\) and all previously realized tuples remain inside the current
elementary extension. After each step,
Lemma~\ref{lem:arith-type-space-invariance} keeps the remaining types
over \(A\) consistent in the new extension.
\end{proof}

\begin{proposition}[Realization characterization of arithmetical stability]
\label{prop:stability-realization-characterization}
Assume \(T\) is complete, arithmetically definable, and has no finite
models. Then the following are equivalent:
\begin{enumerate}[label=(\arabic*)]
\item \(T\) is arithmetically stable.
\item For every arithmetically extendible model \(\MM\models T\) and
      every arithmetical set \(A\subseteq M\), there is an
      arithmetically extendible elementary extension
      \[
      \MM\preccurlyeq\NN\models T
      \]
      which realizes all arithmetic finite types over \(A\).
\end{enumerate}
\end{proposition}

\begin{proof}
Assume first that \(T\) is arithmetically stable. Let
\(\MM\models T\) be arithmetically extendible and let
\(A\subseteq M\) be arithmetical. Choose an arithmetically decidable
elementary extension
\[
\MM\preccurlyeq\widehat{\MM}\models T.
\]
By Lemma~\ref{lem:arith-type-space-invariance}, the arithmetic type
spaces over \(A\) inside \(\MM\) and \(\widehat{\MM}\) agree.

By arithmetical stability, there is an arithmetical list \((p_e(\bar x_e))_{e \in E}\) of all types in the family
\(S_{<\omega}^{\mathrm{ar}}(A,\widehat{\MM})\). Expand the language by constants for the elements of \(\widehat M\) and,
for each \(e\in E\), by a tuple of new constants
\(\bar c_e\) of the same length as \(\bar x_e\). Consider the theory
\[
\Sigma:=
\Diag_{\mathrm{el}}(\widehat{\MM})
\cup
\bigcup_{e\in E} p_e(\bar c_e).
\]
The theory \(\Sigma\) is arithmetically definable because the membership relation
\(\theta\in p_e\) is arithmetical. Every
finite fragment mentions only finitely many of the types \(p_e\) and
finitely many formulas from each of them, so it is satisfiable by
Lemma~\ref{lem:finite-simultaneous-realization}. The elementary-extension
invariance used there is what keeps the remaining types consistent after
earlier ones have been realized in elementary extensions.

By Lemma~\ref{lem:base-preserving-coded-compactness}, \(\Sigma\) has an
arithmetically decidable model in which the named copy of
\(\widehat{\MM}\) is the actual old structure. The \(\LL\)-reduct is
then an arithmetically decidable elementary extension of \(\MM\)
realizing every arithmetic finite type over \(A\). In particular it is
arithmetically extendible.

Conversely, assume clause~\textup{(2)}. Let \(\MM\models T\) be
arithmetically decidable, and let \(A\subseteq M\) be arithmetical. By
clause~\textup{(2)}, choose an arithmetically extendible elementary
extension
\[
\MM\preccurlyeq\NN\models T
\]
which realizes all arithmetic finite types over \(A\). Choose an
arithmetically decidable elementary extension
\[
\NN\preccurlyeq\widehat{\NN}\models T.
\]

Let \(k<\omega\) be large enough so that both the satisfaction relation
of \(\widehat{\NN}\) and the set \(A\) are computable in \(0^{(k)}\).
Index types by pairs \((n,\bar b)\), where \(n\geq 1\) and
\(\bar b\in \widehat N^n\), and let the corresponding type be
\[
\tp^{\widehat{\NN}}(\bar b/A).
\]
For a formula code \(\ulcorner\varphi(\bar x,\bar a)\urcorner\) with
\(\bar a\in A^{<\omega}\), membership in this indexed type is decided
by the \(0^{(k)}\)-computable relation
\[
\widehat{\NN}\models \varphi(\bar b,\bar a).
\]
Thus the family of all types realized in \(\widehat{\NN}\) over \(A\)
is computable in \(0^{(k)}\), uniformly in the arity.

Every type realized in \(\widehat{\NN}\) over the arithmetical set \(A\)
is arithmetic, because \(\widehat{\NN}\) is arithmetically decidable.
Conversely, every arithmetic finite type over \(A\) inside \(\MM\) is
realized in \(\NN\), hence in \(\widehat{\NN}\), by the choice of
\(\NN\). By Lemma~\ref{lem:arith-type-space-invariance}, these are
exactly the arithmetic type spaces over \(A\) inside \(\MM\). Therefore
\(
S_{<\omega}^{\mathrm{ar}}(A,\MM)
\)
has an arithmetical code. Since \(\MM\) and \(A\) were
arbitrary, \(T\) is arithmetically stable.
\end{proof}

\begin{lemma}[Arithmetical elementary amalgamation]
\label{lem:arith-elementary-amalgamation}
Let
\[
\MM_0\preccurlyeq\MM_1,\MM_2\models T
\]
be arithmetically decidable models, with the domains of \(\MM_1\) and
\(\MM_2\) first replaced by disjoint arithmetical copies over \(M_0\).
Then there is an arithmetically decidable model \(\KK\models T\) and
elementary embeddings
\[
f_i:\MM_i\to\KK
\qquad(i=1,2)
\]
such that \(f_1\restr_{M_0}=f_2\restr_{M_0}\).
The amalgam is not asserted to be strong: the images of \(\MM_1\) and
\(\MM_2\) may have additional intersection beyond the common copy of
\(M_0\).
\end{lemma}

\begin{proof}
Work in the language with constants for the elements of \(M_1\cup M_2\),
identifying the two constants for each element of \(M_0\). Consider
\[
\Sigma:=
\Diag_{\mathrm{el}}(\MM_1)
\cup_{\Diag_{\mathrm{el}}(\MM_0)}
\Diag_{\mathrm{el}}(\MM_2).
\]
Since the two diagrams and the common domain \(M_0\) are arithmetical,
\(\Sigma\) is arithmetically definable.

For finite satisfiability, collect the finitely many constants from the
\(\MM_1\)-side not belonging to \(M_0\) as variables \(\bar x\).  The
\(\MM_1\)-part of the fragment gives a finite conjunction
\(\theta(\bar x,\bar m)\), with \(\bar m\in M_0^{<\omega}\), such that
\(\MM_1\models \exists\bar x\,\theta(\bar x,\bar m)\).  Since
\(\MM_0\preccurlyeq\MM_1\), choose \(\bar u\in M_0^{|\bar x|}\)
satisfying this existential requirement; then interpret those finitely
many \(\MM_1\)-constants by \(\bar u\), and interpret the \(\MM_2\)-side
inside \(\MM_2\). Hence every finite fragment of \(\Sigma\) is
satisfiable.

By Theorem~\ref{thm:compact}, \(\Sigma\) has an arithmetically decidable
model. The interpretations of the named constants give elementary
embeddings of \(\MM_1\) and \(\MM_2\) into this model, agreeing on
\(M_0\). The full elementary diagrams make each individual embedding
injective, although the two images need not be disjoint over \(M_0\).
\end{proof}

\begin{lemma}[Literal left-base elementary amalgamation]
\label{lem:literal-left-amalgamation}
In the setting of Lemma~\ref{lem:arith-elementary-amalgamation}, one may
choose the amalgam so that \(\MM_1\) is contained literally as an
elementary submodel. Thus there are an arithmetically decidable
\(\KK\models T\) and an elementary embedding
\(j:\MM_2\to\KK\) over \(M_0\) with
\[
\MM_1\preccurlyeq\KK.
\]
No disjointness beyond the common base is asserted.
\end{lemma}

\begin{proof}
Apply Lemma~\ref{lem:arith-elementary-amalgamation}, obtaining an
arithmetically decidable amalgam with elementary embeddings of
\(\MM_1\) and \(\MM_2\). Then use the base-preserving recoding lemma,
Lemma~\ref{lem:base-preserving-coded-compactness}, equivalently the
same transport of the coded domain, to fix the embedded copy
\(f_1[M_1]\) pointwise as the original domain \(M_1\). The remaining
elements are moved into an arithmetical fresh subset disjoint from this
reserved copy, with infinite complement. The transported structure is
arithmetically decidable, contains \(\MM_1\) literally, and carries the
embedded copy of \(\MM_2\) over \(M_0\).
\end{proof}

\begin{corollary}[Model-base form]
\label{cor:model-base-realization-characterization}
Assume \(T\) is complete, arithmetically definable, and has no finite
models. Then \(T\) is arithmetically stable if and only if for every
arithmetically decidable model \(\MM_0\models T\), there is an
arithmetically extendible elementary extension
\[
\MM_0\preccurlyeq\NN\models T
\]
which realizes all arithmetic finite types over \(M_0\).
\end{corollary}

\begin{proof}
The forward direction is immediate from
Proposition~\ref{prop:stability-realization-characterization}.

For the converse, it is enough to prove clause~\textup{(2)} of
Proposition~\ref{prop:stability-realization-characterization}. Let
\(\MM\models T\) be arithmetically extendible and let
\(A\subseteq M\) be arithmetical. Choose an arithmetically decidable
elementary extension
\[
\MM\preccurlyeq\widehat{\MM}.
\]
Pass to a Skolem expansion of \(\widehat{\MM}\), choosing least
witnesses from the arithmetically decidable presentation of
\(\widehat{\MM}\). This Skolem expansion is arithmetically decidable.
Let
\[
\MM_0:=\Hull^{\widehat{\MM}^{\mathrm{Sk}}}(A)\!\restr_\LL.
\]
By Theorem~\ref{thm:sk}, this Skolem hull is arithmetically decidable,
contains \(A\), and is elementary in \(\widehat{\MM}\); hence
\(\MM_0\preccurlyeq\widehat{\MM}\).

By the assumed model-base property, choose an arithmetically extendible
extension \(\NN_0\succeq\MM_0\) realizing all arithmetic finite types
over \(M_0\), and then choose an arithmetically decidable elementary
extension
\[
\NN_0\preccurlyeq\widehat{\NN}_0.
\]
By Lemma~\ref{lem:literal-left-amalgamation}, amalgamate
\(\widehat{\MM}\) and \(\widehat{\NN}_0\) over \(\MM_0\), preserving
\(\widehat{\MM}\) literally. We obtain an arithmetically decidable model
\(\NN\) with
\[
\MM\preccurlyeq\widehat{\MM}\preccurlyeq\NN
\]
and with an elementary copy of \(\NN_0\) over \(\MM_0\) inside \(\NN\).

Now let \(p\in S_n^{\mathrm{ar}}(A,\MM)\). By
Lemma~\ref{lem:arith-type-realized}, realize \(p\) in an
arithmetically decidable elementary extension of \(\widehat{\MM}\), and
let \(q\) be the type of the realizing tuple over \(M_0\). Then \(q\) is
arithmetic by Lemma~\ref{lem:arith-type-realized} and satisfies
\(q\!\upharpoonright_A=p\). The model \(\NN_0\) realizes \(q\), so the
amalgam \(\NN\) realizes \(p\). By
Proposition~\ref{prop:stability-realization-characterization}, \(T\) is
arithmetically stable.
\end{proof}

\begin{lemma}[Realizing a prescribed arithmetic type in a prescribed degree]
\label{lem:preserve-type-code-degree}
Let \(\MM_0\models T\) be arithmetically decidable, let
\[
p(\bar x)\in S_n^{\mathrm{ar}}(M_0,\MM_0),
\]
and let \(D>\mathbf 0\) be an arithmetic degree. Then there is an
arithmetically extendible model
\[
\MM\models T
\]
such that
\[
\MM_0\preccurlyeq\MM,
\qquad
\degAr(M)=D,
\]
and \(\MM\) realizes \(p\).
\end{lemma}

\begin{proof}
By Lemma~\ref{lem:arith-type-realized}, choose an arithmetically
decidable elementary extension
\[
\MM_0\preccurlyeq\widehat{\MM}\models T
\]
and a tuple \(\bar b\in \widehat M^n\) realizing \(p\).

Pass to an arithmetically decidable Skolem expansion of
\(\widehat{\MM}\) in the language naming all elements of \(\widehat M\).
Apply Proposition~\ref{prop:skolem-independent-sequence}, moreover clause, with
\(A=\widehat M\) and \(\psi(x)\equiv(x=x)\) to obtain an arithmetically
decidable elementary extension
\[
\widehat{\MM}^{\mathrm{Sk}}\preccurlyeq\BB^{\mathrm{Sk}}
\]
and an arithmetical sequence \((a_i)_{i\in\N}\) in the home sort which
is Skolem-independent over \(\widehat M\).

Choose \(S\subseteq\N\) with \(\degAr(S)=D\), and put
\[
\MM:=
\Hull^{\BB^{\mathrm{Sk}}}
\bigl(\widehat M\cup\{a_i:i\in S\}\bigr)\!\restr_\LL.
\]
Then \(\MM_0\preccurlyeq\widehat{\MM}\preccurlyeq\MM\), and
\(\bar b\in M^n\), so \(\MM\) realizes \(p\). Since \(\MM\) is a Skolem
hull inside the arithmetically decidable model \(\BB^{\mathrm{Sk}}\),
it is arithmetically extendible.

By the Skolem-hull coding lemma,
Lemma~\ref{lem:skolem-hull-degree-over-base}, applied over the
arithmetical base \(\widehat M\), we have
\(\degAr(M)=\degAr(S)=D\).
\end{proof}

\begin{proposition}[Categoricity at one degree implies arithmetical stability]
\label{prop:cat-implies-stability}
If \(T\) is \(D\)-categorical for some arithmetic degree \(D>\mathbf 0\),
then \(T\) is arithmetically stable.
\end{proposition}

\begin{proof}
Assume \(T\) is \(D\)-categorical for some \(D>\mathbf 0\). Suppose,
toward a contradiction, that \(T\) is not arithmetically stable. By
Corollary~\ref{cor:model-base-realization-characterization}, there is
an arithmetically decidable model \(\MM_0\models T\) such that no
arithmetically extendible elementary extension of \(\MM_0\) realizes all
arithmetic finite types over \(M_0\).

By Corollary~\ref{cor:realizing-arbitrary-degree}\textup{(ii)}, choose
an arithmetically extendible model
\(\MM_{\mathrm{low}}\succeq\MM_0\) with
\(\degAr(M_{\mathrm{low}})=D\).
By the choice of \(\MM_0\), the model \(\MM_{\mathrm{low}}\) omits some
arithmetic finite type over \(M_0\). Fix
\(p(\bar x)\in S_n^{\mathrm{ar}}(M_0,\MM_0)\) omitted in
\(\MM_{\mathrm{low}}\).

By Lemma~\ref{lem:preserve-type-code-degree}, choose an arithmetically
extendible model \(\MM_{\mathrm{high}}\succeq\MM_0\) with
\(\degAr(M_{\mathrm{high}})=D\) which realizes \(p\).

By \(D\)-categoricity, there is a \(D\)-preserving isomorphism
\(f:\MM_{\mathrm{low}}\cong\MM_{\mathrm{high}}\) fixing \(M_0\)
pointwise. If \(\bar b\in M_{\mathrm{high}}^n\) realizes \(p\), then
\(f^{-1}(\bar b)\) realizes \(p\) in \(\MM_{\mathrm{low}}\),
contradicting the choice of \(p\).
\end{proof}

\section{Vaughtian pairs}
\label{sec:vaughtian-pairs}

We now prove the second forward implication in the arithmetic
Baldwin--Lachlan strategy, namely that \(D\)-categoricity rules out
Vaughtian pairs. Starting from a Vaughtian pair, we first reduce to a
configuration over a fixed arithmetically decidable base. We then build
two arithmetically extendible models of the same arithmetic degree: a
``low'' model in which the distinguished definable set remains
arithmetical, and a ``high'' model in which that same definable set
recovers the full ambient degree. This contradicts \(D\)-categoricity.

\begin{definition}[Vaughtian pair]
\label{def:vaughtian-pair}
Let \(T\) be complete. A pair \(\MM_1\preccurlyeq \MM_2\models T\) with
\(\MM_1\neq \MM_2\) is a \emph{Vaughtian pair} if there exist a tuple
\(\bar a\in M_1^{<\omega}\) and a formula \(\varphi(x,\bar y)\) such
that \(\varphi(\MM_1,\bar a)=\varphi(\MM_2,\bar a)\), and this common
definable set is infinite. If moreover \(\MM_1\) and
\(\MM_2\) are arithmetically decidable, we call
\((\MM_1,\MM_2)\) an \emph{arithmetically decidable Vaughtian pair}.
\end{definition}

\begin{lemma}
\label{lem:vaughtian-equivalent-ar}
Assume \(T\) is complete and arithmetically definable. Then \(T\) has a
Vaughtian pair if and only if it has an arithmetically decidable
Vaughtian pair.
\end{lemma}

\begin{proof}
Only the forward direction requires argument. Given a Vaughtian pair
\(\MM_1\preccurlyeq \MM_2\models T\) witnessed by
\(\bar a\in M_1^{<\omega}\) and \(\varphi(x,\bar y)\), expand the
language by a unary predicate \(P\), constants for \(\bar a\), and one
additional constant \(d\). The theory \(\Sigma\) asserts that the
ambient structure models \(T\); that \(P\) defines an elementary
submodel (via the Tarski--Vaught scheme); that the named constants and
every realization of \(\varphi(x,\bar a)\) lie in \(P\); that
\(\varphi(P,\bar a)\) is infinite; and that \(d\notin P\). This theory
is arithmetically definable, and every finite fragment is realized by
\((\MM_1,\MM_2)\). By Theorem~\ref{thm:compact}, \(\Sigma\) has an
arithmetically decidable model, whose \(P\)-reduct gives the required
pair.
\end{proof}

\begin{lemma}
\label{lem:vaughtian-pair-infinite-gap}
Assume \(T\) is complete and arithmetically definable. If \(T\) has a
Vaughtian pair, then it has an arithmetically decidable Vaughtian pair
\(\MM_1\preccurlyeq \MM_2\models T\) with \(M_2\setminus M_1\) infinite.
\end{lemma}

\begin{proof}
Starting from a Vaughtian pair, modify the compactness argument in
Lemma~\ref{lem:vaughtian-equivalent-ar} by replacing the single
constant \(d\) with countably many constants \(d_0,d_1,\dots\),
required to be pairwise distinct and outside \(P\).

The finite satisfiability check is made in the pair language.  Fix
\(r<\omega\).  Take \(r\) copies of the original Vaughtian pair over the
common named \(P\)-part, and choose in the \(i\)-th copy an element
outside \(P\) to interpret \(d_i\).  By the usual elementary-amalgamation
argument for pair diagrams over a named elementary submodel, the finite
union of these pair diagrams is consistent, with the \(P\)-parts
identified and the chosen outside elements required to be distinct.  Each
copy satisfies
\[
\forall x\,(\varphi(x,\bar a)\rightarrow P(x)),
\]
and the common pair diagram includes the Tarski--Vaught scheme for \(P\).
Thus every finite fragment of the expanded theory is satisfiable as a
pair structure preserving the original Vaughtian trace. Apply
Theorem~\ref{thm:compact} as before.
\end{proof}

By Lemmas~\ref{lem:vaughtian-equivalent-ar} and
\ref{lem:vaughtian-pair-infinite-gap}, if \(T\) has a Vaughtian pair,
then there is an arithmetically decidable pair structure
\[
\MM_1\preccurlyeq \MM_2\models T
\]
in a language \(\LL^\pair=\LL\cup\{P\}\), a tuple
\(\bar a\in M_1^{<\omega}\), and a formula \(\varphi(x,\bar y)\) such
that:
\begin{enumerate}[label=(\roman*)]
\item \(P\) names \(M_1\) inside \(\MM_2\);
\item \((\MM_1,\MM_2)\) is a Vaughtian pair witnessed by
      \((\bar a,\varphi)\);
\item the common set
      \[
      \varphi(\MM_1,\bar a)=\varphi(\MM_2,\bar a)
      \]
      is infinite;
\item \(M_2\setminus M_1\) is infinite.
\end{enumerate}

In the constructions below we may replace the ambient pair by an
arithmetically decidable pair model over the same named \(P\)-part.  The
\(P\)-part is then used as the common base model.  This harmless
renaming is important: the sentence
\(\forall x\,(\varphi(x,\bar a)\rightarrow P(x))\) guarantees that the
distinguished definable set remains inside the chosen base.

\begin{lemma}[Fixed \(P\)-part with an outside coding sequence]
\label{lem:fixed-p-outside-coding-sequence}
In the arithmetically decidable pair constructed above, there is an
arithmetically decidable pair model \(\widehat{\MM}^{\pair}\) over the
fixed named \(P\)-part, satisfying the elementary pair theory of the
original pair over that named \(P\)-part, such that
\[
P^{\widehat{\MM}}=P^{\MM_2}
\]
literally, and there is an arithmetical sequence
\[
(c_n)_{n\in\N}\subseteq \widehat M\setminus P^{\widehat{\MM}}
\]
which is Skolem-independent over this actual \(P\)-part in the
pair-Skolem language.
\end{lemma}

\begin{proof}
Work in the pair language after naming every element of the old
\(P\)-part.  Let the base theory be the complete
\(\LL^\pair(P)\)-theory of the original pair \((\MM_2,P^{\MM_2})\) with
these constants.  Thus the old \(P\)-part is fixed pointwise, but the
outside part of \(\MM_2\) is not named literally.  Now apply
Skolemization to this expanded pair language.  Add constants \(c_n\) and
impose the usual outside-\(P\) and
Skolem-independence requirements:
\[
\neg P(c_n)
\]
for each \(n\), and, for each pair-Skolem term \(\tau\) not using
\(c_n\),
\[
c_n\neq \tau(\bar p,c_{i_0},\ldots,c_{i_{r-1}}),
\]
where \(\bar p\) is a tuple from the named \(P\)-part and
\(i_j\neq n\).

We also schedule explicit dense requirements ensuring that the
\(P\)-part does not grow. Enumerate all closed terms \(\sigma\) in the
Henkin language. For each \(\sigma\), require one of the following
finite alternatives:
\[
\neg P(\sigma)
\qquad\text{or}\qquad
\sigma=p
\quad\text{for some named }p\in P^{\MM_2}.
\]
At the construction stage for \(\sigma\), search through the arithmetical
list
\[
\{\neg P(\sigma)\}\cup\{\sigma=p:p\in P^{\MM_2}\}
\]
and choose the least member whose addition to the current finite condition
is consistent.  This is the finite extension \(F_\sigma\) required by
Corollary~\ref{cor:reserved-constant-omitting}.  The search is
arithmetical because finite syntactic consistency with the
arithmetically definable pair theory is recognized by the proof-search
predicate used in the Henkin construction.

The search always succeeds. Indeed, any finite set of requirements
mentions only finitely many \(c_n\)'s, finitely many terms, and finitely
many old \(P\)-parameters.  For finite satisfiability, choose a Skolem
expansion of the original pair interpreting the finitely many Skolem
functions occurring in the fragment.  The finite Skolem avoidance lemma,
applied in this expanded pair language to the infinite definable set
\(M_2\setminus P^{\MM_2}\), lets us interpret the finitely many \(c_n\)'s
outside \(P^{\MM_2}\) while satisfying the avoidance requirements.  Each
scheduled closed term then evaluates either outside \(P^{\MM_2}\) or to a
specific old element \(p\in P^{\MM_2}\), so the corresponding dense
alternative is also satisfied.

Thus the reserved dense-requirement form of the omitting/Henkin
construction, Corollary~\ref{cor:reserved-constant-omitting}, produces
an arithmetically decidable pair model satisfying all requirements. The
closed-term alternatives imply that every element of the term model
lying in \(P\) is equal to a named element of the old \(P\)-part; after
the usual transport to a coded domain, \(P^{\widehat{\MM}}=P^{\MM_2}\)
literally. The constants \(c_n\) remain outside \(P\) and satisfy the
Skolem-independence scheme over this fixed \(P\)-part.
\end{proof}

\begin{lemma}[Low model]
\label{lem:vaughtian-low}
For every nonzero arithmetic degree \(D\), there exist an
arithmetically decidable model \(\MM_0\models T\), a tuple
\(\bar a\in M_0^{<\omega}\), and an arithmetically extendible model
\(\MM_{\mathrm{low}}\models T\) such that
\[
\MM_0\preccurlyeq \MM_{\mathrm{low}},
\qquad
\degAr(M_{\mathrm{low}})=D,
\]
and
\[
\degAr(\varphi(\MM_{\mathrm{low}},\bar a))=\mathbf 0.
\]
Moreover, \(\MM_0\) may be chosen to be the fixed \(P\)-part of the
arithmetically decidable pair model constructed in
Lemma~\ref{lem:fixed-p-outside-coding-sequence}.
\end{lemma}

\begin{proof}
Apply Lemma~\ref{lem:fixed-p-outside-coding-sequence}. We obtain an
arithmetically decidable pair model
\(\widehat{\MM}^{\pair}\) with \(P^{\widehat{\MM}}=P^{\MM_2}\)
literally and an arithmetical sequence
\[
(c_n)_{n\in\N}\subseteq \widehat M\setminus P^{\widehat{\MM}}
\]
Skolem-independent over this actual \(P\)-part.

Let
\(\MM_0:=\widehat{\MM}^\pair\!\restr_{P^{\widehat{\MM}}}\!\restr_\LL\).
Thus \(\MM_0\models T\) is arithmetically decidable,
\(\MM_0\preccurlyeq \widehat{\MM}\!\restr_\LL\), and
\(\bar a\in M_0^{<\omega}\). Moreover, since the pair satisfies
\(\forall x\,(\varphi(x,\bar a)\rightarrow P(x))\), we have
\(\varphi(\widehat{\MM},\bar a)=\varphi(\MM_0,\bar a)\).

Fix \(S\subseteq\N\) with \(\degAr(S)=D\), and define
\[
\MM_{\mathrm{low}}^{\pair}:=
\Hull^{\widehat{\MM}^{\pair}}
\bigl(P^{\widehat{\MM}}\cup\{c_n:n\in S\}\bigr),
\qquad
\MM_{\mathrm{low}}:=\MM_{\mathrm{low}}^{\pair}\!\restr_\LL.
\]
Since \(\MM_{\mathrm{low}}^{\pair}\) is a Skolem hull inside
\(\widehat{\MM}^{\pair}\), Theorem~\ref{thm:sk}(3) gives
\(\MM_0\preccurlyeq \MM_{\mathrm{low}}
\preccurlyeq \widehat{\MM}\!\restr_\LL\),
so \(\MM_{\mathrm{low}}\) is arithmetically extendible.
By Lemma~\ref{lem:skolem-hull-degree-over-base}, applied in the pair
Skolem language over the arithmetical base \(P^{\widehat{\MM}}\), we
have \(\degAr(M_{\mathrm{low}})=D\).

Finally, the pair sentence
\(\forall x\,(\varphi(x,\bar a)\rightarrow P(x))\) is inherited by the
elementary submodel \(\MM_{\mathrm{low}}^{\pair}\), so
\(\varphi(\MM_{\mathrm{low}},\bar a)\subseteq P^{\widehat{\MM}}=M_0\);
together with
\(\MM_0\preccurlyeq \MM_{\mathrm{low}}\) this gives
\(\varphi(\MM_{\mathrm{low}},\bar a)
=\varphi(\MM_0,\bar a)\).
Since \(\MM_0\) is arithmetically decidable and \(\bar a\) is finite,
\(\degAr(\varphi(\MM_{\mathrm{low}},\bar a))=\mathbf 0\).
\end{proof}

\begin{lemma}[High model]
\label{lem:vaughtian-high}
Let \(\MM_0\models T\) be the arithmetically decidable base obtained in
Lemma~\ref{lem:vaughtian-low}. For every nonzero arithmetic degree
\(D\), there exists an arithmetically extendible model
\(\MM_{\mathrm{high}}\models T\) such that
\[
\MM_0\preccurlyeq \MM_{\mathrm{high}},
\qquad
\degAr(M_{\mathrm{high}})=D,
\]
and
\[
\degAr(\varphi(\MM_{\mathrm{high}},\bar a))=D.
\]
\end{lemma}

\begin{proof}
We code the prescribed degree directly into the distinguished definable
set \(\varphi(x,\bar a)\).

Pass to an arithmetically decidable Skolem expansion of an elementary
extension of \(\MM_0\). By
Proposition~\ref{prop:skolem-independent-sequence}, moreover clause, applied with
base \(M_0\) and infinite definable set \(\varphi(x,\bar a)\), we obtain an
arithmetically decidable elementary extension
\[
\BB^{\mathrm{Sk}}\succcurlyeq \MM_0^{\mathrm{Sk}}
\]
together with an arithmetical sequence
\[
(b_n)_{n\in\N}\subseteq \varphi(\BB\!\restr_\LL,\bar a)
\]
which is Skolem-independent over \(M_0\).

Fix \(S\subseteq\N\) with \(\degAr(S)=D\), and define
\[
\MM_{\mathrm{high}}:=
\Hull^{\BB^{\mathrm{Sk}}}
\bigl(M_0\cup\{b_n:n\in S\}\bigr)\!\restr_\LL.
\]
Then
\(\MM_0\preccurlyeq \MM_{\mathrm{high}}
\preccurlyeq \BB\!\restr_\LL\),
so \(\MM_{\mathrm{high}}\) is arithmetically extendible.
By Lemma~\ref{lem:skolem-hull-degree-over-base}, applied over the
arithmetical base \(M_0\),
\(\degAr(M_{\mathrm{high}})=D\).

Since each \(b_n\) lies in \(\varphi(\BB\!\restr_\LL,\bar a)\) and
\(\MM_{\mathrm{high}}\preccurlyeq \BB\!\restr_\LL\), Skolem independence
gives
\[
b_n\in \varphi(\MM_{\mathrm{high}},\bar a)
\quad\Longleftrightarrow\quad
b_n\in M_{\mathrm{high}}
\quad\Longleftrightarrow\quad
n\in S.
\]
Thus \(S\leAr \varphi(\MM_{\mathrm{high}},\bar a)\).
For the reverse reduction,
\(\varphi(\MM_{\mathrm{high}},\bar a)\leAr
M_{\mathrm{high}}\) by Lemma~\ref{lem:ambient-arithmetical-inequalities}
\textup{(i)}.
Combining this with \(\degAr(M_{\mathrm{high}})=D\), we obtain
\(\degAr(\varphi(\MM_{\mathrm{high}},\bar a))=D\).
\end{proof}

\begin{proposition}[Categoricity at one degree implies no Vaughtian pairs]
\label{prop:d-cat-no-vaughtian}
If \(T\) is \(D\)-categorical for some arithmetic degree \(D>\mathbf 0\),
then \(T\) has no Vaughtian pair. Equivalently, \(T\) has no
arithmetically decidable Vaughtian pair.
\end{proposition}

\begin{proof}
Assume toward a contradiction that \(T\) has a Vaughtian pair. By
Lemmas~\ref{lem:vaughtian-equivalent-ar} and
\ref{lem:vaughtian-pair-infinite-gap}, and then by the setup described
above, choose an arithmetically decidable base model
\[
\MM_0\models T,
\]
a tuple \(\bar a\in M_0^{<\omega}\), and a formula
\(\varphi(x,\bar y)\) such that the corresponding Vaughtian-pair
configuration satisfies:
\begin{enumerate}[label=(\roman*)]
\item \(\varphi(\MM_0,\bar a)\) is infinite;
\item in the relevant arithmetically decidable pair model,
      \[
      \forall x\,(\varphi(x,\bar a)\rightarrow P(x));
      \]
\item the complement of the \(P\)-part is infinite.
\end{enumerate}

Applying Lemmas~\ref{lem:vaughtian-low} and
\ref{lem:vaughtian-high} with the given degree \(D\), using the same
base \(\MM_0\), we obtain arithmetically extendible models
\[
\MM_0\preccurlyeq \MM_{\mathrm{low}},\MM_{\mathrm{high}}\models T
\]
such that
\[
\degAr(M_{\mathrm{low}})=\degAr(M_{\mathrm{high}})=D
\]
and
\[
\degAr(\varphi(\MM_{\mathrm{low}},\bar a))=\mathbf 0
\qquad\text{while}\qquad
\degAr(\varphi(\MM_{\mathrm{high}},\bar a))=D.
\]

By \(D\)-categoricity over the common arithmetically decidable base
model \(\MM_0\), there is a \(D\)-preserving isomorphism
\[
f:\MM_{\mathrm{low}}\cong\MM_{\mathrm{high}}
\]
fixing \(M_0\) pointwise. In particular \(f(\bar a)=\bar a\), so
\[
f[\varphi(\MM_{\mathrm{low}},\bar a)]
=
\varphi(\MM_{\mathrm{high}},\bar a).
\]
Since
\(\degAr(\varphi(\MM_{\mathrm{high}},\bar a))=D\)
and \(f^{-1}\) is \(D\)-preserving, the inverse half of the
\(D\)-preservation definition gives
\(\degAr(\varphi(\MM_{\mathrm{low}},\bar a))=D\),
contradicting
\(\degAr(\varphi(\MM_{\mathrm{low}},\bar a))=\mathbf 0\).
Therefore \(T\) has no Vaughtian pair.
\end{proof}

\section{Morley's analysis for arithmetic types}
\label{sec:morley-analysis-arith-types}

We now begin the rank analysis on spaces of arithmetic types, and
throughout the first part of the section we assume that \(T\) is
arithmetically stable. The section has three main tasks. First, we
introduce a finite Cantor--Bendixson-style rank on the coded spaces of
arithmetic types and analyze the corresponding finite derivative stages.
Second, after adding the assumption that \(T\) has no Vaughtian pairs, we
prove the isolated-type finiteness statement needed later. Third, we use
these facts to relate positive rank to infinitude on the base model.

The rank used here is inspired by the Cantor--Bendixson analysis of
Stone spaces, but it is not literally the usual topological
Cantor--Bendixson rank. In our setting the spaces of arithmetic types
are available only as coded arithmetical objects, and the relevant notion
of isolation is isolation inside the coded arithmetic type space.

Let \(\MM\models T\) be an arithmetically decidable model, and let
\(A\subseteq M\) be arithmetical. In this section, when the ambient
model is clear, we write \(S_n^{\mathrm{ar}}(A)\) for
\(S_n^{\mathrm{ar}}(A,\MM)\).

\begin{definition}[Finite Cantor--Bendixson rank]
\label{def:finite-cb-rank}
Let \(\MM\models T\) be arithmetically decidable, let
\(A\subseteq M\) be arithmetical, and let \(n\geq 1\).

For \(p(\bar x)\in S_n^{\mathrm{ar}}(A,\MM)\) and \(m<\omega\), we
simultaneously define the relations \(\CB(p)\geq m\) and
\(\CB(p)=m\) by induction on \(m\).
Only finite stages are being defined here. A type need not have an exact
finite rank, and no infinite Cantor--Bendixson rank is used anywhere in
the argument.

\begin{enumerate}[label=(\roman*)]
\item For every \(p\in S_n^{\mathrm{ar}}(A,\MM)\), set
      \(\CB(p)\geq 0\).

\item Suppose \(m<\omega\) and that the relations \(\CB(q)\geq m\) and
      \(\CB(q)=m\) have already been defined for all
      \(q\in S_n^{\mathrm{ar}}(A,\MM)\). We say that \(\CB(p)=m\) if
      \(\CB(p)\geq m\) and there is a formula
      \(\varphi(\bar x,\bar a)\in p\), with
      \(\bar a\in A^{<\omega}\), such that whenever
	      \(q\in S_n^{\mathrm{ar}}(A,\MM)\) satisfies
	      \(\CB(q)\geq m\) and \(\varphi(\bar x,\bar a)\in q\), one has
	      \(q=p\).

\item We say that \(\CB(p)\geq m+1\) if \(\CB(p)\geq m\) and
      \(\CB(p)\neq m\).
\end{enumerate}

We also write
\[
S_n^{\mathrm{ar}}(A,\MM)^{(\geq m)}
:=
\{p\in S_n^{\mathrm{ar}}(A,\MM):\CB(p)\geq m\}.
\]
When the ambient model is clear, we write this as
\(S_n^{\mathrm{ar}}(A)^{(\geq m)}\).
\end{definition}

\begin{lemma}
\label{lem:finite-cb-stages-arith}
Let \(\MM\models T\) be arithmetically decidable, let
\(A\subseteq M\) be arithmetical, let \(n\geq 1\), and suppose that
\(A\) and \(S_n^{\mathrm{ar}}(A,\MM)\) are computable in \(X\). Then, for
every \(m<\omega\):
\begin{enumerate}[label=(\roman*)]
\item the relation \(\CB(p)\geq m\) on
      \(p\in S_n^{\mathrm{ar}}(A,\MM)\) is computable in
      \(X^{(2m)}\);
\item the relation \(\CB(p)=m\) on
      \(p\in S_n^{\mathrm{ar}}(A,\MM)\) is computable in
      \(X^{(2m+2)}\).
\end{enumerate}
\end{lemma}

\begin{proof}
We argue by induction on \(m\). For \(m=0\), the relation \(\CB(p)\geq 0\) is just membership in
\(S_n^{\mathrm{ar}}(A,\MM)\), which is computable in \(X\).

Next, \(\CB(p)=m\) holds exactly when \(\CB(p)\geq m\) and there is
some formula \(\varphi(\bar x,\bar a)\in p\) such that whenever
\(q\in S_n^{\mathrm{ar}}(A,\MM)\) satisfies \(\CB(q)\geq m\) and
\(\varphi(\bar x,\bar a)\in q\), one has \(q=p\). Since
\(\CB(q)\geq m\) and \(A\) are computable in
\(X^{(2m)}\), this condition is computable in \(X^{(2m+2)}\) (here we are assuming the convention in Remark~\ref{rem:indexed-type-space-convention}).

Finally, \(\CB(p)\geq m+1\) if and only if \(\CB(p)\geq m\) and
\(\CB(p)\neq m\).
By the induction hypothesis, both clauses are computable in \(X^{(2m+2)}\).
\end{proof}

\noindent For a formula \(\theta(\bar x,\bar a)\) with parameters from
the arithmetical set \(A\), write
\[
[\theta(\bar x,\bar a)]_{\mathrm{ar}}
:=
\{p\in S_n^{\mathrm{ar}}(A,\MM):\theta(\bar x,\bar a)\in p\}.
\]
This is the basic clopen subset of the coded arithmetic type space
determined by the formula code.

\begin{lemma}[Density of finite-rank arithmetic types]
\label{lem:cb-rank-dense}
Assume that \(T\) is arithmetically stable. Let
\(\MM\models T\) be arithmetically decidable, let \(A\subseteq M\) be
arithmetical, let \(n\geq 1\), let \(m<\omega\), and let
\(\varphi(\bar x,\bar a)\) be a formula with parameters from \(A\) such
that
\[
[\varphi(\bar x,\bar a)]_{\mathrm{ar}}
\cap
S_n^{\mathrm{ar}}(A,\MM)^{(\geq m)}
\neq\varnothing.
\]
Then there is some
\(
p(\bar x)\in
[\varphi(\bar x,\bar a)]_{\mathrm{ar}}
\cap
S_n^{\mathrm{ar}}(A,\MM)
\)
such that
\[
\CB(p)=m.
\]
In particular, isolated arithmetic \(n\)-types are dense in
\(S_n^{\mathrm{ar}}(A,\MM)\).
\end{lemma}

\begin{proof}
By arithmetical stability, fix an arithmetical code \(X\) for \(S_n^{\mathrm{ar}}(A,\MM)\). All clopen sets and derivative stages below are taken inside this coded
arithmetic type space, not inside the full Stone space.

First suppose \(m=0\). Assume toward a contradiction that
\(
[\varphi(\bar x,\bar a)]_{\mathrm{ar}}\) is nonempty but contains no isolated arithmetic type. We shall construct an \(X'\)-computable tree
    \[
    \{\varphi_\sigma(\bar x):\sigma\in 2^{<\omega}\}
    \]
such that:
\begin{enumerate}[label=(\roman*)]
\item \([\varphi_\sigma(\bar x)]_{\mathrm{ar}}\neq\varnothing\);
\item \([\varphi_\sigma(\bar x)]_{\mathrm{ar}}\) contains no isolated
      arithmetic type;
\item each child implies its parent;
\item the two children of any node are inconsistent.
\end{enumerate}

Set \(\varphi_\emptyset(\bar x):=\varphi(\bar x,\bar a)\).
Suppose \(\varphi_\sigma\) has been defined. Since
\([\varphi_\sigma]_{\mathrm{ar}}\) contains no isolated type, it cannot
be a singleton. Let \(\psi(\bar x)\) be the formula with least code which lies in some type in \([\varphi_\sigma]_{\mathrm{ar}}\) and lies outside some type in \([\varphi_\sigma]_{\mathrm{ar}}\). This formula can be computed using \(X'\). Define
\[
\varphi_{\sigma^\frown 0}:=
\varphi_\sigma\wedge\psi,
\qquad
\varphi_{\sigma^\frown 1}:=
\varphi_\sigma\wedge\neg\psi.
\]
Both children are nonempty in the arithmetic type space, both imply the
parent, and they are incompatible. Moreover, neither child contains an
isolated arithmetic type; otherwise the same isolating formula, conjoined
with the child formula, would isolate an arithmetic type inside
\([\varphi_\sigma]_{\mathrm{ar}}\), contrary to the induction
hypothesis at \(\sigma\).

Let \(F \in 2^\omega\) be arithmetical such that \(F\) is not \(X'\)-computable, e.g., \(F = X''\). The theory
\[
\Diag_{\mathrm{el}}(\MM)
\cup
\{\varphi_{F\upharpoonright s}(\bar x):s<\omega\}
\]
is arithmetically definable, and every finite
subset is satisfiable. By
compactness and Lemma~\ref{lem:arith-type-realized}, there is an
arithmetic type
\(q_F\in S_n^{\mathrm{ar}}(A,\MM)\)
containing every \(\varphi_{F\upharpoonright s}\).

We claim that \(F\leq_T X' \oplus q_F\). Indeed, given \(s\), using \(X'\) and the previously recovered
    initial segment \(F\upharpoonright s\), compute the two child formulas
    \(\varphi_{(F\upharpoonright s)^\frown 0}\) and
    \(\varphi_{(F\upharpoonright s)^\frown 1}\).
    These are contradictory and exhaustive
    relative to the parent formula already in \(q_F\), so exactly one
    belongs to the complete type \(q_F\), recovering \(F(s)\). 
    
Since \(q_F\) is \(X\)-computable, it follows that \(F\) is \(X'\)-computable. This contradiction proves the case \(m=0\).

The argument for \(m>0\) proceeds similarly, so we only describe the needed modifications. Instead of working in \(S_n^{\mathrm{ar}}(A,\MM)\), we work in \(S_n^{\mathrm{ar}}(A,\MM)^{(\geq m)}\). The latter remains arithmetical due to Lemma~\ref{lem:finite-cb-stages-arith}. Notice that we do not need \(q_F\) itself to lie in the derivative
\(S_n^{\mathrm{ar}}(A,\MM)^{(\geq m)}\): completeness of \(q_F\) is
enough to decide, at each stage, which of the two contradictory child
formulas belongs to \(q_F\), and hence to recover \(F\).
\end{proof}

For the remainder of this section, assume in addition that \(T\) has no
Vaughtian pairs. The next two results are deliberately formulated over a
fixed ambient diagram. Thus, when parameters are named, they are named
inside one fixed arithmetically decidable or arithmetically extendible
ambient model; we do not compare isolated types over the same abstract
set of parameters as it may sit in unrelated extensions.

\begin{lemma}[Isolated arithmetic types over a fixed base are realized]
\label{lem:arith-isolated-realized}
Let \(\AAA\models T\) be arithmetically decidable, let \(A^\ast\) be
the domain of \(\AAA\), let \(A\subseteq A^\ast\) be arithmetical, and
let \(p(\bar x)\in S_{\bar x}^{\mathrm{ar}}(A,\AAA)\) be isolated by a
formula \(\theta(\bar x,\bar a)\), with \(\bar a\in A^{<\omega}\). Then
\(\AAA\) realizes \(p\). Consequently, every elementary extension
\(\BB\succeq\AAA\) realizes \(p\).
\end{lemma}

\begin{proof}
Since \(\theta(\bar x,\bar a)\in p\) and \(p\) is consistent with
\(\Diag_{\mathrm{el}}(\AAA)\), the sentence
\(\exists\bar x\,\theta(\bar x,\bar a)\) holds in \(\AAA\).
Choose \(\bar b\in (A^\ast)^{|\bar x|}\) realizing \(\theta\) in
\(\AAA\). Then \(\tp^{\AAA}(\bar b/A)\) is arithmetic (since \(\AAA\)
is decidable and \(A\) is arithmetical), contains \(\theta\), and hence
equals \(p\) by isolation. Any elementary extension \(\BB\succeq\AAA\)
realizes \(p\) via the same tuple \(\bar b\).
\end{proof}

\begin{proposition}[Isolated arithmetic one-types have finite trace]
\label{prop:arith-omit-all-new-realizations}
Assume \(T\) is arithmetically stable and has no Vaughtian pairs. Let
\(\BB\models T\) be arithmetically extendible, let \(A\subseteq B\) be
arithmetical, and suppose that \(A\) contains the domain of an
elementary submodel \(\MM_0\preccurlyeq\BB\). Let
\(\varphi(x,\bar a)\), with \(\bar a\in A^{<\omega}\), isolate a
complete arithmetic one-type
\[
p(x)\in S_1^{\mathrm{ar}}(A,\BB).
\]
Then \(\varphi(B,\bar a)\) is finite.
\end{proposition}

\begin{proof}
Suppose not. Passing to an arithmetically decidable elementary extension
via Lemma~\ref{lem:arith-type-space-invariance}, we may assume that
\(\BB\) is arithmetically decidable and that \(\varphi(B,\bar a)\)
is infinite.

We use the standard Vaughtian-pair omitting argument. Work in the
language \(\LL(B)\), naming every element of \(B\), and let
\(T_B:=\Diag_{\mathrm{el}}(\BB)\). Since \(\varphi\) isolates \(p\),
the only arithmetic one-type over \(A\) containing \(\varphi(x,\bar a)\)
is \(p\). Consider the partial type
\[
\Gamma_p(x)
:=
p(x)\cup\{x\neq b:b\in B\}.
\]
This partial type is unsupported over \(T_B\): any formula
\(\chi(x,\bar b)\) consistent with \(T_B\) is realized by some
\(c\in B\), and then \(T_B\cup\{\chi(x,\bar b)\wedge x=c\}\) is
consistent while \(x\neq c\in \Gamma_p\).

Apply the reserved-constant omitting theorem,
Corollary~\ref{cor:reserved-constant-omitting}, over the fixed diagram
\(T_B\), omitting the (singleton) family \(\{\Gamma_p\}\). Thus we add
a fresh constant \(d\), impose the dense requirements \(d\neq b\) for
\(b\in B\), and omit \(\Gamma_p\) in the same Henkin construction. At
an omission stage the current finite condition may include finitely many
inequalities \(d\neq b\); these side conditions are folded into the
consistency test for the next negated formula from \(\Gamma_p\). The
unsupportedness argument above is local, so it applies below such a
finite condition. We obtain an arithmetically decidable elementary
extension \(\BB\preccurlyeq\CC\) omitting \(\Gamma_p\), and the
interpretation of \(d\) lies outside \(B\). Hence \(\CC\neq\BB\).

No new element of \(C\setminus B\) satisfies \(\varphi(x,\bar a)\):
if \(c\in C\setminus B\) did, then \(\tp^{\CC}(c/A)\) would be an
arithmetic type containing \(\varphi(x,\bar a)\), hence equal to \(p\)
by isolation, so \(c\) would realize \(\Gamma_p\), contradicting the
construction. Therefore \(\varphi(C,\bar a)=\varphi(B,\bar a)\).
This common set is infinite and \(\BB\preccurlyeq\CC\) is proper,
yielding a Vaughtian pair---contradiction.
\end{proof}

\begin{proposition}[Positive rank forces infinitude on the base]
\label{prop:cb-positive-infinite-trace}
Assume \(T\) is arithmetically stable and has no Vaughtian pairs. Let
\(\MM_0\models T\) be arithmetically decidable, and suppose
\[
p(\bar x)\in S_{|\bar x|}^{\mathrm{ar}}(M_0,\MM_0)
\]
satisfies \(\CB(p)\geq 1\). Then every formula
\(\varphi(\bar x,\bar a)\in p\), with
\(\bar a\in M_0^{<\omega}\), has infinitely many realizations in
\(M_0\).
\end{proposition}

\begin{proof}
Suppose \(\varphi(\bar x,\bar a)\in p\) and
\[
\varphi(M_0^{|\bar x|},\bar a)
=
\{\bar d_1,\dots,\bar d_k\}
\]
is finite, with the \(\bar d_i\)'s distinct. Since
\(\varphi(\bar x,\bar a)\in p\), this finite set is nonempty: if
\(k=0\), then \(\MM_0\models\neg\exists\bar x\,\varphi(\bar x,\bar a)\),
so no type over \(M_0\) could contain \(\varphi(\bar x,\bar a)\).

Since \(T\) is arithmetically stable, we may fix an elementary extension \(\MM \succeq \MM_0\) which models \(T\) and realizes all arithmetic finite types over \(M_0\). By elementarity,
\[
\MM\models\exists^{=k}\bar x\,\varphi(\bar x,\bar a).
\]
Thus
\(
\varphi(M^{|\bar x|},\bar a)=\{\bar d_1,\dots,\bar d_k\}
\). It follows that every type in \(S_{|\bar x|}^{\mathrm{ar}}(M_0,\MM_0)\) is isolated by a formula of the form
\[
\varphi(\bar x,\bar a)\wedge \bar x=\bar d_i
\]
and hence has Cantor-Bendixson rank \(0\). This
contradicts \(\CB(p)\geq 1\).
\end{proof}

\section{Prime models and strong minimality}
\label{sec:prime-model-strong-minimality}

In this section we derive the main structural consequences of
arithmetical stability together with the absence of Vaughtian pairs.
Throughout we therefore assume that \(T\) is complete, arithmetically
definable, has no finite models, is arithmetically stable, and has no
Vaughtian pairs.

We begin with the standard model-theoretic notions.

\begin{definition}[Atomic and prime models]
\label{def:atomic-prime}
\begin{enumerate}[label=(\roman*)]
\item A model \(\NN\models T\) is \emph{atomic} if every finite tuple
      from \(N\) realizes an isolated complete type over \(\emptyset\).
\item A model \(\PP\models T\) is \emph{prime} if for every model
      \(\MM\models T\), there is an elementary embedding
      \(f:\PP\to\MM\).
\end{enumerate}
\end{definition}

\begin{lemma}
\label{lem:isolation-complete-vs-arithmetic}
Let \(p(\bar x)\in S_n^{\mathrm{ar}}(\emptyset)\) be an arithmetic type.
Then \(p\) is isolated in the space of arithmetic
types over \(\emptyset\) if and only if \(p\) is isolated as an ordinary
complete type over \(\emptyset\).
\end{lemma}

\begin{proof}
If a formula \(\theta(\bar x)\in p\) isolates \(p\) among all complete
types over \(\emptyset\), then it certainly isolates \(p\) inside the
subspace of arithmetic types.

Conversely, suppose \(\theta(\bar x)\in p\) isolates \(p\) among
arithmetic types, but does not isolate a complete type. Then there is a
formula \(\psi(\bar x)\) such that both
\[
T\cup\{\theta(\bar x),\psi(\bar x)\}
\qquad\text{and}\qquad
T\cup\{\theta(\bar x),\neg\psi(\bar x)\}
\]
are consistent. Indeed, take a complete type \(q\neq p\) containing
\(\theta\), and choose \(\psi\in p\setminus q\).

Add a tuple of new constants \(\bar c\). The two theories
\[
T\cup\{\theta(\bar c),\psi(\bar c)\}
\qquad\text{and}\qquad
T\cup\{\theta(\bar c),\neg\psi(\bar c)\}
\]
are arithmetically definable and finitely satisfiable. By
Theorem~\ref{thm:compact}, each has an arithmetically decidable model.
The types of \(\bar c\) in these two models are arithmetic types over
\(\emptyset\), both extend the same formula \(\theta\), and they are
distinct because they decide \(\psi\) differently. This contradicts the
assumption that \(\theta\) isolates \(p\) among arithmetic types.
\end{proof}

\begin{proposition}[Existence of an arithmetically decidable atomic model]
\label{prop:atomic-model}
Assume \(T\) is arithmetically stable. Then there exists an
arithmetically decidable model \(\NN\models T\) which is atomic.
\end{proposition}

\begin{proof}
Fix an arithmetically decidable model \(\MM \models T\) (Theorem~\ref{thm:compact}). By arithmetic stability and Lemma~\ref{lem:finite-cb-stages-arith}, we may fix an arithmetical code for the types
\[
\mathcal{R} = \{p \in S_{<\omega}^{\mathrm{ar}}(\emptyset,\MM): \mathrm{CB}(p) \geq 1\}.
\]
Each type in \(\mathcal R\) is unsupported over \(T\): by
Lemma~\ref{lem:isolation-complete-vs-arithmetic}, nonisolation among
arithmetic types is the same as ordinary nonisolation, so no consistent
formula can imply every formula in such a type.

Applying Theorem~\ref{thm:arith-omit-countable-family} to \(T\) and
\(\mathcal R\), we obtain an arithmetically decidable model
\(\NN\models T\) omitting every nonisolated arithmetic type over
\(\emptyset\). Every tuple from \(N\) realizes an arithmetic type that
is not omitted, hence is isolated among arithmetic types, hence is
isolated as a complete type by
Lemma~\ref{lem:isolation-complete-vs-arithmetic}. Thus \(\NN\) is
atomic.
\end{proof}

\begin{proposition}[Atomic models are prime]
\label{prop:atomic-implies-prime}
Let \(\NN\models T\) be atomic. Then \(\NN\) is prime.

Moreover, if \(\NN\) and \(\MM\) are arithmetically decidable, then
\(\NN\) admits an elementary embedding into \(\MM\) whose image is an
arithmetical subset of \(M\).
\end{proposition}

\begin{proof}
The first assertion is the standard back-and-forth over isolated types:
enumerate \(N=\{a_0,a_1,\dots\}\), and at stage \(s\) choose \(b_s\in M\)
such that \(\tp^\MM(b_0,\dots,b_s)=\tp^\NN(a_0,\dots,a_s)\), using the
isolating formula for \(\tp^\NN(a_0,\dots,a_s)\) to find a realization
in \(\MM\). The map \(f(a_s):=b_s\) is elementary.

For the moreover clause, assume both \(\NN\) and \(\MM\) are
arithmetically decidable. The construction above can be made effective:
at stage \(s\), arithmetically search for the least formula code
\(\ulcorner\theta_s\urcorner\) such that
\(\NN\models\theta_s(a_0,\dots,a_s)\) and \(\theta_s\) isolates
\(\tp^\NN(a_0,\dots,a_s)\). The isolation check is arithmetical because
provability from the arithmetically definable theory \(T\) is
arithmetical. Then choose \(b_s\) as the least element of \(M\)
satisfying \(\theta_s(b_0,\dots,b_{s-1},b_s)\) in \(\MM\). Both
searches are arithmetical relative to a sufficiently high finite jump
computing the satisfaction relations of \(\NN\) and \(\MM\). Hence the
graph of \(f\) and its image \(f[N]\) are arithmetical.
\end{proof}

\begin{corollary}[Existence of an arithmetically decidable prime model]
\label{cor:prime-model}
Assume \(T\) is arithmetically stable. Then there exists an
arithmetically decidable model \(\PP\models T\) which is prime.
\end{corollary}

\begin{proof}
Combine
Proposition~\ref{prop:atomic-model} with
Proposition~\ref{prop:atomic-implies-prime}.
\end{proof}
\begin{lemma}[Effective embeddings of arithmetically decidable prime models]
\label{lem:arith-prime-effective-atomic}
Let \(\PP\models T\) be an arithmetically decidable prime model. Then
\(\PP\) is atomic, and for any arithmetically decidable \(\MM\models T\)
there is an elementary embedding \(f:\PP\to\MM\) with arithmetical graph
and image.
\end{lemma}

\begin{proof}
The standard prime-implies-atomic argument applies: if some tuple
realized a nonisolated type, the omitting types theorem would give a
model omitting that type, contradicting primeness. The effective
embedding then follows from
Proposition~\ref{prop:atomic-implies-prime}.
\end{proof}

\begin{lemma}[Positive-rank refinement inside an infinite trace]
\label{lem:infinite-trace-positive-rank}
Assume \(T\) is arithmetically stable and has no Vaughtian pairs. Let
\(\AAA\models T\) be arithmetically decidable, write \(A\) for its
domain, and let
\(\eta(x,\bar a)\) be a formula with parameters from \(A\). If
\(\eta(A,\bar a)\) is infinite, then some arithmetic type
\(p(x)\in S_1^{\mathrm{ar}}(A,\AAA)\) contains \(\eta(x,\bar a)\) and
satisfies \(\CB(p)\geq 1\).
\end{lemma}

\begin{proof}
Because \(\eta(A,\bar a)\) is infinite, the partial type
\(\Diag_{\mathrm{el}}(\AAA)\cup\{\eta(c,\bar a)\}\cup\{c\neq a:a\in A\}\)
is finitely satisfiable. It is arithmetically definable, so by
Lemma~\ref{lem:base-preserving-coded-compactness} it has an
arithmetically decidable model extending \(\AAA\). Let
\(p(x)\in S_1^{\mathrm{ar}}(A,\AAA)\) be the type over \(A\) of the
interpretation of \(c\). Then \(p\)
contains \(\eta(x,\bar a)\) and is nonalgebraic over \(A\).

If \(\CB(p)=0\), choose an isolating formula
\(\theta(x,\bar a_0)\in p\). By
Lemma~\ref{lem:arith-isolated-realized}, \(\AAA\) realizes this isolated
type, so some element of \(A\) satisfies \(\theta\). By isolation, that
element realizes \(p\), contradicting \(p\ni x\neq a\) for every
\(a\in A\). Hence \(\CB(p)\geq1\).
\end{proof}

\begin{lemma}[Existence of a minimal formula]
\label{lem:minimal-exists-ar}
Assume \(T\) is arithmetically stable and has no Vaughtian pairs. Let
\(\AAA\models T\) be arithmetically decidable. Then there is a formula
\(\varphi(x,\bar a)\), with parameters \(\bar a\) from \(A\), such that
\(\varphi(A,\bar a)\) is infinite and, for every formula
\(\psi(x,\bar c)\) with parameters from \(A\), either
\[
\varphi(A,\bar a)\cap \psi(A,\bar c)
\]
or
\[
\varphi(A,\bar a)\setminus \psi(A,\bar c)
\]
is finite.
\end{lemma}

\begin{proof}
Since \(A\) is infinite, compactness produces an arithmetically decidable
elementary extension of \(\AAA\) realizing the nonalgebraic partial type
\(\{x\neq a:a\in A\}\). Let
\(p(x)\in S_1^{\mathrm{ar}}(A,\AAA)\) be the arithmetic type of such a
realization.

We first show that \(\CB(p)\geq 1\). If \(\CB(p)=0\), then some formula
\(\theta(x,\bar a)\in p\) isolates \(p\) among arithmetic types over
\(A\). Since \(p\) contains \(x\neq a\) for every \(a\in A\), every
realization of \(\theta\) in any elementary extension of \(\AAA\) lies
outside \(A\). But Lemma~\ref{lem:arith-isolated-realized} says
\(\AAA\) itself realizes the isolated type, giving a realization in
\(A\)---contradiction.

By Lemma~\ref{lem:cb-rank-dense} applied at rank \(1\) to the formula
\(x=x\), choose \(q(x)\in S_1^{\mathrm{ar}}(A,\AAA)\) with
\(\CB(q)=1\). By definition, there is a formula
\(\varphi(x,\bar a)\in q\) isolating \(q\) among the rank-at-least-one
arithmetic types. Since \(\CB(q)\geq 1\),
Proposition~\ref{prop:cb-positive-infinite-trace} gives
\(\varphi(A,\bar a)\) infinite.

We claim that \(\varphi(x,\bar a)\) is minimal over \(A\). Suppose not:
some formula \(\psi(x,\bar c)\) splits \(\varphi(A,\bar a)\) into two
infinite pieces. One of \(\varphi\wedge\psi\) or
\(\varphi\wedge\neg\psi\) belongs to \(q\); say
\(\varphi\wedge\psi\in q\). Put
\(\chi(x):=\varphi(x,\bar a)\wedge\neg\psi(x,\bar c)\).
Since \(\chi(A)\) is infinite,
Lemma~\ref{lem:infinite-trace-positive-rank} gives an arithmetic type
\(r(x)\) containing \(\chi\) with \(\CB(r)\geq 1\). Then \(r\neq q\)
(they disagree on \(\psi\)), yet \(\varphi\in r\) and
\(\CB(r)\geq 1\), contradicting the isolation of \(q\) among
rank-at-least-one types containing \(\varphi\).
\end{proof}

\begin{lemma}[No Vaughtian pairs eliminate \(\exists^\infty\)]
\label{lem:no-vp-elim-exists-infty}
Assume \(T\) is complete, has no finite models, and has no Vaughtian
pairs. Then for every formula \(\theta(x,\bar y)\) there exists
\(n_\theta<\omega\) such that for every model \(\MM\models T\) and every
tuple \(\bar b\in M^{|\bar y|}\), if \(\theta(M,\bar b)\) is finite,
then
\[
|\theta(M,\bar b)|\leq n_\theta.
\]
Equivalently, \(T\) eliminates \(\exists^\infty\).
\end{lemma}

\begin{proof}
Suppose not. Then there is a formula \(\theta(x,\bar y)\) such that for
every \(n<\omega\) there are \(\MM_n\models T\) and
\(\bar b_n\in M_n^{|\bar y|}\) with
\(n<|\theta(M_n,\bar b_n)|<\omega\).

Expand the language by a predicate \(P\) and constants \(\bar c\), and
let \(\Sigma\) assert that \(T\) holds, \(P\) names an elementary
submodel (via the Tarski--Vaught scheme), \(\bar c\in P\), every
realization of \(\theta(x,\bar c)\) lies in \(P\), \(P\) is proper, and
\(\theta(x,\bar c)\) is infinite.

For finite satisfiability, given \(N\), choose \(\MM_n,\bar b_n\) with
\(|\theta(M_n,\bar b_n)|>N\). Take an elementary submodel
\(\NN_n\preccurlyeq\MM_n\) containing \(\bar b_n\). Since
\(\MM_n\models\exists^{=q}x\,\theta(x,\bar b_n)\) for some finite
\(q>N\), elementarity gives
\(\theta(N_n,\bar b_n)=\theta(M_n,\bar b_n)\), so all realizations
already lie in \(\NN_n\). Passing to a proper elementary extension of
\(\MM_n\) (by compactness) gives a model of every finite fragment of
\(\Sigma\).

By compactness, \(\Sigma\) has a model, yielding a Vaughtian pair---contradiction.
\end{proof}

\begin{lemma}[Minimal formulas are strongly minimal]
\label{lem:minimal-implies-strongly-minimal-ar}
Assume \(T\) is arithmetically stable and has no Vaughtian pairs. Let
\(\AAA\models T\) be arithmetically decidable, and let
\(\varphi(x,\bar a)\) be minimal over \(A\). Then
\(\varphi(x,\bar a)\) is strongly minimal over \(A\): in every
elementary extension
\[
\AAA\preccurlyeq \MM\models T,
\]
every definable subset of \(\varphi(M,\bar a)\) is finite or cofinite.
\end{lemma}

\begin{proof}
By Lemma~\ref{lem:no-vp-elim-exists-infty}, the theory \(T\) eliminates
\(\exists^\infty\).

Fix a formula \(\psi(x,\bar y)\). Let
\[
n:=\max(n_{\varphi\wedge\psi},\,n_{\varphi\wedge\neg\psi}),
\]
where the bounds are given by elimination of \(\exists^\infty\).

Suppose \(\varphi(x,\bar a)\) is not strongly minimal. Then there are
\(\AAA\preccurlyeq \MM\models T\) and \(\bar b\in M^{|\bar y|}\) such
that both
\[
\varphi(M,\bar a)\cap\psi(M,\bar b)
\qquad\text{and}\qquad
\varphi(M,\bar a)\setminus\psi(M,\bar b)
\]
are infinite. In particular,
\(\MM\) satisfies the first-order sentence asserting the
existence of a tuple \(\bar y\) for which both
\(\exists^{>n}x\,(\varphi\wedge\psi)\) and
\(\exists^{>n}x\,(\varphi\wedge\neg\psi)\) hold.
By elementarity, some \(\bar b_0\in A^{|\bar y|}\) satisfies
the same two counting conditions in \(\AAA\). By the choice of \(n\),
neither
\(\varphi(A,\bar a)\cap\psi(A,\bar b_0)\) nor
\(\varphi(A,\bar a)\setminus\psi(A,\bar b_0)\)
can be finite, contradicting minimality of \(\varphi\) over \(A\).
\end{proof}

\begin{proposition}[Existence of a strongly minimal formula]
\label{prop:strongly-minimal-formula}
Assume \(T\) is arithmetically stable and has no Vaughtian pairs. Let
\(\PP\models T\) be an arithmetically decidable prime model. Then there
is a formula \(\varphi(x,\bar p)\), with parameters \(\bar p\) from
\(P\), such that \(\varphi(x,\bar p)\) is strongly minimal over \(P\).
\end{proposition}

\begin{proof}
By Lemma~\ref{lem:minimal-exists-ar}, applied to the arithmetically
decidable prime model \(\PP\), there is a one-variable minimal formula
\(\varphi(x,\bar p)\) over \(P\). The no-Vaughtian-pairs hypothesis then
allows Lemma~\ref{lem:minimal-implies-strongly-minimal-ar} to convert
minimality into strong minimality.
\end{proof}

Fix an arithmetically decidable prime model \(\PP\models T\) and a
strongly minimal formula \(\varphi(x,\bar p)\) over \(P\). For any
elementary extension
\[
\PP\preccurlyeq \MM\models T,
\]
write
\[
C:=\varphi(M,\bar p).
\]
We will show that
\[
M=\acl(P\cup C).
\]

\begin{lemma}
\label{lem:finite-support-acl-from-phi}
If \(a\in \acl^{\MM}(P\cup C)\), then there is a finite tuple
\(\bar c\) from \(C\) such that
\(a\in \acl^{\MM}(P\cup \bar c)\).
\end{lemma}

\begin{proof}
Any witness to algebraicity involves only finitely many parameters from
\(C\); since \(P\) is already part of the base, these finitely many
elements of \(C\) suffice.
\end{proof}

\begin{lemma}[Rank-zero witnesses are algebraic]
\label{lem:rank-zero-witnesses-algebraic}
Assume \(T\) is arithmetically stable and has no Vaughtian pairs. Let
\(\MM_0\models T\) be arithmetically decidable, let
\(\BB\models T\) be arithmetically extendible with
\(\MM_0\preccurlyeq \BB\),
and let \(A\subseteq B\) be arithmetical with \(M_0\subseteq A\).
If \(\eta(x,\bar a)\), with \(\bar a\in A^{<\omega}\),
isolates a complete arithmetic type
\(r(x)\in S_1^{\mathrm{ar}}(A,\BB)\),
then \(\eta(B,\bar a)\) is finite. In particular, every realization of
\(\eta\) in \(\BB\) belongs to \(\acl^{\BB}(A)\).
\end{lemma}

\begin{proof}
This is exactly
Proposition~\ref{prop:arith-omit-all-new-realizations}, applied to the
parameter set \(A\), the isolated type \(r\), and the model \(\BB\).
\end{proof}

\begin{lemma}
\label{lem:tv-step-from-phi-finite-tuple}
Assume \(T\) is arithmetically stable and has no Vaughtian pairs. Let
\(\PP\) and \(\MM\) be arithmetically decidable models of \(T\) with
\(\PP\preccurlyeq \MM\), and write
\(C:=\varphi(M,\bar p)\).
Let \(\bar c\) be a finite tuple from \(C\), let
\(a\in \acl^{\MM}(P\cup \bar c)\), and suppose
\(\MM\models \exists x\,\psi(x,a)\).
Then some
\(d\in \acl^{\MM}(P\cup \bar c)\)
satisfies
\(\MM\models \psi(d,a)\).
\end{lemma}

\begin{proof}
Put \(A:=P\cup \bar c\cup\{a\}\), which is arithmetical since \(\PP\)
is decidable and \(\bar c,a\) are finite. Since
\(\MM\models \exists x\,\psi(x,a)\) and \(\MM\) is arithmetically
decidable, the basic clopen set
\([\psi(x,a)]_{\mathrm{ar}}\subseteq S_1^{\mathrm{ar}}(A,\MM)\)
is nonempty.

By Lemma~\ref{lem:cb-rank-dense} with \(m=0\), choose
\(r(x)\in S_1^{\mathrm{ar}}(A,\MM)\) with
\(\psi(x,a)\in r\) and \(\CB(r)=0\). Let
\(\eta(x,\bar e)\in r\) isolate \(r\); replacing \(\eta\) by
\(\eta\wedge \psi(x,a)\), we may assume \(\eta\) implies \(\psi(x,a)\).
Since all parameters lie in \(M\), the sentence
\(\exists x\,\eta(x,\bar e,a)\) holds in \(\MM\); choose \(d\in M\)
realizing \(\eta\).
Then \(\MM\models \psi(d,a)\).

By Lemma~\ref{lem:rank-zero-witnesses-algebraic} (with \(\MM_0=\PP\)
and parameter set \(A\)), \(\eta(M,\bar e,a)\) is finite, so
\(d\in \acl^{\MM}(A)\). Since
\(a\in\acl^{\MM}(P\cup\bar c)\), transitivity gives
\(\acl^{\MM}(A)=\acl^{\MM}(P\cup\bar c)\), and thus
\(d\in\acl^{\MM}(P\cup\bar c)\).
\end{proof}

\begin{lemma}[The algebraic closure of \(P\cup C\) is elementary]
\label{lem:acl-of-prime-and-phi-is-elementary}
Assume \(T\) is arithmetically stable and has no Vaughtian pairs. Let
\(\PP\) and \(\MM\) be arithmetically decidable models of \(T\) with
\(\PP\preccurlyeq \MM\), and write
\(C:=\varphi(M,\bar p)\).
Let \(N:=\acl^{\MM}(P\cup C)\).
Then
\(\MM\restr_N\preccurlyeq \MM\).
\end{lemma}

\begin{proof}
Apply the Tarski--Vaught test. Let \(a\in N\), and suppose
\(\MM\models \exists x\,\psi(x,a)\).
(Here \(a\) may be a finite tuple from \(N\).)
By Lemma~\ref{lem:finite-support-acl-from-phi}, choose a finite
\(\bar c\subseteq C\) with \(a\in \acl^{\MM}(P\cup\bar c)\).
Then Lemma~\ref{lem:tv-step-from-phi-finite-tuple} gives
\(d\in \acl^{\MM}(P\cup\bar c)\subseteq N\)
with \(\MM\models\psi(d,a)\).
\end{proof}

We emphasize that the model \(\MM\) in the following proposition is not
assumed to be arithmetically decidable, or even arithmetically
extendible. This will be important in the later comparison with the
classical Baldwin--Lachlan theorem.

\begin{proposition}[Generation by the strongly minimal set]
\label{prop:model-generated-by-strongly-minimal-set}
Assume \(T\) is arithmetically stable and has no Vaughtian pairs. Then
every elementary extension
\[
\PP\preccurlyeq \MM\models T
\]
satisfies
\[
M=\acl(P\cup\varphi(M,\bar p)).
\]
\end{proposition}

\begin{proof}
We split the argument into two cases.

\medskip
\noindent\textbf{Case 1: \(\MM\) is arithmetically decidable.}
Let
\(C:=\varphi(M,\bar p)\) and \(N:=\acl^{\MM}(P\cup C)\).
By Lemma~\ref{lem:acl-of-prime-and-phi-is-elementary},
\(\MM\restr_N\preccurlyeq \MM\).
Since \(C\subseteq N\), we have
\(\varphi(N,\bar p)=\varphi(M,\bar p)\).
If \(N\neq M\), then
\((\MM\restr_N,\MM)\)
is a Vaughtian pair witnessed by \(\varphi(x,\bar p)\), contradicting
the hypothesis. Hence \(N=M\).

\medskip
\noindent\textbf{Case 2: \(\MM\) is arbitrary.}
Put
\(C:=\varphi(M,\bar p)\) and \(N:=\acl^{\MM}(P\cup C)\).
Suppose toward a contradiction that \(N\neq M\). Choose
\(b\in M\setminus N\).

Work in the language \(\LL(P)\cup\{d\}\), where the elements of \(P\)
are named by constants and \(d\) is a fresh constant. Let \(\Gamma\) be
the set of all sentences
\[
\forall y_1\cdots\forall y_n\,
\Bigl[
\bigwedge_{i=1}^n \varphi(y_i,\bar p)\wedge
\exists^{\leq k}x\,\theta(x,\bar y)
\to
\neg\theta(d,\bar y)
\Bigr],
\]
as \(\theta(x,\bar y)\) ranges over all \(\LL(P)\)-formulas, and as
\(n\) and \(k\) vary. Informally, \(\Gamma\) asserts that \(d\) avoids
every finite definable set over \(P\) with parameters from
\(\varphi(x,\bar p)\).
The scheme \(\Gamma\) is arithmetically definable.

Let
\(\Sigma:=\Diag_{\mathrm{el}}^{\LL(P)}(\PP)\cup\Gamma\).
Then \(\Sigma\) is satisfiable: expand the given model \(\MM\) by
interpreting \(d\) as \(b\). For each instance of \(\Gamma\) with
\(\bar c\in C^n\), if
\(\theta(M,\bar c)\) is finite then it is contained in
\(\acl^{\MM}(P\cup\bar c)\subseteq N\), and \(b\notin N\) gives the
required inequality.

By Lemma~\ref{lem:base-preserving-coded-compactness},
\(\Sigma\) has an arithmetically decidable model \(\MM'\) with
\(\PP\preccurlyeq\MM'\restr_\LL\). Let
\(d':=d^{\MM'}\) and \(C':=\varphi(M'\restr_\LL,\bar p)\).
By \(\Gamma\), \(d'\notin \acl^{\MM'}(P\cup C')\).
But \(\MM'\restr_\LL\) is arithmetically decidable, so Case~1 gives
\(M'=\acl^{\MM'}(P\cup C')\ni d'\), a contradiction.
\end{proof}

\section{From strongly minimal control to arithmetic degree}
\label{sec:sm-control-arith-degree}

The missing ingredient for the main categoricity theorem is
degree-theoretic: we must extend
Lemma~\ref{lem:sm-definable-full-degree} from the strongly minimal
setting to the present one, where the theory need not be strongly
minimal but is controlled by a strongly minimal set over a prime base.
Concretely, we show that an infinite definable set not already trapped in
the algebraic closure of its parameters recovers the full ambient
arithmetic degree.

Throughout this section, we assume the following hypothesis.

\medskip

\noindent
\textbf{Standing hypothesis.}
There exist an arithmetically decidable prime model \(\PP\models T\) and
a formula \(\psi(x,\bar p)\) with parameters from \(P\) such that
\(\psi(x,\bar p)\) is strongly minimal over \(P\), and every elementary
extension
\[
\PP\preccurlyeq \NN\models T
\]
satisfies
\[
N=\acl(P\cup \psi(N,\bar p)).
\]

\medskip

In Section~\ref{sec:proof-main-theorems}, this hypothesis will be
obtained from the preceding structural analysis and identified with the
strongly minimal control clause in the arithmetic Baldwin--Lachlan
theorem.

The only use of imaginaries in this section is the classical restricted
one needed for a definable family of subsets of a strongly minimal set:
we use the quotient by the relation of defining the same set, and then
replace the resulting imaginary by a finite tuple from the strongly
minimal set. Appendix~\ref{subsec:arith-imaginaries} explains the
corresponding arithmetical coding when the ambient structure is
arithmetically decidable, but the finite-fiber coding lemma below is a
purely model-theoretic statement.

We first transfer the strongly minimal control furnished by the
standing hypothesis from the prime model to an arbitrary
arithmetically decidable base.

\begin{lemma}
\label{lem:sm-over-arith-model}
Let \(\MM_0\models T\) be arithmetically decidable. Then there is a
formula \(\varphi(x,\bar m_0)\), with parameters from \(M_0\), such
that:
\begin{enumerate}[label=(\roman*)]
\item \(\varphi(x,\bar m_0)\) is strongly minimal over \(M_0\);
\item whenever \(\MM_0\preccurlyeq \MM\models T\) and
      \(C:=\varphi(M,\bar m_0)\), one has
      \(M=\acl(M_0\cup C)\).
\end{enumerate}
Consequently, if \(B\subseteq C\) is a basis of \(C\) over
\(C_0:=C\cap M_0=\varphi(M_0,\bar m_0)\), then
\(M=\acl(M_0\cup B)\).
\end{lemma}

\begin{proof}
By the standing hypothesis, there exist an arithmetically decidable
prime model \(\PP\models T\) and a formula \(\psi(x,\bar p)\) with
parameters from \(P\) such that \(\psi(x,\bar p)\) is strongly minimal
over \(P\), and for every elementary extension
\(\PP\preccurlyeq \NN\models T\) one has
\(N=\acl(P\cup \psi(N,\bar p))\).

Since \(\PP\) is prime and both \(\PP\) and \(\MM_0\) are
arithmetically decidable, Lemma~\ref{lem:arith-prime-effective-atomic}
yields an elementary embedding \(f:\PP\to \MM_0\)
whose image is an arithmetical subset of \(M_0\). Until the explicit
identification is made, the base copy is \(f[P]\subseteq M_0\); after
recoding along this arithmetical elementary embedding, we identify
\(\PP\) with its image and assume
\(\PP\preccurlyeq \MM_0\). The formula and all parameters are
transported along \(f\).
Let \(\bar m_0:=f(\bar p)\), and let
\(\varphi(x,\bar m_0)\) be the translate of \(\psi(x,\bar p)\).

Now let \(\MM_0\preccurlyeq \MM\models T\), and put
\(C:=\varphi(M,\bar m_0)\).
Since \(\PP\preccurlyeq \MM_0\preccurlyeq \MM\), the set \(C\) is the
corresponding translate of \(\psi(M,\bar p)\). Because
\(\psi(x,\bar p)\) is strongly minimal over \(P\), the translate
\(\varphi(x,\bar m_0)\) is strongly minimal over \(M_0\): strong
minimality is preserved by elementary embeddings, and the assertion is
about all definable subsets of the translated set in elementary
extensions, not merely about subsets definable over the original prime
copy.

Applying the standing hypothesis to \(\PP\preccurlyeq \MM\), we obtain
\(M=\acl(P\cup C)\). Since \(P\subseteq M_0\), this gives
\(M=\acl(M_0\cup C)\).

Finally, let \(B\subseteq C\) be a basis of \(C\) over
\(C_0:=C\cap M_0\).
Since \(\varphi(x,\bar m_0)\) is strongly minimal over \(M_0\), the map
\(A\mapsto \acl(M_0\cup A)\cap C\)
is a pregeometry on \(C\). The phrase ``basis over \(C_0\)'' means that
\(B\) is independent over \(C_0\) and spans \(C\) in this induced
pregeometry. Hence
\(C\subseteq \acl(M_0\cup B)\),
and therefore
\(M=\acl(M_0\cup C)=\acl(M_0\cup B)\).
\end{proof}

\begin{fact}[Finite-dimensional pieces in an almost strongly minimal model]
\label{fact:finite-dimensional-almost-sm}
Let \(\MM_0\preccurlyeq\MM\models T\), let
\(\varphi(x,\bar m_0)\) be strongly minimal over \(M_0\), and put
\(C:=\varphi(M,\bar m_0)\). If \(M=\acl(M_0\cup C)\), then for every
\(B\subseteq C\),
\[
\acl^{\MM}(M_0\cup B)\preccurlyeq \MM.
\]
Moreover,
\[
\acl^{\MM}(M_0\cup B)\cap C
=
\cl_C(C_0\cup B),
\qquad
C_0:=C\cap M_0,
\]
where \(\cl_C(X):=\acl^{\MM}(M_0\cup X)\cap C\) is the pregeometry
induced on \(C\).
\end{fact}

\begin{proof}
For finite \(B\), equivalently for finite dimension over \(C_0\), this
is the standard finite-dimensional submodel lemma for
almost-strongly-minimal theories. It follows from stationarity of the
generic type of a strongly minimal set over algebraically closed bases;
see, for example, Marker~\cite[Chapter~6]{marker2002}.

For arbitrary \(B\), reduce to the finite-dimensional case by
Tarski--Vaught. If a finite parameter tuple
\(\bar a\) lies in \(\acl(M_0\cup B)\), then finite character gives a
finite \(B_0\subseteq B\) with
\(\bar a\in\acl(M_0\cup B_0)\). The finite-dimensional case gives
\(\acl(M_0\cup B_0)\preccurlyeq M\). Hence any existential formula over
\(\bar a\) realized in \(M\) already has a witness in
\(\acl(M_0\cup B_0)\subseteq\acl(M_0\cup B)\). This proves
\(\acl(M_0\cup B)\preccurlyeq M\). The displayed identity follows from
the fixed convention
\(\cl_C(X)=\acl^{\MM}(M_0\cup X)\cap C\): finite character again reduces
both inclusions to finite subsets of \(B\).
\end{proof}

\begin{lemma}[Finite-dimensional control submodels]
\label{lem:finite-dimensional-control-submodel}
Let \(\MM_0\models T\) be arithmetically decidable, let
\(\varphi(x,\bar m_0)\) be as in Lemma~\ref{lem:sm-over-arith-model},
and let \(\MM_0\preccurlyeq\MM\models T\). Put
\(C:=\varphi(M,\bar m_0)\). If \(\bar c\) is a finite tuple from \(C\),
then
\[
\acl^{\MM}(M_0\cup\bar c)\preccurlyeq \MM.
\]
If, in addition, \(\MM\preccurlyeq\widehat{\MM}\) with
\(\widehat{\MM}\) arithmetically decidable, then the submodel \(\acl^{\MM}(M_0\cup\bar c)\) is
arithmetically decidable, uniformly in the code of
\(\bar c\).
\end{lemma}

\begin{proof}
By Lemma~\ref{lem:sm-over-arith-model}, if
\(C:=\varphi(M,\bar m_0)\), then \(M=\acl(M_0\cup C)\). Hence
Fact~\ref{fact:finite-dimensional-almost-sm}, applied with
\(B=\{\text{entries of }\bar c\}\), gives
\(\acl^{\MM}(M_0\cup\bar c)\preccurlyeq\MM\).

For the effective statement, use the fixed arithmetically decidable
extension \(\widehat{\MM}\). First, the domain \(\acl^{\MM}(M_0\cup\bar c)\) is arithmetical,
uniformly in the finite code of \(\bar c\), by
Lemma~\ref{lem:ambient-arithmetical-inequalities}\textup{(iii)}. Second, to show that the elementary diagram of \(\acl^{\MM}(M_0\cup\bar c)\) is arithmetical, note that 
\[
\acl^{\MM}(M_0\cup\bar c)\preccurlyeq \MM \preccurlyeq \widehat{\MM},
\]
so
satisfaction in \(\acl(M_0\cup\bar c)\) is obtained by restricting the
arithmetical satisfaction relation of \(\widehat{\MM}\) to the arithmetical domain just described.
\end{proof}

\begin{fact}[Restricted weak coding of imaginaries on a strongly minimal set]
\label{fact:restricted-weak-ei-sm}
Let \(D\) be a strongly minimal set, in its induced structure over a
parameter set \(A\), and let \(E\) be an \(A\)-definable equivalence
relation on a definable subset of \(D^n\). For every imaginary
\(e\in D^n/E\), there is a finite tuple \(\bar d\in D^{<\omega}\) such
that
\[
e\in\dcl^{\eq}(A,\bar d)
\qquad\text{and}\qquad
\bar d\in\acl^{\eq}(A,e).
\]
This is the standard weak elimination of imaginaries for strongly
minimal structures; see Marker~\cite[Chapter~6]{marker2002}. We use
only this finite-tuple coding consequence.
\end{fact}

\begin{lemma}[Finite-fiber coding on a strongly minimal set]
\label{lem:coding-finite-fibers}
Let \(A\subseteq M\), let \(D=\varphi(M,\bar a)\) be strongly minimal
over \(A\), and let \(\rho(\bar x,\bar y)\) be a formula over \(A\),
where \(\bar x\) is a finite tuple of home-sort variables and \(\bar y\)
ranges in \(D^n\). Assume that for some \(N<\omega\),
\[
|\rho(M^{|\bar x|},\bar y)|\leq N
\qquad\text{for all }\bar y\in D^n.
\]
Fix \(\bar b\in D^n\), and write
\[
S(\bar b):=\rho(M^{|\bar x|},\bar b).
\]
Then there exist \(m<\omega\), a tuple \(\bar d\in D^m\), and formulas
\[
\gamma(\bar z),\qquad \eta(\bar x,\bar z)
\]
over \(A\), with \(\bar z\) ranging in \(D^m\), such that:
\begin{enumerate}[label=(\roman*)]
\item \(\gamma(\bar d)\) and
      \(\eta(M^{|\bar x|},\bar d)=S(\bar b)\);
\item for all \(\bar z\in D^m\), if \(\gamma(\bar z)\), then
      \[
      |\eta(M^{|\bar x|},\bar z)|\leq N;
      \]
\item for every \(\bar z_0\in D^m\) satisfying \(\gamma(\bar z_0)\),
      there are only finitely many \(\bar z\in D^m\) such that
      \(\gamma(\bar z)\) and
      \[
      \eta(M^{|\bar x|},\bar z)=\eta(M^{|\bar x|},\bar z_0).
      \]
\end{enumerate}
In other words, the finite set \(S(\bar b)\) can be coded by a tuple
from a definable domain inside the strongly minimal set in such a way
that only finitely many genuine codes define the same set.
\end{lemma}

\begin{proof}
Define an \(A\)-definable equivalence relation \(E\) on \(D^n\) by
\[
E(\bar y,\bar y')
\iff
\forall \bar x\,(\rho(\bar x,\bar y)\leftrightarrow \rho(\bar x,\bar y')).
\]
Write
\[
e:=[\bar b]_E.
\]
By Fact~\ref{fact:restricted-weak-ei-sm}, applied in the induced
strongly minimal structure on \(D\), there is a finite tuple
\(\bar d\in D^m\) such that
\[
e\in\dcl^{\eq}(A,\bar d)
\qquad\text{and}\qquad
\bar d\in\acl^{\eq}(A,e).
\]
Choose formulas \(\delta(u,\bar z)\) and \(\alpha(\bar z,u)\) in the
represented quotient \(D^n/E\) witnessing these properties, with
\(\delta\) functional in \(u\) on its domain and
\(\alpha(\bar z,u)\) having uniformly finite fibers in \(\bar z\).
Pull back \(\delta\) and \(\alpha\) to the home sort via representatives
of \(E\)-classes; \(E\)-invariance makes the resulting
\(\LL(A)\)-conditions well-defined. Concretely, let
\[
\gamma(\bar z)
\;\Longleftrightarrow\;
\exists \bar y\in D^n\,
\bigl(\widetilde\delta(\bar y,\bar z)
\wedge \widetilde\alpha(\bar z,\bar y)\bigr),
\]
and
\[
\eta(\bar x,\bar z)\;\Longleftrightarrow\;
\exists \bar y\in D^n\,
\bigl(\widetilde\delta(\bar y,\bar z)
\wedge \widetilde\alpha(\bar z,\bar y)
\wedge \rho(\bar x,\bar y)\bigr),
\]
where \(\widetilde\delta\) and \(\widetilde\alpha\) denote the
home-sort pullbacks.

Since \(\delta(e,\bar d)\) and \(\alpha(\bar d,e)\) hold, \(\gamma(\bar
d)\) holds and
\[
\eta(M^{|\bar x|},\bar d)=S(\bar b),
\]
proving (i). For any \(\bar z\in D^m\), the formula
\(\delta(u,\bar z)\) determines at most one \(E\)-class \(u\). If
\(\gamma(\bar z)\) holds, then \(\eta(M^{|\bar x|},\bar z)\) is exactly the
corresponding set \(\rho(M^{|\bar x|},\bar y)\), for any representative
\(\bar y\) of that class. Hence
\[
|\eta(M^{|\bar x|},\bar z)|\leq N,
\]
proving (ii).

For (iii), fix \(\bar z_0\in D^m\) satisfying \(\gamma(\bar z_0)\). If
\(\gamma(\bar z)\) and
\[
\eta(M^{|\bar x|},\bar z)=\eta(M^{|\bar x|},\bar z_0),
\]
then the unique \(E\)-class coded by \(\bar z\) defines the same
\(\rho\)-set as the unique \(E\)-class coded by \(\bar z_0\). By the
definition of \(E\), these two classes are equal. Thus both
\(\bar z\) and \(\bar z_0\) lie in the finite fiber of \(\alpha\) over
that class. Equivalently, in the home sort they lie in one finite
fiber of the representative relation \(\widetilde\alpha\). Hence there
are only finitely many such \(\bar z\).
\end{proof}

For \(X\subseteq M^r\), write
\[
\operatorname{coord}(X):=\{m\in M:m\text{ occurs as a coordinate of some }
\bar x\in X\}.
\]

The next lemma is the technical core of the section. The idea is to
extract from a point of \(X\) not algebraic over \(M_0\cup \bar a\) a
finite-fiber coding relation on the controlling strongly minimal set, and
then use strong minimality to show that a cofinite subset of that set is
algebraic over \(M_0\cup \bar a\cup \operatorname{coord}(X)\).

\begin{lemma}
\label{lem:infinite-definable-recovers-model}
Let \(\MM_0\models T\) be arithmetically decidable, let
\(\varphi(x,\bar m_0)\) be strongly minimal over \(M_0\), let
\(\MM_0\preccurlyeq \MM\models T\), and put
\(C:=\varphi(M,\bar m_0)\). Assume \(M=\acl(M_0\cup C)\). Let
\(\bar a\in M^{<\omega}\), and let \(X\subseteq M^r\)
be an infinite set definable over \(M_0\cup \bar a\). Assume moreover
that \(X\not\subseteq \acl(M_0\cup \bar a)^r\). Then
\(C\subseteq \acl(M_0\cup \bar a\cup \operatorname{coord}(X))\).
Consequently,
\[
M=\acl(M_0\cup \bar a\cup \operatorname{coord}(X)).
\]
\end{lemma}

\begin{proof}
Fix a formula \(\chi(\bar x,\bar a,\bar m_X)\), with
\(\bar m_X\in M_0^{<\omega}\), defining \(X\). Choose
\(\bar c\in X\setminus \acl(M_0\cup \bar a)^r\), which is possible by
hypothesis; equivalently, some coordinate of \(\bar c\) is not algebraic
over \(M_0\cup\bar a\).

Since \(M=\acl(M_0\cup C)\), the tuple \(\bar c\) is algebraic over
\(M_0\cup C\). Hence there is a finite tuple \(\bar b\in C^n\) such that
\(\bar c\in \acl(M_0\cup \bar a\cup \bar b)\). Choose a formula
\(\psi(\bar x,\bar y,\bar a,\bar m)\),
with \(\bar m\in M_0^{<\omega}\), and an integer \(N<\omega\) such that
\(\MM\models \psi(\bar c,\bar b,\bar a,\bar m)\) and
\(\MM\models \exists^{\leq N}\bar x\,\psi(\bar x,\bar b,\bar a,\bar m)\).
Enlarge \(\bar m\), if necessary, so that \(\bar m_X\) is among its
coordinates. Suppress this harmless enlargement in the notation below.
Replace \(\psi\) by
\[
\rho(\bar x,\bar y,\bar a,\bar m)
:\iff
\psi(\bar x,\bar y,\bar a,\bar m)
\wedge
\exists^{\leq N}\bar x\,
\psi(\bar x,\bar y,\bar a,\bar m)
\wedge
\chi(\bar x,\bar a,\bar m_X).
\]
Then \(\MM\models \rho(\bar c,\bar b,\bar a,\bar m)\), and for every
\(\bar y\in C^n\), \(|\rho(M^r,\bar y,\bar a,\bar m)|\leq N\).

Apply Lemma~\ref{lem:coding-finite-fibers} over the parameter set
\(A:=M_0\cup \bar a\)
to the formula \(\rho(\bar x,\bar y,\bar a,\bar m)\) and the tuple
\(\bar b\). Here all extra parameters \(\bar m_X,\bar m\) are from
\(M_0\), so the only non-base parameters are the displayed tuple
\(\bar a\). We obtain \(m<\omega\), a tuple \(\bar d\in C^m\), and a
pair of formulas
\[
\gamma(\bar z,\bar a,\bar m),
\qquad
\eta(\bar x,\bar z,\bar a,\bar m)
\]
over \(M_0\cup\bar a\) such that:
\begin{enumerate}[label=(\roman*)]
\item \(\gamma(\bar d,\bar a,\bar m)\) and
\[
\eta(M^r,\bar d,\bar a,\bar m)
=
\rho(M^r,\bar b,\bar a,\bar m);
\]
\item for all \(\bar z\in C^m\), if \(\gamma(\bar z,\bar a,\bar m)\),
then
\[
|\eta(M^r,\bar z,\bar a,\bar m)|\leq N;
\]
\item for every \(\bar z_0\in C^m\) satisfying
      \(\gamma(\bar z_0,\bar a,\bar m)\), there are only finitely many
      \(\bar z\in C^m\) with \(\gamma(\bar z,\bar a,\bar m)\) and
\[
\eta(M^r,\bar z,\bar a,\bar m)
=
\eta(M^r,\bar z_0,\bar a,\bar m).
\]
\end{enumerate}

Let
\[
S:=
\eta(M^r,\bar d,\bar a,\bar m)
=
\rho(M^r,\bar b,\bar a,\bar m).
\]
This is a finite subset of \(X\) containing \(\bar c\). Choose a tuple
\(\bar e\in X^N\) enumerating the finite nonempty set \(S\), with
repetitions if necessary. Thus \(\bar e\) is an \(N\)-tuple of
\(r\)-tuples from \(X\).

Since \(S\) contains such a tuple \(\bar c\), the tuple \(\bar e\) is
not algebraic over \(M_0\cup\bar a\). Indeed, if \(\bar e\) were
algebraic over this base, then each coordinate of \(\bar e\) would be
algebraic over the same base. Since the coordinates of \(\bar e\)
enumerate the coordinates of the tuples in \(S\), every coordinate of
\(\bar c\) would lie in \(\acl(M_0\cup\bar a)\), contradicting the
choice of \(\bar c\).

Let \(F\) be the set of all \(\bar z\in C^m\) such that
\(\gamma(\bar z,\bar a,\bar m)\) and
\(\eta(M^r,\bar z,\bar a,\bar m)=S\).
Since \(S\) is finite and named by \(\bar e\), this equality is
first-order expressible, so \(F\) is definable over
\(M_0\cup\bar a\cup\bar e\). It contains
\(\bar d\), and property~(iii), applied to \(\bar z_0=\bar d\), shows
that \(F\) is finite. Hence
\(\bar d\in \acl(M_0\cup\bar a\cup\bar e)\).

Now define a relation \(R(\bar u,\bar z)\), over \(M_0\cup\bar a\), by
requiring that \(\bar z\in C^m\) satisfies
\(\gamma(\bar z,\bar a,\bar m)\) and
\(\bar u\in X^N\) enumerates the finite set
\(\eta(M^r,\bar z,\bar a,\bar m)\).
This is first-order expressible. Let
\[
W:=\{\bar u\in X^N:\exists \bar z\,R(\bar u,\bar z)\},
\qquad
Y:=\{\bar z\in C^m:\exists \bar u\,R(\bar u,\bar z)\}.
\]
The set \(W\) is definable over \(M_0\cup\bar a\) and contains
\(\bar e\notin \acl(M_0\cup\bar a)\), so \(W\) is infinite: a finite
definable set over \(M_0\cup\bar a\) consists entirely of tuples
algebraic over that base.

Since \(W\) is infinite and each fiber over \(Y\) is finite,
\(Y\) is infinite. Hence \(m\geq 1\), and some coordinate projection
\(\pi_i(Y)\subseteq C\) is infinite.
Since \(\pi_i(Y)\) is definable over
\(M_0\cup\bar a\cup\bar m\), strong minimality makes it cofinite in \(C\).

Fix \(d\in \pi_i(Y)\).
Then there are \(\bar u\in W\) and \(\bar z\in Y\) such that
\(R(\bar u,\bar z)\) and \(z_i=d\).
For fixed \(\bar u\), there are only finitely many tuples \(\bar z\)
with \(R(\bar u,\bar z)\), again by property~(iii). Thus the set
\[
F_{\bar u,i}:=\{d'\in C:\exists \bar z\,
   (R(\bar u,\bar z)\wedge z_i=d')\}
\]
is a finite set definable over \(M_0\cup \bar a\cup \bar u\). Since
\(d\in F_{\bar u,i}\), the element \(d\) is algebraic over
\(M_0\cup \bar a\cup \bar u\).
Because \(\bar u\in X^N\), this gives
\(d\in \acl(M_0\cup\bar a\cup\operatorname{coord}(X))\). Thus
\[
\pi_i(Y)\subseteq
\acl(M_0\cup\bar a\cup\operatorname{coord}(X)).
\]

Because \(\pi_i(Y)\) is cofinite in \(C\), the complement
\(C\setminus \pi_i(Y)\) is finite and definable over \(M_0\cup\bar a\).
Hence
\[
C\setminus \pi_i(Y)\subseteq \acl(M_0\cup\bar a)
\subseteq \acl(M_0\cup\bar a\cup\operatorname{coord}(X)).
\]
Combining the two inclusions, we obtain
\[
C\subseteq \acl(M_0\cup\bar a\cup\operatorname{coord}(X)).
\]

Finally, since \(M=\acl(M_0\cup C)\),
the transitivity of algebraic closure yields
\[
M=\acl(M_0\cup C)
\subseteq \acl(M_0\cup\bar a\cup\operatorname{coord}(X))
\subseteq M.
\]
Therefore
\[
M=\acl(M_0\cup\bar a\cup\operatorname{coord}(X)).
\]
\end{proof}

\begin{proposition}[Infinite definable sets recover the ambient degree]
\label{prop:infinite-definable-recovers-sm-set}
Assume the standing hypothesis of this section. Let \(D>\mathbf 0\), let
\(\MM_0\models T\) be arithmetically decidable, let
\[
\MM_0\preccurlyeq \MM\models T
\]
be arithmetically extendible with \( \degAr(M)=D. \)

Suppose \(X\subseteq M^r\) is an infinite set which is definable over \(M_0\cup\bar a\) for some \(\bar a\in M^{<\omega}\). If
\[
X\not\subseteq \acl(M_0\cup\bar a)^r,
\]
then \( \degAr(X)=D. \)
\end{proposition}

\begin{proof}
By Lemma~\ref{lem:ambient-arithmetical-inequalities}\textup{(i)},
we have \(X\leAr M\). To show the converse, we begin by fixing some strongly minimal formula \(\varphi(x,\bar m_0)\) over \(M_0\) (Lemma~\ref{lem:sm-over-arith-model}). Define \(C := \varphi(M,\bar m_0)\). Then one has \(M = \acl(M_0 \cup C)\). Since \(X\) is infinite and not contained in
\(\acl(M_0\cup\bar a)^r\),
Lemma~\ref{lem:infinite-definable-recovers-model} gives
\[
M=\acl_{\MM}(M_0\cup\bar a\cup\operatorname{coord}(X)).
\]
Notice \(\operatorname{coord}(X)\leAr X\): an element \(u\) is in
\(\operatorname{coord}(X)\) iff for some coded tuple
\(\bar x\in X\), one coordinate of \(\bar x\) is \(u\). Therefore 
\[
M_0\cup\bar a\cup\operatorname{coord}(X) \leAr X.
\]
It follows from Lemma~\ref{lem:ambient-arithmetical-inequalities}(iii) (relativized to \(X\)) that \(M \leAr X\). We conclude that \(\degAr(X) = \degAr(M) = D\).
\end{proof}

\section{Proof of the main theorem}
\label{sec:proof-main-theorems}

We now carry out the arithmetic Baldwin--Lachlan strategy. With the
structural analysis of the previous sections in place, the proof follows
the classical pattern more closely. One first matches the generic
geometry coming from the controlling strongly minimal set, then extends
that matching over finite algebraic pieces, and finally performs a
back-and-forth construction. The only genuinely new point, specific to
the present setting, is the degree-theoretic input from
Section~\ref{sec:sm-control-arith-degree}, which replaces the direct
strongly minimal degree calculation used in Section~\ref{sec:basic-facts-examples-dcat}.

We begin by isolating the arithmetic analogue of
Fact~\ref{fact:sm-standard-facts}\textup{(iv)}. In the strongly minimal
case, that fact says that independent tuples over a fixed base all
realize the same generic type. The next lemma supplies the corresponding
statement in the present strongly-minimal-control setting.

\begin{fact}[Stationarity of the generic type in a strongly minimal set]
\label{fact:sm-generic-stationarity}
Let \(C=\varphi(\mathbb U,\bar m_0)\) be a strongly minimal set in a
monster model, and let \(B=\acl(B)\) contain the parameters
\(\bar m_0\). For each \(n<\omega\), any two \(n\)-tuples from \(C\)
which are independent over \(B\cap C\) have the same complete type over
\(B\), in the full ambient language. This type is the \(n\)-fold
nonforking power of the generic type of \(C\) over \(B\), and a tuple
from \(C^n\) realizes it if and only if it is independent over
\(B\cap C\), where independence is computed in the pregeometry induced
on \(C\).
\end{fact}

\begin{proof}
This is the standard stationarity of the generic type of a strongly
minimal formula over an algebraically closed base;
see Marker~\cite[Chapter~6]{marker2002}. When tuples
come from different models, they are compared after embedding into a
common monster.
\end{proof}

\begin{lemma}[Generic arithmetic types on the controlling strongly minimal set]
\label{lem:generic-arith-n-type-on-sm-set}
Let \(\MM_0\models T\) be arithmetically decidable, let
\(\varphi(x,\bar m_0)\) be strongly minimal over \(M_0\), and put
\(C_0:=\varphi(M_0,\bar m_0)\).
Let \(\bar d\) be a finite tuple from \(\varphi(N,\bar m_0)\) in an
arithmetically decidable elementary extension
\(\MM_0\preccurlyeq \NN\models T\), and set
\(\MM_{\bar d}:=\acl_{\NN}(M_0\cup \bar d)\).
Then \(\MM_{\bar d}\) is arithmetically decidable. Moreover, for each
\(n\geq 1\), there is a unique arithmetic type
\(p_n^{\bar d}(\bar x)\in
S_n^{\mathrm{ar}}(M_{\bar d},\MM_{\bar d})\)
such that for every elementary extension
\(\MM_{\bar d}\preccurlyeq \BB\models T\), writing
\(C_{\BB}:=\varphi(B,\bar m_0)\),
and every \(\bar b=(b_0,\dots,b_{n-1})\in C_{\BB}^n\), one has
\[
\bar b\models p_n^{\bar d}
\quad\Longleftrightarrow\quad
b_0,\dots,b_{n-1}
\text{ are independent over } C_0\cup\bar d.
\]
In particular, taking \(\bar d=\emptyset\), for each \(n\geq 1\) there
is a unique arithmetic type \(p_n(\bar x)\in
S_n^{\mathrm{ar}}(M_0,\MM_0)\)
whose realizations are exactly the \(n\)-tuples from the controlling
strongly minimal set independent over \(C_0\).
\end{lemma}

\begin{proof}
By Lemma~\ref{lem:finite-dimensional-control-submodel},
\(\MM_{\bar d}\preccurlyeq\NN\). The same lemma, using the
arithmetically decidable ambient model \(\NN\), gives an
arithmetically decidable presentation of \(\MM_{\bar d}\).

Put \(\MM_*:=\MM_{\bar d}\) and
\(C_*:=\varphi(M_*,\bar m_0)\).
Since \(\MM_*\) is algebraically closed in elementary extensions,
Fact~\ref{fact:sm-generic-stationarity} gives, for each \(n<\omega\), a
unique complete type over \(M_*\) realized by \(n\)-tuples from the
controlling strongly minimal set which are independent over
\(C_*\). The set \(C_*\) is exactly the closure of
\(C_0\cup\bar d\) in the pregeometry induced on the strongly minimal
set, so independence over \(C_*\) is the same as independence over
\(C_0\cup\bar d\).

It remains only to see that this unique generic type is arithmetic. Add
new constants \(c_0,\dots,c_{n-1}\) and consider the arithmetically
definable partial type over \(M_*\) consisting of
\(\varphi(c_i,\bar m_0)\) for \(i<n\), together with, for each \(i<n\),
every implication
\[
\exists^{\leq k}x\,
\bigl(\varphi(x,\bar m_0)\wedge
\theta(x,c_0,\dots,c_{i-1},\bar m)\bigr)
\ \longrightarrow\
\neg\theta(c_i,c_0,\dots,c_{i-1},\bar m),
\]
where \(k<\omega\), \(\bar m\in M_*^{<\omega}\), and \(\theta\) ranges
over formulas. The enumeration of the tuples \(\bar m\) is arithmetical
because \(M_*\) has an arithmetically decidable presentation, and
\(\exists^{\leq k}\) is expressed by the usual first-order bounded
cardinality formula. This scheme says exactly that each \(c_i\) avoids
every finite definable subset of the strongly minimal set over
\(M_*\cup\{c_0,\dots,c_{i-1}\}\).
Equivalently, in the pregeometry on the strongly minimal set, it says
\[
c_i\notin \operatorname{cl}_C(C_*\cup\{c_0,\dots,c_{i-1}\}).
\]
Every finite fragment is satisfiable in an elementary extension of
\(\MM_*\): choose the constants successively outside the finite union of
the algebraic loci named in the current finite fragment. Equalities
\(c_i=c_j\) are themselves finite algebraic conditions, so distinctness
is enforced whenever it is required by independence. By
Lemma~\ref{lem:base-preserving-coded-compactness}, applied over the
elementary diagram of \(\MM_*\), the partial type has an arithmetically
decidable realization in some elementary extension
\(\MM_*\preccurlyeq\widehat{\MM}\models T\).
Let \(\bar a\) be the realized tuple and set
\(p_n^{\bar d}:=\tp^{\widehat{\MM}}(\bar a/M_*)\).
Since \(\widehat{\MM}\) is arithmetically decidable and \(M_*\) is
arithmetical, \(p_n^{\bar d}\) is an arithmetic type.

By Fact~\ref{fact:sm-generic-stationarity}, every independent
\(n\)-tuple from \(\varphi(B,\bar m_0)\) in any elementary extension
\(\MM_*\preccurlyeq\BB\models T\) realizes this type. Conversely, if a
tuple from the strongly minimal set is dependent over \(C_0\cup\bar d\),
then some coordinate is algebraic over \(M_*\) together with the earlier
coordinates. That dependence is witnessed by a formula whose fiber in
the strongly minimal set has size at most some \(k\), so it is one of
the finite-algebraicity conditions forbidden by the displayed avoidance
scheme. Hence the tuple cannot realize
\(p_n^{\bar d}\).
\end{proof}

\begin{lemma}[Stationarity over finite control bases]
\label{lem:stationarity-finite-control-bases}
Let \(\MM_0\) be arithmetically decidable, let
\(\MM_0\preccurlyeq\MM_1,\MM_2\models T\), let
\(\varphi(x,\bar m_0)\) be the controlling formula supplied by
Lemma~\ref{lem:sm-over-arith-model}. Put
\(C_i:=\varphi(M_i,\bar m_0)\) and \(C_0:=\varphi(M_0,\bar m_0)\), and
assume \(M_i=\acl(M_0\cup C_i)\) for \(i=1,2\). Let
\(B_i\subseteq C_i\) be bases over \(C_0\), and let
\(g:B_1\to B_2\) be a bijection. Let \(\bar b\) be a finite tuple from
\(B_1\), and let \(\bar b':=g(\bar b)\). Suppose
\[
\tp^{\MM_1}(A,\bar b/M_0)=
\tp^{\MM_2}(A',\bar b'/M_0),
\]
where \(A\subseteq\acl(M_0\cup\bar b)\) and
\(A'\subseteq\acl(M_0\cup\bar b')\) are finite. If \(\bar e\) and
\(\bar e'\) are corresponding finite tuples from \(C_1\) and \(C_2\),
independent over \(C_0\cup\bar b\) and \(C_0\cup\bar b'\), respectively,
then
\[
\tp^{\MM_1}(A,\bar b,\bar e/M_0)=
\tp^{\MM_2}(A',\bar b',\bar e'/M_0).
\]
\end{lemma}

\begin{proof}
Fix once and for all enumerations of the finite sets \(A,A'\) matching
the displayed type equality, and read \(h[A]\)-style notation
coordinatewise with those enumerations. Work after embedding
\(\MM_1\) and \(\MM_2\) into a sufficiently large monster model. The
displayed equality of types then gives, by homogeneity, an automorphism
\(\sigma\) fixing \(M_0\) and sending \((A,\bar b)\) to
\((A',\bar b')\). Since \(\sigma\) fixes \(M_0\), it fixes the defining
parameters \(\bar m_0\) of \(C\), and hence preserves the controlling
strongly minimal set. It sends
\(\acl(M_0\cup\bar b)\) onto \(\acl(M_0\cup\bar b')\). These are
elementary algebraically closed bases by
Lemma~\ref{lem:finite-dimensional-control-submodel}.

The tuples \(\bar e\) and \(\bar e'\) have the same length and are
listed in the corresponding order. The tuple \(\sigma(\bar e)\) is
independent over \(C_0\cup\bar b'\) in the induced pregeometry on the
controlling strongly minimal set.
By Fact~\ref{fact:sm-generic-stationarity}, the generic type of the
strongly minimal set over the algebraically closed base
\(\acl(M_0\cup\bar b')\) is stationary. Thus \(\sigma(\bar e)\) and
\(\bar e'\) have the same complete type over
\(\acl(M_0\cup\bar b')\). An automorphism fixing this base sends
\(\sigma(\bar e)\) to \(\bar e'\); composing it with \(\sigma\) gives
the required equality of types over \(M_0\).
\end{proof}

\begin{lemma}[Finite-dimensional definable sets are arithmetical]
\label{lem:finite-dimensional-trap-arithmetical}
Let \(\MM_0\) be arithmetically decidable, let
\(\varphi(x,\bar m_0)\) be the controlling formula from
Lemma~\ref{lem:sm-over-arith-model}, and let
\(\MM_0\preccurlyeq\MM\models T\). Put \(C=\varphi(M,\bar m_0)\). If
\(\bar b\in C^{<\omega}\), \(N_{\bar b}:=\acl(M_0\cup\bar b)\), and
\(\bar a\in N_{\bar b}^{<\omega}\), then any definable set
\[
X=\psi(M^r,\bar a)
\]
which is contained in \(N_{\bar b}^r\) is arithmetical, relative only to
the finite data \(\bar b,\bar a\) and the fixed arithmetically decidable
ambient presentation witnessing extendibility.
\end{lemma}

\begin{proof}
By Lemma~\ref{lem:finite-dimensional-control-submodel},
\(N_{\bar b}\preccurlyeq M\) and \(N_{\bar b}\) has an arithmetically
decidable presentation, uniformly in \(\bar b\). If
\(X\subseteq N_{\bar b}^r\), then elementarity gives
\[
X=\psi(N_{\bar b}^r,\bar a).
\]
Thus the decision procedure for \(X\subseteq M^r\) has two cases. If
\(\bar x\notin N_{\bar b}^r\), answer no, using the arithmetical domain
predicate for \(N_{\bar b}\). If \(\bar x\in N_{\bar b}^r\), evaluate
\(\psi(\bar x,\bar a)\) in the arithmetically decidable presentation of
\(N_{\bar b}\). Hence \(X\) is arithmetical.
\end{proof}

\begin{theorem}[Arithmetic Baldwin--Lachlan theorem]
\label{thm:arith-baldwin-lachlan}
Let \(T\) be a complete arithmetically definable theory in an
arithmetically definable countable language with no finite models.
Consider the following statements:
\begin{enumerate}[label=(\arabic*)]
\item \(T\) is \(D\)-categorical for some nonzero arithmetic degree
      \(D\);
\item \(T\) is arithmetically stable and has no Vaughtian pairs;
\item there exist an arithmetically decidable prime model
      \(\PP\models T\) and a formula \(\varphi(x,\bar p)\), with
      parameters from \(P\), such that \(\varphi(x,\bar p)\) is strongly
      minimal over \(P\), and every elementary extension
      \(\PP\preccurlyeq \MM\models T\) satisfies
      \(M=\acl(P\cup \varphi(M,\bar p))\);
\item \(T\) is \(D\)-categorical for every nonzero arithmetic degree
      \(D\).
\end{enumerate}
Then
\[
(1)\Rightarrow(2)\Rightarrow(3)\Rightarrow(4).
\]
As an external metatheoretic consequence, whenever a nonzero arithmetic
degree is available, the four statements are equivalent.
\end{theorem}

\begin{proof}
The implication \((1)\Rightarrow(2)\) is the combination of
Proposition~\ref{prop:cat-implies-stability} and
Proposition~\ref{prop:d-cat-no-vaughtian}. The implication
\((2)\Rightarrow(3)\) is the content of
Corollary~\ref{cor:prime-model},
Proposition~\ref{prop:strongly-minimal-formula}, and
Proposition~\ref{prop:model-generated-by-strongly-minimal-set}.

It remains to prove \((3)\Rightarrow(4)\).

Assume clause~\textup{(3)}. Fix a nonzero arithmetic degree \(D\), an
arithmetically decidable model \(\MM_0\models T\), and arithmetically
extendible models \(\MM_1,\MM_2\models T\) such that
\(\MM_0\preccurlyeq \MM_1,\MM_2\) and
\(\degAr(M_1)=\degAr(M_2)=D\).

Let \(\varphi(x,\bar m_0)\) be as in
Lemma~\ref{lem:sm-over-arith-model}, and for \(i=1,2\) put
\(C_i:=\varphi(M_i,\bar m_0)\), with
\(C_0:=\varphi(M_0,\bar m_0)\).
Then \(\varphi(x,\bar m_0)\) is strongly minimal over \(M_0\), and
\(M_i=\acl(M_0\cup C_i)\) for \(i=1,2\), by
Lemma~\ref{lem:sm-over-arith-model}.

Let \(B_i\subseteq C_i\) be a basis of \(C_i\) over \(C_0\). Then
\(M_i=\acl(M_0\cup B_i)\).
We choose \(B_i\) arithmetically in \(M_i\), as follows. Since
\(\MM_0\preccurlyeq\MM_i\), the model \(M_0\) is algebraically closed
inside \(M_i\). Hence, on \(C_i\), algebraic dependence over \(M_0\)
agrees with algebraic dependence over \(C_0=C_i\cap M_0\).
If \(C_i\) had finite dimension over \(C_0\), then \(M_i\) would be the
algebraic closure of \(M_0\) together with finitely many elements of
\(C_i\). By Lemma~\ref{lem:ambient-arithmetical-inequalities}\textup{(iii)}
this would make \(M_i\) arithmetical, contradicting
\(\degAr(M_i)=D>\mathbf 0\). Thus the dimension is infinite.

With respect to the fixed coding of the domain as a subset of
\(\N\), recursively set
\[
b^i_n:=\min\Bigl(C_i\setminus
\acl(M_0\cup\{b^i_0,\dots,b^i_{n-1}\})\Bigr).
\]
The displayed algebraic
closures are arithmetical uniformly in the finite tuple by
Lemma~\ref{lem:ambient-arithmetical-inequalities}\textup{(iii)}, so
\(B_i=\{b^i_n:n<\omega\}\leAr M_i\).
Maximality follows from finite character: any element outside
\(\acl(M_0\cup B_i)\) would eventually be selected by the greedy
recursion.
Since \(M_i=\acl(M_0\cup B_i)\) and membership in this closure is
arithmetical in \(B_i\), we also have \(M_i\leAr B_i\), so
\(\degAr(B_i)=D\).
Choose increasing \(D\)-arithmetical enumerations
\(B_1=\{b^1_0,b^1_1,\dots\}\) and
\(B_2=\{b^2_0,b^2_1,\dots\}\), and let \(g:B_1\to B_2\) be the induced
\(D\)-arithmetical bijection.

The proof now follows the same pattern as
Proposition~\ref{prop:strongly-minimal-D-cat}. The only new input is
Lemma~\ref{lem:generic-arith-n-type-on-sm-set}, which supplies the
generic-type uniqueness for independent tuples in the controlling
strongly minimal set.

\medskip
\noindent
\textbf{Claim 1.}
Let \(\bar b\) be a finite tuple from \(B_1\), and let
\(\bar b':=g(\bar b)\).
Suppose
\[
\tp^{\MM_1}(\bar u,\bar b/M_0)
=
\tp^{\MM_2}(\bar u',\bar b'/M_0),
\]
where \(\bar u\in \acl(M_0\cup\bar b)\) and
\(\bar u'\in \acl(M_0\cup\bar b')\). If
\(\bar a\in \acl(M_0\cup\bar b)\), then there is
\(\bar a'\in M_2^{|\bar a|}\) such that
\[
\tp^{\MM_1}(\bar u,\bar a,\bar b/M_0)
=
\tp^{\MM_2}(\bar u',\bar a',\bar b'/M_0).
\]

\medskip
\noindent
\emph{Proof of Claim~1.}
Choose a formula \(\theta(\bar x,\bar z,\bar y,\bar m)\), with
\(\bar m\in M_0^{<\omega}\), such that
\(\MM_1\models \theta(\bar a,\bar u,\bar b,\bar m)\) and
\(\theta(\bar x,\bar u,\bar b,\bar m)\) has exactly \(q<\omega\)
realizations in \(\MM_1\). The type hypothesis transfers the
exact-cardinality assertion to \(\MM_2\). Let \(F\) be the set of
realizations of \(\theta(\bar x,\bar u',\bar b',\bar m)\) in \(\MM_2\);
it has size \(q\).

Some \(\bar a'\in F\) must realize the same type over
\((\bar u',\bar b',M_0)\) that \(\bar a\) realizes over
\((\bar u,\bar b,M_0)\): otherwise, for each \(\bar c\in F\) choose
a formula, with parameters from \(M_0\), separating the corresponding
types. The conjunction of \(\theta\) with all
these separating formulas is realized in
\(\MM_1\), hence also in \(\MM_2\), but every realization in \(\MM_2\)
lies in \(F\) and fails one of the conjuncts. Contradiction.
\hfill\(\dashv\)

\medskip

By symmetry, the analogous statement holds with the roles of
\(\MM_1\) and \(\MM_2\) reversed.

A finite partial map \(h:A\to M_2\) is called \emph{good} if:
\begin{enumerate}[label=(\roman*)]
\item \(A\subseteq M_1\) is finite;
\item \(h\) fixes \(M_0\cap A\) pointwise;
\item there is a finite tuple \(\bar b\) from \(B_1\) such that
      \[
      A\subseteq \acl(M_0\cup\bar b),
      \qquad
      h[A]\subseteq \acl(M_0\cup g(\bar b)),
      \]
      and
      \[
      \tp^{\MM_1}(A,\bar b/M_0)
      =
      \tp^{\MM_2}(h[A],g(\bar b)/M_0).
      \]
\end{enumerate}
Here \(A\) is read in increasing tuple-code order, and \(h[A]\) is read
in the corresponding image order; type equality over \(M_0\) means
equality of all formulas with parameters from the arithmetically
decidable presentation of \(M_0\).

Since \(M_1=\acl(M_0\cup B_1)\), every finite tuple from \(M_1\) is
contained in the domain of some good map. By symmetry, every finite
tuple from \(M_2\) is contained in the range of some good map.

\medskip
\noindent
\textbf{Claim 2.}
Every good map extends to a good map whose domain contains any
prescribed element of \(M_1\), and also to a good map whose range
contains any prescribed element of \(M_2\).

\medskip
\noindent
\emph{Proof of Claim~2.}
Let \(h:A\to M_2\) be good, witnessed by a finite tuple \(\bar b\) from
\(B_1\).

Let \(a\in M_1\). Since \(M_1=\acl(M_0\cup B_1)\), there is a finite
tuple \(\bar d\) from \(B_1\) such that
\(A\cup\{a\}\subseteq \acl(M_0\cup\bar d)\).
Enlarging \(\bar b\) if necessary, assume that \(\bar b\) is a subtuple
of \(\bar d\). Put \(\bar d':=g(\bar d)\).

We first observe that \(h\) is also good as witnessed by \(\bar d\).
Indeed, list the new coordinates of \(\bar d\setminus\bar b\) in the
order inherited from \(\bar d\), and list their images in
\(\bar d'\setminus g(\bar b)\) in the corresponding order. These two
tuples are independent over \(C_0\cup\bar b\) and
\(C_0\cup g(\bar b)\), respectively. Since \(h\) is already good over
\(\bar b\), Lemma~\ref{lem:stationarity-finite-control-bases} shows
that adding these corresponding independent tuples preserves the type
over \(M_0\).
Thus
\[
\tp^{\MM_1}(A,\bar d/M_0)
=
\tp^{\MM_2}(h[A],\bar d'/M_0).
\]

Claim~1 then gives \(a'\in M_2\) such that
\[
\tp^{\MM_1}(A,a,\bar d/M_0)
=
\tp^{\MM_2}(h[A],a',\bar d'/M_0).
\]
Therefore \(h\cup\{(a,a')\}\) is good.

The extension on the range side is proved symmetrically.
\hfill\(\dashv\)

\medskip

We now carry out the back-and-forth. Since \(M_1,M_2\) are
\(D\)-arithmetical and \(M_0\) is arithmetical, take the increasing
\(D\)-arithmetical enumerations
\(M_1\setminus M_0=\{a_0,a_1,a_2,\dots\}\) and
\(M_2\setminus M_0=\{c_0,c_1,c_2,\dots\}\). Finite tuples from \(B_1\)
are listed in increasing natural-number order, and we minimize their
fixed primitive recursive tuple-codes.

We define a normalized tree \(\mathcal T\) of witnessed good maps. A
node code consists of the finite graph of \(h\), its finite domain
\(A=\operatorname{dom}(h)\) (included for bounded coding), and the
witnessing basis tuple \(\bar b\). The root is the empty map, witnessed
by the empty tuple. Suppose \((h,A,\bar b)\) is a node at level \(s\).
A child is obtained as follows. First choose the least tuple-code tuple
\(\bar b^*\) from
\(B_1\) such that \(\bar b\subseteq\bar b^*\), the map \(h\) remains
good as witnessed by \(\bar b^*\),
\[
a_s\in\acl(M_0\cup\bar b^*)
\quad\text{and}\quad
c_s\in\acl(M_0\cup g(\bar b^*)).
\]
Then extend \(h\) only by the elements needed at this stage: add
\(a_s\) to the domain if it is not already there, and, if \(c_s\) is not
already in the range, add one preimage of \(c_s\) from
\(\acl(M_0\cup\bar b^*)\). The resulting finite map is required to be
good, witnessed by \(\bar b^*\). No other new elements are allowed in a
child. Claim~2 guarantees that the required tuple and at least one such
extension exist at every level, so \(\mathcal T\) is infinite.

\medskip
\noindent
\textbf{Claim 3.}
\(\mathcal T\) is finitely branching.

\medskip
\noindent
\emph{Proof of Claim~3.}
Fix a node \((h,A,\bar b)\) at level \(s\). In a child, the witness
\(\bar b'\) is forced to be the finite increasing tuple from \(B_1\) with
least tuple-code satisfying the normalized conditions above. Such a
tuple exists by Claim~2, so the relevant set of tuple-codes is nonempty
and has a least element. The fact that the old map remains good over
this larger witness follows from
Lemma~\ref{lem:stationarity-finite-control-bases}, using
Lemma~\ref{lem:generic-arith-n-type-on-sm-set} for the generic type,
exactly as in the proof of Claim~2.

Once \(\bar b'\) is fixed, a normalized child adds only \(a_s\) and, if
needed, one preimage of \(c_s\). The possible images of \(a_s\) lie in a
finite algebraic set over \(M_0\cup g(\bar b')\), and, after such an
image is chosen, the possible preimages of \(c_s\) lie in a finite
algebraic set over \(M_0\cup\bar b'\). Hence there are only finitely
many possible child maps. Thus \(\mathcal T\) is finitely branching.
\hfill\(\dashv\)

\medskip
\noindent
\textbf{Claim 4.}
\(\mathcal T\) is \(D\)-arithmetical, hence it has a \(D\)-arithmetical infinite path.

\medskip
\noindent
\emph{Proof of Claim~4.}
The sets \(B_1,B_2\) and the map \(g:B_1\to B_2\) are
\(D\)-arithmetical. For each finite tuple \(\bar b\), the algebraic
closures \(\acl(M_0\cup\bar b)\) and
\(\acl(M_0\cup g(\bar b))\) are arithmetical uniformly in \(\bar b\), by
Lemma~\ref{lem:ambient-arithmetical-inequalities}\textup{(iii)}. The
type comparison in the goodness condition is \(D\)-arithmetical, using
arithmetically decidable elementary extensions witnessing
arithmetical extendibility: it is a universal comparison of formula
codes and finite parameter tuples from \(M_0\), evaluated in those
decidable extensions. The least-witness condition is arithmetical
because, for a displayed candidate \(\bar b^*\), one checks all smaller
tuple codes by bounded search. The rest of the normalized transition
condition is a finite algebraicity test for the one or two new map
values.

Thus \(\mathcal{T}\) is
\(D\)-arithmetical. By
Lemma~\ref{lem:finite-jump-path-finitely-branching-tree}, \(\mathcal T\)
has an infinite \(D\)-arithmetical path.
\hfill\(\dashv\)

\medskip

Let
\[
f:=\operatorname{id}_{M_0}\cup\bigcup_s h_s,
\]
where the union is taken along such an infinite path. Then
\(f:M_1\to M_2\) is a total \(D\)-arithmetical bijection fixing \(M_0\)
pointwise, and every finite restriction of \(f\) is elementary over
\(M_0\). Therefore \(f:\MM_1\cong_{M_0}\MM_2\). The inverse is
\(D\)-arithmetical as well, since the graph of \(f^{-1}\) is obtained
from the \(D\)-arithmetical graph of \(f\) by reversing coordinates.

It remains to show that \(f\) is \(D\)-preserving. Let
\(X\subseteq M_1^r\) be parameter-definable, say
\(X=\psi(M_1^r,\bar a)\), and
suppose \(\degAr(X)=D\). Then \(X\) is infinite. Since
\(M_1=\acl(M_0\cup B_1)\), there is a finite tuple \(\bar b\) from
\(B_1\) such that \(\bar a\in\acl(M_0\cup\bar b)\). By
Lemma~\ref{lem:finite-dimensional-control-submodel},
\[
N_{\bar b}:=\acl(M_0\cup\bar b)
\]
is an arithmetically decidable elementary submodel of \(\MM_1\). If
\(X\subseteq\acl(M_0\cup\bar a)^r\), then \(X\subseteq N_{\bar b}^r\)
and Lemma~\ref{lem:finite-dimensional-trap-arithmetical} would make
\(X\) arithmetical, contradicting \(\degAr(X)=D>\mathbf 0\). Hence
\[
X\not\subseteq\acl(M_0\cup\bar a)^r.
\]

By elementarity of \(f\),
\(f[X]=\psi(M_2^r,f(\bar a))\).
Since \(f(\bar a)\in\acl(M_0\cup g(\bar b))\), the image \(f[X]\) is
an infinite definable set with
\(f[X]\not\subseteq \acl(M_0\cup f(\bar a))^r\) (otherwise
\(f^{-1}\) would trap \(X\) in \(\acl(M_0\cup\bar a)^r\)).
By Proposition~\ref{prop:infinite-definable-recovers-sm-set} applied in
\(\MM_2\), \(\degAr(f[X])=D\).

The same argument applied to \(f^{-1}\) shows that if
\(Y\subseteq M_2^r\) is parameter-definable and \(\degAr(Y)=D\), then
\(\degAr(f^{-1}[Y])=D\).
Therefore \(f\) is \(D\)-preserving.

Thus \(T\) is \(D\)-categorical. Since \(D>\mathbf 0\) was arbitrary,
clause~\textup{(4)} follows.
\end{proof}

\begin{theorem}[Arithmetic Morley theorem]
\label{thm:morley}
Let \(T\) be a complete arithmetically definable theory in an
arithmetically definable countable language with no finite models. If
\(T\) is \(D_1\)-categorical for some nonzero arithmetic degree \(D_1\),
then \(T\) is \(D_2\)-categorical for every nonzero arithmetic degree
\(D_2\).
\end{theorem}

\begin{proof}
This is the implication \((1)\Rightarrow(4)\) in
Theorem~\ref{thm:arith-baldwin-lachlan}, under the same completeness,
arithmetical definability, and no-finite-model hypotheses. Concretely,
\(D_1\)-categoricity gives arithmetical stability and no Vaughtian
pairs, hence the structural strongly minimal control clause, and that
clause gives \(D_2\)-categoricity for each nonzero \(D_2\).
\end{proof}

\begin{fact}[Classical Baldwin--Lachlan theorem]
\label{fact:classical-baldwin-lachlan}
Let \(T\) be a complete countable first-order theory in the infinite
case considered in this paper. The following are equivalent:
\begin{enumerate}[label=(\arabic*)]
\item \(T\) is \(\kappa\)-categorical for some uncountable cardinal
      \(\kappa\);
\item \(T\) is \(\kappa\)-categorical for every uncountable cardinal
      \(\kappa\);
\item there exist a prime model \(\PP\models T\) and a nonalgebraic
      formula
      \(\varphi(x,\bar p)\), with parameters from \(P\), such that
      \(\varphi(x,\bar p)\) is strongly minimal over \(P\), and every
      elementary extension \(\PP\preccurlyeq \MM\models T\) satisfies
      \(M=\acl(P\cup \varphi(M,\bar p))\).
\end{enumerate}
\end{fact}

\begin{proof}
This is the usual Baldwin--Lachlan theorem for countable complete
theories, used here only in the infinite/no-finite-model setting. In
the standard formulation the prime model is over \(\emptyset\), and the
parameters \(\bar p\) are then named from that prime model; see
Marker~\cite[Chapter~6]{marker2002}.
\end{proof}

\begin{proposition}[Arithmetically decidable prime model under uncountable categoricity]
\label{prop:zfc-arith-prime-under-uncountable-cat}
Assume \textsf{ZFC}. Let \(T\) be a complete arithmetically definable
theory in an arithmetically definable countable language. If \(T\) is
uncountably categorical, then \(T\) has an arithmetically decidable
prime model.
\end{proposition}

\begin{proof}
By the classical Baldwin--Lachlan theorem, isolated types are dense
over \(\emptyset\).
Since \(T\) is arithmetically definable, provability from \(T\) is
arithmetical, and the predicate saying that a formula isolates a
complete type is arithmetical.

Carry out a Henkin construction with constants
\((c_i)_{i\in\N}\), scheduling all finite tuples of closed terms. At
each stage, decide the next sentence and add Henkin witnesses as usual.
At the isolation stage for a scheduled tuple \(\bar t\), let \(\delta\)
be the current finite condition. Enlarge \(\bar t\) if necessary to
include all closed terms in \(\delta\), and choose the least formula
\(\theta\) that refines \(\delta\), is consistent with \(T\), and
isolates a complete type. Density guarantees such a refinement, and
the least-choice relation is arithmetical.

The resulting complete Henkin theory \(T^*\) is arithmetically
definable, and its term model is arithmetically decidable by the usual
truth lemma. The isolation requirements ensure
that every finite tuple realizes an isolated type over
\(\emptyset\), so the term model is atomic, hence prime.
\end{proof}

We now isolate the comparison with the classical uncountable
categoricity theorem.

\begin{theorem}[Equivalence with uncountable categoricity]
\label{thm:morley-zfc}
Assume \textsf{ZFC}. Let \(T\) be a complete arithmetically definable
theory in an arithmetically definable countable language with no finite
models. Then the following are equivalent:
\begin{enumerate}[label=(\arabic*)]
\item \(T\) is \(D\)-categorical for some nonzero arithmetic degree
      \(D\);
\item \(T\) is \(D\)-categorical for every nonzero arithmetic degree
      \(D\);
\item \(T\) is \(\kappa\)-categorical for some uncountable cardinal
      \(\kappa\);
\item \(T\) is \(\kappa\)-categorical for every uncountable cardinal
      \(\kappa\).
\end{enumerate}
\end{theorem}

\begin{proof}
By Theorem~\ref{thm:arith-baldwin-lachlan}, clauses~\textup{(1)} and
\textup{(2)} imply the structural condition that there exist an
arithmetically decidable prime model \(\PP\models T\) and a strongly
minimal formula over \(P\) generating every elementary extension over
\(P\). After forgetting the coded presentation and the effectiveness
requirements, this is exactly the structural clause of
Fact~\ref{fact:classical-baldwin-lachlan}. Hence
\textup{(1)} and \textup{(2)} imply \textup{(3)} and \textup{(4)}.

Conversely, suppose \(T\) is uncountably categorical. By
Fact~\ref{fact:classical-baldwin-lachlan}, there is a prime model
\(\PP\models T\) and a strongly minimal formula over \(P\) generating
every elementary extension over \(P\). By
Proposition~\ref{prop:zfc-arith-prime-under-uncountable-cat}, choose an
arithmetically decidable prime model \(\PP_0\models T\). Prime models
of a complete theory are isomorphic: equivalently, countable atomic
models admit the standard back-and-forth over isolated finite types. We
therefore transport the strongly minimal formula and the generation
property from \(\PP\) to \(\PP_0\). This is an external \(\textsf{ZFC}\)
transport; the isomorphism between the two prime models is not asserted
to be arithmetical, and no effectiveness is needed at this comparison
stage.
Thus the structural clause in Theorem~\ref{thm:arith-baldwin-lachlan}
holds. Since \textsf{ZFC} supplies nonzero arithmetic degrees, that
theorem gives \(D\)-categoricity for every nonzero arithmetic degree.

Hence all four clauses are equivalent.
\end{proof}

\begin{remark}[Base version of uncountable categoricity]
\label{rem:base-version-uncountable-categoricity}
The classical theorem has a corresponding over-a-base form. If \(T\) is
uncountably categorical and \(M_0\models T\) is countable, then for every
uncountable cardinal \(\kappa\), any two elementary extensions
\[
M_0\preccurlyeq M_1,M_2\models T
\]
of cardinality \(\kappa\) are isomorphic over \(M_0\). This follows by
running the Baldwin--Lachlan analysis over \(M_0\): the controlling
strongly minimal set has bases of size \(\kappa\) in both extensions, and
a bijection between these bases extends to an isomorphism fixing \(M_0\).
Conversely, one recovers the usual uncountable categoricity statement only
after running the same Baldwin--Lachlan analysis, which in particular
supplies the prime model. The over-base formulation by itself is not an
independent route to the existence of a prime model.
\end{remark}
\section{Concluding remarks and open problems}
\label{sec:concluding-remarks}

The main theorem suggests a general question: can one develop
classification theory in settings where cardinality is replaced by some
other measure of complexity? In the present paper, the relevant measure
is arithmetic degree, together with the restriction to arithmetically
decidable presentations, and the Baldwin--Lachlan analysis still goes
through in a meaningful form. This leads to the following problems.

\begin{enumerate}[label=(\arabic*), leftmargin=2em]

\item \textbf{Turing degree instead of arithmetic degree.}
Can the results of this paper be reformulated at the level of Turing
degrees rather than arithmetic degrees?

This is the most immediate technical question left open by the paper.
Many arguments here are effective in spirit, but arithmetic degree has
one important structural advantage: it is invariant under finite Turing
jumps. That invariance is used repeatedly, for example in compactness
arguments, in the coding of type spaces, in the passage from
arithmetically decidable structures to arithmetical satisfaction
relations, and in the low-basis-theorem step. A Turing-degree version
would therefore require replacing several uses of ``arithmetical in''
by a more rigid degree-preserving analysis which keeps track of the
actual finite jumps.

\item \textbf{Borel complexity.}
Is there an analogue of the categoricity/stability paradigm in which
cardinality is replaced by Borel complexity, or more generally by Borel
reducibility?

There is already a large body of work relating model theory to Borel
equivalence relations. The question here is more specific: is there a
classification-theoretic notion parallel to categoricity in power, but
organized by descriptive complexity rather than by the number of models
in a fixed cardinal? Even formulating the right analogue is not obvious.
One would need a replacement for the role played by ``all models of
size \(\kappa\)'' in the classical theory.

\item \textbf{Usual cardinality without choice.}
What survives of classification theory in \(\mathsf{ZF}\) when one keeps
ordinary cardinality but drops the axiom of choice?

Here one should distinguish sharply between well-ordered and
non-well-ordered cardinals. For well-ordered cardinals, Shelah proved
that substantial parts of the classical picture persist in
\(\mathsf{ZF}\); in particular, he obtained choiceless versions of
Morley's theorem over well-ordered cardinals.

For non-well-ordered cardinals, the situation deteriorates quickly. A
basic obstruction is already visible for vector spaces: without choice,
one cannot prove in general that two infinite-dimensional vector spaces
over \(\mathbb F_2\) of the same non-well-orderable cardinality are
isomorphic. Thus even the classification of models of the theory of
vector spaces can fail in the most elementary case. Closely related
basis principles are poorly understood: for instance, broad assertions
that vector spaces have bases can already imply substantial choice
principles, while fixed-field versions remain more delicate.

\item \textbf{Complexity in place of cardinality.}
To what extent can stability theory and classification theory be
developed when the role of cardinality is replaced by another invariant?

The point is not merely to define a new notion of categoricity. In the
arithmetically decidable setting, categoricity in one nonzero arithmetic
degree forces genuine structure: arithmetical stability, no Vaughtian
pairs, a prime model, and control by a strongly minimal set. It is
natural to ask for a general framework in which an external complexity
invariant plays the same role that cardinality plays classically. At
present we do not know what hypotheses on such an invariant would be
enough to support analogues of stability, rank, prime models, and
geometric control.

A first internal test case is the following: is arithmetical stability
actually equivalent to \(\omega\)-stability? In the present formulation,
arithmetical stability asserts that, over every arithmetically
decidable model and every arithmetical parameter set, the corresponding
spaces of arithmetic types are uniformly bounded by some finite jump
\(0^{(k)}\). The results of this paper show that \(D\)-categoricity
implies this boundedness condition, and then recover the same sort of
strongly minimal control one expects in the \(\omega\)-stable setting.
This suggests that arithmetical stability may simply be a reformulation
of \(\omega\)-stability. On the other hand, it is conceivable that the
arithmetic viewpoint detects a finer dividing line. Resolving this seems
to require a theory of regular types in arithmetical settings.

\item \textbf{A counterpart of countable categoricity.}
Is there a natural arithmetic analogue of countable categoricity?

The present paper deals with the analogue of uncountable categoricity.
The corresponding analogue of \(\aleph_0\)-categoricity is much less
clear. In the classical setting, countable categoricity is characterized
by Ryll--Nardzewski and is closely tied to finiteness of type spaces
over \(\emptyset\) and to oligomorphic automorphism groups. None of
these formulations transfers directly to the arithmetic setting.

One possible approach is to ask for uniqueness, up to computable or
arithmetical isomorphism, among arithmetically decidable models. Another
is to seek a reformulation in terms of arithmetical type spaces over
arithmetically decidable bases. At present it is not clear which notion,
if any, captures the right analogue of classical countable
categoricity.
\end{enumerate}

These questions all point in the same direction. The results of this
paper suggest that the Baldwin--Lachlan analysis does not depend as
essentially on cardinality as the usual presentation might suggest. The
extent to which this remains true beyond the arithmetically decidable
setting is still open.


\subsection*{Acknowledgements}

The authors thank Chitat Chong and Leonardo N. Coregliano for helpful
discussions. Artificial intelligence tools were used for writing and editing
assistance in the preparation of this paper; the authors are responsible for
the final content.

\appendix
\section{Foundational Constructions and Coding}
\label{sec:appendix-foundational}

This appendix fixes the coding conventions and standard constructions used
in the paper. The body is deliberately written close to the classical
model-theoretic proof, so it uses semantic language such as Henkin models,
Skolem expansions, omitting-types models, represented quotients, and
satisfaction in coded structures. The purpose here is to make explicit that,
from arithmetical input data, the countable objects needed in the proof have
arithmetical presentations.

\subsection{Coded structures and Tarski truth}
\label{subsec:coded-structures-truth}

We work inside a second-order structure of the form
\[
(\N,\mathcal A,+,\times,0,1,<),
\]
where \(\mathcal A\subseteq\mathcal P(\N)\) is the collection of available
sets.

A countable language \(\LL\) is coded by a subset of \(\N\) specifying
its symbols together with their arities and kinds. All such coding is
taken to be primitive recursive in the usual way.

\begin{definition}[Coded structure]
\label{def:coded-structure}
A \emph{coded \(\LL\)-structure} in
\((\N,\mathcal A,+,\times,0,1,<)\) consists of:
\begin{enumerate}[label=(\roman*)]
\item a set \(M\in\mathcal A\) with infinite complement in \(\N\);
\item for each \(n\)-ary relation symbol \(R\in\LL\), a set
      \(R^\MM\subseteq M^n\), coded as an element of \(\mathcal A\);
\item for each \(n\)-ary function symbol \(f\in\LL\), a function
      \(f^\MM:M^n\to M\), coded by its graph as an element of
      \(\mathcal A\);
\item for each constant symbol \(c\in\LL\), an element \(c^\MM\in M\).
\end{enumerate}
\end{definition}

The requirement that \(M\) have infinite complement is only a coding
convenience. It ensures that there is always room in \(\N\) outside the
domain for uniform coding of tuples, formulas, proofs, and other
auxiliary objects, without having to recode the underlying structure at
each stage. Nothing substantive depends on this choice.

Accordingly, when we say that a model has domain contained in \(\N\), we
always mean such a coded structure. Likewise, a subset \(X\subseteq M^n\)
is \emph{available} if its code belongs to the ambient second-order part
\(\mathcal A\).

\begin{definition}[Arithmetically decidable structure]
\label{def:arith-decidable-structure-appendix}
A coded \(\LL\)-structure \(\MM\) is \emph{arithmetically decidable} if
its elementary diagram is an arithmetical subset of \(\N\). Equivalently,
the satisfaction relation
\(\mathrm{Sat}_{\MM}(\ulcorner\varphi\urcorner,\bar a)\) is
arithmetical uniformly in the code of the formula \(\varphi\) and the
tuple \(\bar a\in M^{<\omega}\).
\end{definition}

\begin{remark}
Having arithmetical atomic diagram is weaker than being
arithmetically decidable. In the main text, the presentation condition
used for ambient models is arithmetical decidability, not merely
arithmetical atomic diagram.
\end{remark}

In particular, the domain, the atomic diagram, term evaluation, and full
first-order satisfaction relation of an arithmetically decidable
structure are all arithmetical.

\begin{definition}[Satisfaction for coded structures]
\label{def:sat-coded-structure}
Let \(\MM\) be a coded \(\LL\)-structure. For a formula
\(\varphi(\bar x)\) and a tuple \(\bar a\in M^{|\bar x|}\), the relation
\(\MM\models \varphi(\bar a)\) is the usual Tarskian truth relation,
computed from the coded interpretations of the symbols, the Boolean
connectives, and quantifiers ranging over the coded domain \(M\). Thus,
for example, \(\MM\models \exists x\,\psi(x,\bar a)\) iff
\((\exists b\in M)\,\MM\models\psi(b,\bar a)\).
\end{definition}

Since the quantifiers range only over the coded domain \(M\subseteq\N\),
this is simply the usual Tarskian notion of truth, carried out inside
second-order arithmetic.

\begin{remark}[Tarskian truth and arithmetical decidability]
\label{rem:tarski-truth}
For an arbitrary coded structure, the satisfaction relation is defined
externally by the usual Tarskian recursion. In this paper, the ambient
structures used for effectivity arguments are assumed to be arithmetically
decidable, so their full satisfaction relations are arithmetical by
definition. Thus it is legitimate to speak of elementary diagrams, complete
theories, and types of such coded structures within the arithmetical
framework.
\end{remark}

\subsection{Compactness, completion, Morleyization, and Skolemization}
\label{subsec:compact-skolem-appendix}

Throughout, \(\LL\) is a countable arithmetically definable language,
\(T\) is an arithmetically definable \(\LL\)-theory with no finite
models, and a model means a coded structure with domain contained in
\(\N\), in the sense of Definition~\ref{def:coded-structure}.

Recall that a set \(X\subseteq\N\) is \emph{arithmetical} if it is
definable in \((\N,+,\times,0,1,<)\), equivalently if
\(X\le_T 0^{(n)}\) for some \(n\in\N\). An \(\LL\)-structure \(\MM\)
with domain \(M\subseteq\N\) is \emph{arithmetically decidable} if its
elementary diagram, equivalently its full satisfaction relation with
parameters from \(M\), is arithmetical. We write \(\mathbf 0\) for the
least arithmetic degree, namely the degree of the arithmetical sets.

If \(A\subseteq M\) is arithmetical, the language \(\LL(A)\) is coded by
adding a constant symbol \(c_a\) for each \(a\in A\), using a fixed
primitive recursive pairing of the old language code with the number
\(a\). Thus the set of new constant symbols is arithmetical exactly when
\(A\) is. If \(\MM\) is arithmetically decidable, then the expansion
\((\MM,a)_{a\in A}\) to \(\LL(A)\) is arithmetically decidable: formulas
in the expanded language are translated effectively to \(\LL\)-formulas
with the corresponding parameters from \(A\).

\begin{theorem}[Compactness for arithmetically definable theories]
\label{thm:compact}
Let \(T\) be an arithmetically definable theory in an arithmetically
definable language \(\LL\). If every finite subset of \(T\) has a model,
then \(T\) has an arithmetically decidable model.
\end{theorem}

\begin{proof}
We use the usual Henkin construction and note that all sets arising in
the construction are arithmetical. In the applications below, semantic
finite satisfiability is verified by finite model-theoretic arguments.

Fix an arithmetical predicate \(\mathrm{Ax}_T(n)\) defining the axioms
of \(T\). Extend \(\LL\) to a language \(\LL^H\) by adjoining
constants \((c_i)_{i\in\N}\), and for each formula \(\exists
v\,\varphi(v)\) add the Henkin axiom
\[
\exists v\,\varphi(v)\to \varphi(c_{\ulcorner\varphi\urcorner}).
\]
Let \(H\) be the set of all such axioms and put \(T^0:=T\cup H\).
Since \(\varphi\mapsto \ulcorner H_\varphi\urcorner\) is primitive
recursive, \(H\) is arithmetical. Every finite subset of \(T^0\) is
satisfiable by the usual Henkin argument.

Now enumerate the \(\LL^H\)-sentences as \((\theta_n)_{n\in\N}\), and
define \(T^{n+1}\) by adjoining \(\theta_n\) if
\(T^n\cup\{\theta_n\}\) is consistent and otherwise adjoining
\(\neg\theta_n\). Since provability from an arithmetically definable
theory is arithmetical, consistency of each finite extension considered
in the construction is an arithmetical predicate. Thus the resulting
complete Henkin theory \(T^\ast:=\bigcup_n T^n\) is arithmetically
definable, possibly at a higher finite level.

Let \(\CTerm\) be the primitive recursive set of closed \(\LL^H\)-terms,
and define \(s\equiv t\) iff \(T^\ast\vdash s=t\). Since provability in
\(T^\ast\) is arithmetical, so is \(\equiv\). Instead of leaving the
domain as a quotient, take the least representative of each
\(\equiv\)-class and let \(R\subseteq\CTerm\) be the resulting set of
representatives. The set \(R\) is arithmetical. Interpreting symbols on
representatives in the usual way gives the term model, and we then
transport it along a fixed arithmetical bijection from \(R\) to a coded
subset of \(\N\) with infinite complement. We continue to denote the
transported model by \(M\).
By the standard truth lemma, for every \(\LL^H\)-sentence \(\sigma\),
\[
M\models \sigma
\quad\Longleftrightarrow\quad
T^\ast\vdash \sigma.
\]
More generally, formulas with parameters from the term model are handled
by replacing the parameters with representing closed terms. Since
provability from \(T^\ast\) is arithmetical, the elementary diagram of
\(M\) is arithmetical. Hence \(M\) is arithmetically decidable. By the
truth lemma, \(M\models T^\ast\), and therefore \(M\models T\).
\end{proof}

\begin{lemma}[Base-preserving coded compactness]
\label{lem:base-preserving-coded-compactness}
Let \(\AAA\) be an arithmetically decidable structure with arithmetical
domain \(A\), and let \(\Sigma\) be an arithmetically definable theory
in a language extending \(\LL(A)\) which contains
\(\Diag_{\mathrm{el}}(\AAA)\). If every finite subset of \(\Sigma\) is
satisfiable, then \(\Sigma\) has an arithmetically decidable coded model
\(\MM\) whose domain contains \(A\) literally and in which \(c_a\) is
interpreted as \(a\), for every \(a\in A\). In particular,
\(\AAA\preccurlyeq \MM\restr_\LL\).
\end{lemma}

\begin{proof}
Apply Theorem~\ref{thm:compact} in the expanded language and take the
canonical term model before the final transport. Let \(R\) be the
arithmetical set of least representatives of closed terms modulo
provable equality. For each \(a\in A\), let \(r_a\in R\) be the
representative of the constant \(c_a\). The elementary diagram of
\(\AAA\) contains the distinctness diagram of \(A\), so \(a\mapsto r_a\)
is injective, and its graph is arithmetical.

Choose an arithmetical infinite subset \(F\subseteq\N\setminus A\) with
infinite complement. Transport the term model by a map \(j:R\to\N\)
sending each \(r_a\) to the actual number \(a\) and all other
representatives injectively into \(F\). The
transported domain is arithmetical and has infinite complement, and all
transported relations and functions remain arithmetical because \(R\),
the term model, and \(j\) are arithmetical. By construction the
constants naming \(A\) are interpreted literally. Since
\(\Diag_{\mathrm{el}}(\AAA)\subseteq\Sigma\), the reduct contains
\(\AAA\) as an elementary substructure.
\end{proof}

A first routine consequence is that every consistent arithmetically
definable theory has an arithmetically definable completion.

\begin{corollary}[Arithmetical completion]
\label{cor:arith-completion}
Every consistent arithmetically definable theory extends to a complete
consistent arithmetically definable theory.
\end{corollary}

\begin{proof}
This is exactly the completion part of the proof of
Theorem~\ref{thm:compact}: the theory \(T^\ast\) constructed there is a
complete consistent arithmetically definable extension of \(T\).
\end{proof}

A second routine consequence is that one may pass to a Morleyization
without leaving the arithmetical setting.

\begin{theorem}[Morleyization]
\label{thm:morleyization}
Let \(T\) be an arithmetically definable theory in a countable language
\(\LL\). Then there exist a countable arithmetically definable expansion
\(\LL^{\mathrm M}\) of \(\LL\) and an arithmetically definable theory
\(T^{\mathrm M}\) such that:
\begin{enumerate}[label=(\arabic*)]
\item every \(\LL\)-model of \(T\) has a unique expansion to an
      \(\LL^{\mathrm M}\)-model of \(T^{\mathrm M}\);
\item reduct induces a bijection between models of \(T^{\mathrm M}\) and
      models of \(T\);
\item \(T^{\mathrm M}\) has quantifier elimination;
\item if \(T\) is complete, then \(T^{\mathrm M}\) is complete.
\end{enumerate}
\end{theorem}

\begin{proof}
For each \(\LL\)-formula \(\varphi(\bar x)\), add a relation symbol
\(R_\varphi(\bar x)\), and add the defining axiom
\[
\forall \bar x\,\bigl(R_\varphi(\bar x)\leftrightarrow \varphi(\bar x)\bigr).
\]
Because formulas are effectively coded, the expanded language remains
countable and arithmetically definable, and the set of defining axioms
is arithmetical. The usual argument gives (1) and (2). Quantifier
elimination holds because every formula is equivalent in
\(T^{\mathrm M}\) to an atomic formula of the form \(R_\varphi\). If
\(T\) is complete, then two models of \(T^{\mathrm M}\) have the same
\(\LL\)-reduct theory, hence the same \(\LL^{\mathrm M}\)-theory, so
\(T^{\mathrm M}\) is complete as well.
\end{proof}

We next pass to Skolem expansions.

\begin{theorem}[Skolemization]
\label{thm:sk}
There exist an arithmetically definable expansion
\(\LL^{\mathrm{Sk}}\) of \(\LL\) and an arithmetically definable theory
\(T^{\mathrm{Sk}}\) such that:
\begin{enumerate}[label=(\arabic*)]
\item \(T\) has a model if and only if \(T^{\mathrm{Sk}}\) has a model;
\item for every \(\LL\)-formula \(\varphi(x,\bar y)\), the theory
      \(T^{\mathrm{Sk}}\) contains the corresponding Skolem axiom;
\item if \(\NN\models T^{\mathrm{Sk}}\) and \(A\subseteq N\), then the
      Skolem hull \(\Hull^{\NN}(A)\) is an elementary substructure of
      \(\NN\);
\item if \(\NN\) is arithmetically decidable and
      \(A\subseteq N\) is arithmetical, then
      \(\Hull^{\NN}(A)\) has arithmetical domain and is
      arithmetically decidable.
\end{enumerate}
Here clause~\textup{(1)} is only the classical existence statement for a
Skolem expansion. It does not assert that an arbitrary presentation of an
\(\LL\)-model carries a canonical arithmetically chosen Skolem expansion.
The arithmetical presentation needed for Skolem-hull degree calculations is
the one explicitly assumed in clause~\textup{(4)}.
\end{theorem}

\begin{proof}
For each \(\LL\)-formula \(\varphi(x,\bar y)\) with G\"odel code \(e\),
add a function symbol \(f_e\) together with the corresponding Skolem
axiom. Since the maps sending \(e\) to the code of \(f_e\) and to the
code of its Skolem axiom are primitive recursive, both
\(\LL^{\mathrm{Sk}}\) and \(T^{\mathrm{Sk}}\) are arithmetically
definable. The equivalence in (1) is the usual expansion-reduct
argument.

For (3), let \(M:=\Hull^{\NN}(A)\). By the Tarski--Vaught test, it
suffices to show that whenever \(\bar b\in M^{<\omega}\) and
\(\NN\models \exists x\,\varphi(x,\bar b)\), there is a witness in
\(M\). But then \(\NN\models \varphi(f_\varphi(\bar b),\bar b)\), and
\(M\) is closed under the Skolem functions.

For (4), an element \(b\) belongs to \(\Hull^{\NN}(A)\) iff
\(b=t^{\NN}(\bar a)\) for some \(\LL^{\mathrm{Sk}}\)-term \(t\) and some
tuple \(\bar a\in A^{<\omega}\). Since term evaluation in an
arithmetically decidable structure is arithmetical, membership in the
hull is arithmetical. Thus the hull has arithmetical domain. Since
\(\Hull^\NN(A)\preccurlyeq \NN\) by (3), satisfaction in the hull agrees
with satisfaction in \(\NN\) for formulas with parameters from the hull.
Because \(\NN\) is arithmetically decidable and the hull domain is
arithmetical, the elementary diagram of \(\Hull^\NN(A)\) is
arithmetical. Hence the hull is arithmetically decidable.
\end{proof}

\subsection{Omitting types}

A partial type \(p(\bar x)\) over \(A\) is \emph{principal over a
theory \(T\)} if there is a formula \(\chi(\bar x)\), with parameters
from \(A\), such that \(T\cup\{\exists\bar x\,\chi(\bar x)\}\) is
consistent and, for every \(\varphi(\bar x)\in p\),
\[
T\vdash \forall\bar x\,(\chi(\bar x)\rightarrow\varphi(\bar x)).
\]
It is \emph{nonprincipal} otherwise. For complete types this agrees with
the usual notion of being isolated by a formula belonging to the type.

A partial type \(p(\bar x)\) is \emph{unsupported over \(T\)} if, for
every formula \(\chi(\bar x)\) consistent with \(T\), there is some
\(\varphi(\bar x)\in p\) such that
\[
T\cup\{\exists\bar x\,(\chi(\bar x)\wedge\neg\varphi(\bar x))\}
\]
is consistent. For complete types, nonprincipality implies this
unsupported condition. In the body of the paper we also omit partial
types which include inequalities against a named model, so this is the
form we shall use.

\begin{theorem}[Arithmetic omitting types for arithmetically coded families]
\label{thm:arith-omit-countable-family}
Let \(T^\dagger\) be a consistent arithmetically definable theory in a
countable arithmetically definable language \(\LL^\dagger\). Let
\(\{p_e(\bar x_e):e\in E\}\) be an arithmetically coded family of
partial types over \(\emptyset\). Thus \(E\subseteq\N\) is arithmetical,
the arity \(e\mapsto |\bar x_e|\) is arithmetically coded on \(E\), and
there is an arithmetical relation
\[
\operatorname{Mem}(e,\ulcorner\varphi\urcorner)
\]
which says that the formula \(\varphi(\bar x_e)\) belongs to \(p_e\).
Suppose that each \(p_e\) is unsupported over \(T^\dagger\). Then
\(T^\dagger\) has an arithmetically decidable model omitting every
\(p_e\).
\end{theorem}

\begin{proof}
We adapt the Henkin construction from
Theorem~\ref{thm:compact}, adding omission requirements for the family
\(\{p_e:e\in E\}\).

Expand \(\LL^\dagger\) to a Henkin language \(\LL^H\) by adjoining
constants \(c_0,c_1,\dots\). Since \(\LL^\dagger\) is countable and
arithmetically definable, so is \(\LL^H\). Fix an arithmetical
enumeration \((\theta_s)_{s\in\N}\) of all \(\LL^H\)-sentences. The fresh
constant used for a Henkin witness is always the least \(c_m\) whose code
has not occurred in the finite condition constructed so far. Also
arithmetically enumerate all pairs \((e,\bar t)\) such that \(e\in E\)
and \(\bar t\) is a closed tuple of \(\LL^H\)-terms of length
\(|\bar x_e|\); write this enumeration as
\((e_0,\bar t_0),(e_1,\bar t_1),(e_2,\bar t_2),\dots\).

We build inductively an increasing sequence
\(\Delta_0\subseteq \Delta_1\subseteq \Delta_2\subseteq\cdots\) of
finite sets of \(\LL^H\)-sentences such that
\(T^\dagger\cup\Delta_s\) is consistent for every \(s\).

At even stages we decide one sentence and add a Henkin witness if
needed, exactly as in the proof of Theorem~\ref{thm:compact}. The
consistency test is the finite proof-search predicate
\[
\text{``there is no proof of contradiction from }
T^\dagger\cup\Delta_{2s}\cup F\text{''},
\]
where \(F\) is the finite set currently being tested. Since
\(T^\dagger\) is arithmetically definable and \(\Delta_{2s}\cup F\) is
finite, this predicate is arithmetical, possibly after passing to a
finite jump.

Suppose \(\Delta_{2s}\) has been constructed. Consider \(\theta_s\). If
\(T^\dagger\cup\Delta_{2s}\cup\{\theta_s\}\) is consistent, let
\(S:=\Delta_{2s}\cup\{\theta_s\}\); otherwise let
\(S:=\Delta_{2s}\cup\{\neg\theta_s\}\).
Then \(T^\dagger\cup S\) is consistent.

If the chosen sentence is not existential, set \(\Delta_{2s+1}:=S\).
If the chosen sentence is \(\exists y\,\psi(y)\), choose the first
fresh constant \(c_m\) not used earlier and let
\(\Delta_{2s+1}:=S\cup\{\psi(c_m)\}\).
As usual, this remains consistent.

At odd stages we meet one omission requirement. Suppose
\(\Delta_{2s+1}\) has been constructed, and consider the pair
\((e_s,\bar t_s)\). Write \(p_{e_s}(\bar x)=p(\bar x)\) and
\(\bar t_s=\bar t\).
Let \(\delta\) be the conjunction of the finitely many formulas in
\(\Delta_{2s+1}\). Only finitely many constants occur in \(\delta\)
and \(\bar t\); for the consistency argument only, replace those constants
by fresh variables \(\bar y\). The finite condition itself remains the
original condition in the Henkin language. After the replacement,
\(\delta\) becomes a formula \(\delta'(\bar y)\), and \(\bar t\) becomes a
tuple of terms \(\bar\tau(\bar y)\). This syntactic term substitution is
primitive recursive in the formula and term codes. Since \(\bar t\) was
chosen with length \(|\bar x_{e_s}|\), the formula
\[
\chi(\bar x):=\exists \bar y\,
\bigl(\delta'(\bar y)\wedge \bar x=\bar\tau(\bar y)\bigr)
\]
has the arity of \(p_{e_s}\).

Since \(T^\dagger\cup\Delta_{2s+1}\) is consistent, \(\chi\) is
consistent with \(T^\dagger\). By unsupportedness of \(p\), choose
\(\varphi(\bar x)\in p(\bar x)\) such that
\[
T^\dagger\cup
\{\exists\bar x\,(\chi(\bar x)\wedge\neg\varphi(\bar x))\}
\]
is consistent. Unwinding the definition of \(\chi\), this is exactly the
consistency of
\[
T^\dagger\cup\Delta_{2s+1}\cup\{\neg\varphi(\bar t)\}.
\]
We choose the least formula code \(\ulcorner\varphi\urcorner\) satisfying
both \(\operatorname{Mem}(e_s,\ulcorner\varphi\urcorner)\) and this finite
consistency condition. Unsupportedness guarantees that the search
terminates, while the membership relation for \(p_{e_s}\) and the finite
proof-search consistency predicate are arithmetical.
Put \(\Delta_{2s+2}:=\Delta_{2s+1}\cup\{\neg\varphi(\bar t)\}\).
Then \(T^\dagger\cup\Delta_{2s+2}\) is consistent.

The entire inductive construction is arithmetical. At even stages, the
consistency checks are arithmetical because provability from the
arithmetically definable theory \(T^\dagger\cup\Delta_{2s}\) is
arithmetical. At odd stages, the preceding least-search rule combines the
uniform arithmetical code for the family \(p_e\) with the same finite
proof-search predicate.

Let \(T^*:=T^\dagger\cup\bigcup_{s\in\N}\Delta_s\).
By construction, \(T^*\) is complete, consistent, and Henkin. Since
every step is arithmetical uniformly in the stage number, the theory
\(T^*\) is arithmetically definable.

Let \(\MM^*\) be the canonical term model of \(T^*\), again represented
by least representatives of closed terms modulo provable equality and
then transported to a coded subset of \(\N\) with infinite complement.
Equality on closed terms modulo provable equality is arithmetical. By
the truth lemma, satisfaction of formulas with parameters represented by
closed terms is equivalent to provability from the complete Henkin theory
\(T^*\). Since provability from \(T^*\) is arithmetical, the elementary
diagram of \(\MM^*\) is arithmetical. Hence \(\MM^*\) is an
arithmetically decidable model of \(T^\dagger\).

Finally, \(\MM^*\) omits each \(p_e\). Indeed, let \(\bar a\) be any
tuple from \(\MM^*\) of the same length as \(\bar x_e\). Then \(\bar a\)
has a least closed-term representative \(\bar t\). The pair
\((e,\bar t)\) appears somewhere in the enumeration, and at the
corresponding odd stage we chose a formula \(\varphi(\bar x)\in p_e\)
such that \(\neg\varphi(\bar t)\in T^*\). If another tuple of closed terms
represents the same element of the quotient term model, equality with
\(\bar t\) is provable in \(T^*\), so the truth value of
\(\varphi\) is unchanged. The truth lemma therefore gives
\(\MM^*\models \neg\varphi(\bar a)\). Finally, satisfaction is invariant
under the transport isomorphism from the quotient term model to the final
coded domain. Thus \(\bar a\) does not realize \(p_e\). Since \(\bar a\)
was arbitrary, \(\MM^*\) omits \(p_e\). This holds for every \(e\in E\).
\end{proof}

\begin{corollary}[Reserved-constant version]
\label{cor:reserved-constant-omitting}
The same theorem remains true with an additional uniformly arithmetical
schedule of dense requirements. Concretely, suppose that for each scheduled
index \(i\) and each finite condition \(\Delta\), there is an arithmetically
searchable finite extension \(F\) such that
\[
T^\dagger\cup\Delta\cup F
\]
is consistent, and suppose the search predicate for such \(F\)'s is uniform
in \(i\) and the code of \(\Delta\). Interleaving these dense requirements
with the sentence-decision, witness, and omission requirements still
produces an arithmetically decidable model omitting the family
\((p_e)_{e\in E}\), provided the unsupportedness checks for the omitted
partial types are made below the current finite condition.

The typical application is a reserved constant \(d\) and an arithmetically
named set \(B\): the dense requirements are \(d\neq b\), for \(b\in B\).
The needed finite consistency condition is that every finite list
\(\{d\neq b_0,\ldots,d\neq b_{r-1}\}\) be consistent with the current
theory and the other finite requirements.
\end{corollary}

\begin{proof}
In the proof of Theorem~\ref{thm:arith-omit-countable-family}, interleave
the additional dense requirements with the sentence-decision, witness,
and omission stages. At a dense stage, choose the least successful finite
extension \(F\) in the arithmetical search. At an omission stage, the finite
conjunction \(\delta\) already includes all dense requirements imposed so
far, so unsupportedness is applied relative to this current finite
condition. Since each finite stage is consistent by hypothesis, the
induction and the truth-lemma argument are unchanged.
\end{proof}

As usual, types over arithmetical parameter sets are handled by first
expanding the language by constants for those parameters.

\subsection{Definable quotients and imaginaries in the arithmetic setting}
\label{subsec:arith-imaginaries}

In Section~\ref{sec:sm-control-arith-degree} we use canonical
parameters for definable families of finite subsets of a strongly
minimal set, together with weak elimination of imaginaries for the
induced strongly minimal structure. The purpose of this subsection is to
justify those uses in the present arithmetical framework.

We do not develop the full theory \(T^{\eq}\) here. That would be
possible, but unnecessary for the arguments in the main text. Instead,
we isolate the restricted quotient coding actually needed in
Section~\ref{sec:sm-control-arith-degree}. This is enough for the
finite-fiber coding argument used there.

Throughout, \(\MM\) is an arithmetically decidable structure,
\(A\subseteq M\) is arithmetical, and all definability is understood
relative to the coding conventions fixed earlier in this appendix.

\begin{lemma}[Arithmetic coding of represented definable quotients]
\label{lem:arith-quotient-coding}
Let \(D\subseteq M^n\) be arithmetical, and let
\(E\subseteq D^2\) be an arithmetical equivalence relation. Then the
represented quotient \(D/E\) admits a uniform arithmetical coding in the home sort,
in the following sense:
\begin{enumerate}[label=(\roman*)]
\item an element of the represented quotient is given by a tuple
      \(\bar a\in D\), with equality of representatives governed by \(E\);
\item equality of represented classes is arithmetical, since
      \([\bar a]_E=[\bar b]_E \iff E(\bar a,\bar b)\);
\item every relation on \(D/E\) induced by an arithmetical relation on
      representatives which is invariant under \(E\) in each coordinate is
      again arithmetical;
\item every partial or total function on \(D/E\) whose graph is induced by
      an arithmetical \(E\)-compatible graph on representatives is again
      arithmetically represented.
\end{enumerate}
\end{lemma}

\begin{proof}
Tuples from \(D\) are already coded by natural numbers, so a quotient
element may be represented by any code for a tuple in \(D\). Equality of
represented classes is exactly the relation \(E(\bar a,\bar b)\), which
is arithmetical by hypothesis.

Now let \(R^\ast\) be a relation on \((D/E)^k\) induced by an
arithmetical relation \(R\subseteq D^k\) which is invariant under \(E\)
in each coordinate. Then
\[
R^\ast([\bar a_1]_E,\dots,[\bar a_k]_E)
\iff
R(\bar a_1,\dots,\bar a_k),
\]
and the right-hand side is arithmetical. The same argument applies to
partial or total functions, provided their representative-level graphs are
compatible with \(E\): if \(E(\bar a,\bar a')\) and the graph sends
\(\bar a\) to \(\bar b\) and \(\bar a'\) to \(\bar b'\), then
\(E(\bar b,\bar b')\). Thus all the quotient data needed later may be coded
arithmetically without choosing canonical representatives in the home sort.
\end{proof}

\begin{lemma}[Represented parameters for definable families]
\label{lem:canonical-parameter-arith}
Let \(\rho(\bar x,\bar y)\) be a formula with parameters from \(A\), and let
\(D\subseteq M^{|\bar y|}\) be arithmetical. Define an equivalence
relation \(E_\rho\) on \(D\) by
\(E_\rho(\bar b,\bar c)\iff
\forall \bar x\in M^{|\bar x|}\,
(\rho(\bar x,\bar b)\leftrightarrow \rho(\bar x,\bar c))\).
Then \(E_\rho\) is arithmetical, and the quotient \(D/E_\rho\) is
arithmetically coded in the sense of
Lemma~\ref{lem:arith-quotient-coding}. In particular, for each
\(\bar b\in D\), the represented class \([\bar b]_{E_\rho}\) is a
well-defined quotient element for the definable set
\(\rho(M^{|\bar x|},\bar b)\).
\end{lemma}

\begin{proof}
Since \(\MM\) is arithmetically decidable and \(A\) is arithmetical,
truth of the formula \(\rho(\bar x,\bar y)\) is arithmetical uniformly in
\(\bar x\) and \(\bar y\). Hence the relation \(E_\rho\) is arithmetical.
The represented class does not depend on the chosen representative exactly
because \(E_\rho\) is equality of the corresponding fibers. The rest follows
immediately from Lemma~\ref{lem:arith-quotient-coding}.
\end{proof}

\begin{lemma}[Arithmetic coding in the induced structure]
\label{lem:induced-sm-arith}
Let \(D=\varphi(M,\bar a)\) be strongly minimal over \(A\), where
\(\bar a\) is arithmetical over \(A\). More generally, suppose \(D\) is
definable in an arithmetically decidable ambient structure over
arithmetical data. Then every subset of \(D^m\) definable in the induced
structure on \(D\) over \(A\) and a fixed finite tuple of parameters from
\(D\) is arithmetical relative to the code of that finite tuple.
\end{lemma}

\begin{proof}
A subset of \(D^m\) definable in the induced structure is, by
definition, the trace on \(D^m\) of a set definable in \(\MM\) over
the displayed finite tuple from \(D\) together with named arithmetical
parameters. The finite tuple is part of the input data. Since \(\MM\) is
arithmetically decidable, every such ambient definable set is arithmetical
relative to that finite tuple. Since \(D\) is definable over arithmetical
data, the set \(D^m\) is arithmetical as well. Hence the trace on \(D^m\)
is arithmetical. The strong minimality assumption is part of the surrounding
model-theoretic context; the arithmetical trace calculation itself uses only
ambient decidability, definability of \(D\), and coded finite parameters.
\end{proof}

\begin{proposition}[Weak elimination of imaginaries for coded quotients on a strongly minimal set]
\label{prop:weak-ei-coded-sm}
Let \(D=\varphi(M,\bar a)\) be strongly minimal over \(A\), where
\(\bar a\) is arithmetical over \(A\). Let \(E\) be an \(A\)-definable
equivalence relation on a definable subset of \(D^n\), and let
\(e\in D^n/E\) be represented arithmetically as in
Lemma~\ref{lem:arith-quotient-coding}. Then there is a finite tuple
\(\bar d\in D^{<\omega}\) such that \(e\) is definable from
\(A\cup\bar d\) and \(\bar d\) is algebraic over \(A\cup\{e\}\).
In the represented quotient language, \(e\in\dcl(A,\bar d)\) means that
there is an \(E\)-invariant formula on representatives which defines the
class \(e\) uniquely from \(A\cup\bar d\). Similarly,
\(\bar d\in\acl(A,e)\) means that some formula with the represented
quotient parameter \(e\) has only finitely many real tuple solutions and
contains \(\bar d\).
\end{proposition}

\begin{proof}
The induced structure on \(D\) is strongly minimal in the usual sense.
Indeed, definable subsets of \(D\) in the induced structure are exactly
the traces on \(D\) of subsets of \(M\) definable in \(\MM\), and these
are finite or cofinite because \(D\) is strongly minimal over \(A\).

By the classical weak elimination of imaginaries theorem for strongly
minimal structures in the exact form needed here, every imaginary is
interalgebraic with a finite real tuple. Applied to the represented
quotient \(D^n/E\) in the induced structure on \(D\), the imaginary
element represented by \(e\) is interalgebraic over \(A\) with some finite
tuple \(\bar d\) from \(D\); see, for example,
Marker~\cite[Chapter~6]{marker2002}. That is,
\(e\in\dcl^{\eq}(A,\bar d)\) and \(\bar d\in\acl^{\eq}(A,e)\) in the
induced structure. The tuple \(\bar d\) depends on the particular quotient
element \(e\); no uniform arithmetical choice of such a tuple is asserted in
this proposition.

It remains to interpret this statement in the present arithmetical
coding. By Lemma~\ref{lem:induced-sm-arith}, the induced structure on
\(D\) is arithmetically coded in the home sort. By
Lemma~\ref{lem:arith-quotient-coding}, the represented quotient \(D^n/E\)
is also arithmetically coded. Accordingly, definability
and algebraicity in the induced structure and its represented quotient
sorts are expressible arithmetically in the home sort. A definability
witness gives an \(E\)-invariant formula defining the represented class
from \(\bar d\), and an algebraicity witness gives a formula with finitely
many real tuple solutions over the represented class. In applications to
finite fibers, the formulas \(\gamma\) and \(\eta\) are extracted from
these two witnesses.
\end{proof}

This is the only imaginary-element machinery used in the paper; no full
imaginary expansion \(T^{\eq}\) is used. In particular,
Section~\ref{sec:sm-control-arith-degree} uses only the finite-fiber
consequence: the represented canonical parameter of a finite definable
fiber is replaced by a finite tuple from the strongly minimal set itself.

    \bibliographystyle{amsalpha}

\begin{thebibliography}{99}
    
    \bibitem[BL71]{baldwinlachlan1971}
    J.~T. Baldwin and A.~H. Lachlan, \emph{On strongly minimal sets}, J. Symbolic Logic {\bf 36} (1971), 79--96.

        \bibitem[BGT12]{breuillardgreentao2012}
E.~Breuillard, B.~Green, and T.~Tao,
\emph{The structure of approximate groups},
Publ. Math. Inst. Hautes \'{E}tudes Sci. \textbf{116} (2012),
115--221.
DOI: 10.1007/s10240-012-0043-9.

\bibitem[Hru12]{hrushovski2012}
E.~Hrushovski,
\emph{Stable group theory and approximate subgroups},
J. Amer. Math. Soc. \textbf{25} (2012), no.~1, 189--243.
DOI: 10.1090/S0894-0347-2011-00708-X.

    
    \bibitem[JTZ23]{jingtranzhang2023}
    Y.~Jing, C.-M. Tran, and R.~Zhang,
    \emph{A nonabelian Brunn--Minkowski inequality},
    Geom. Funct. Anal. \textbf{33} (2023), no.~4, 1048--1100.
    DOI: 10.1007/s00039-023-00647-6.

    \bibitem[Joc72]{jockusch1972} C.~G. Jockusch Jr., \emph{Ramsey's theorem and recursion theory}, J. Symbolic Logic {\bf 37} (1972), 268--280.

    \bibitem[JLR91]{jockuschlewisremmel1991} C.~G. Jockusch Jr., A.~A. Lewis and J.~B. Remmel, \emph{$\Pi^0_1$-classes and Rado's selection principle}, J. Symbolic Logic {\bf 56} 
    (1991), no.~2, 684--693.

    
\bibitem[Kap24]{Kaplan2024DefinablePQTheoremNIP}
I.~Kaplan,
\emph{A definable $(p,q)$-theorem for NIP theories},
Advances in Mathematics \textbf{436} (2024), 109418.


    \bibitem[Mac24]{machado2024}
    S.~Machado,
    \emph{Minimal doubling for small subsets in compact Lie groups},
    arXiv:2401.14062, 2024.

    \bibitem[Mat04]{Matousek2004BoundedVCFractionalHelly}
J.~Matou{\v{s}}ek,
\emph{Bounded VC-dimension implies a fractional Helly theorem},
Discrete \& Computational Geometry \textbf{31} (2004), no.~2, 251--255.

    \bibitem[Mar02]{marker2002}
    D.~Marker,
    \emph{Model Theory: An Introduction},
    Graduate Texts in Mathematics, vol.~217,
    Springer, New York, 2002.
    
    \bibitem[Mor65]{morley1965}
    M.~Morley,
    \emph{Categoricity in power},
    Trans. Amer. Math. Soc. \textbf{114} (1965), 514--538.








    \end{thebibliography}

\end{document}